\documentclass[11pt,letterpaper]{amsart}

\usepackage{amsfonts}
\usepackage{amssymb}
\usepackage{amsmath}
\usepackage{color}
\usepackage{fancyhdr}
\usepackage{extramarks}
\usepackage{amsthm}
\usepackage{hyperref}
\usepackage{verbatim}
\usepackage{float}
\usepackage[all]{xy}
\usepackage{graphicx}
\usepackage{caption}
\usepackage{subcaption}
\usepackage{mathtools}
\usepackage{tikz-cd}
\usepackage{enumitem}
\usepackage{titlesec}
\usepackage{MnSymbol}
\usepackage{relsize}
\usepackage{mathtools}
\usepackage{mathrsfs}  
\usepackage{lipsum}
\usepackage[utf8]{inputenc}

\titleformat{\section}
  {\normalfont\fontsize{12}{15}\bfseries}{\thesection}{1em}{}

\titleformat{\subsection}
  {\normalfont\fontsize{12}{15}\bfseries}{\thesubsection}{1em}{}

\newtheorem{theorem}{Theorem}[section]

\newtheorem{lemma}[theorem]{Lemma}
\newtheorem{proposition}[theorem]{Proposition}

\newtheorem{corollary}[theorem]{Corollary}

\theoremstyle{definition}

\newtheorem{definition}[theorem]{Definition}

\newtheorem{remark}[theorem]{Remark}

\graphicspath{ {images/} }

\newcommand{\mb}{\mathbb}

\newcommand{\Q}{\mb{Q}}
\newcommand{\R}{\mb{R}}
\newcommand{\Z}{\mb{Z}}
\newcommand{\C}{\mb{C}}

\newcommand*{\sheafhom}{\mathcal{H}\kern -.5pt om}

\numberwithin{equation}{section}
\renewcommand{\subsectionmark}[1]{}

\title[Hodge Structures of Global Smoothings]{On the Hodge Structures of Global Smoothings of Normal Crossing Varieties}
\author{Kuan-Wen Chen}
\address{Department of Mathematics, Stony Brook University, New York, NY 10027}
\email{kuan-wen.chen@stonybrook.edu}

\begin{document}
\begin{abstract} We study the Hodge theory of smooth fibers of semistable degenerations of compact complex manifolds with central fibers of K\"ahler type. If every stratum of the central fiber is a $\partial\overline{\partial}$-manifold, we prove that the local invariant cycle property in all degrees implies that the nearby fibers satisfy the $\partial\overline{\partial}$-lemma. In middle degree, the corresponding property alone yields a canonical Hodge decomposition on the nearby fibers and an explicit formula, in terms of the limiting mixed Hodge structure, for the signature of the cup-product pairing on each Hodge component. This gives a partial converse to Clemens' local invariant cycle theorem for K\"ahler degenerations \cite{Clemens2}.

As applications, we study global smoothings of ordinary double points and extend results of Friedman \cite{RF1} and Li \cite{CL} on Clemens manifolds. We also apply our results to the non-K\"ahler Calabi--Yau manifolds constructed by Hashimoto--Sano \cite{HaSa} and Sano \cite{Sa}. In particular, these manifolds satisfy the $\partial\overline{\partial}$-lemma, even though they are not in Fujiki class $\mathcal{C}$, and in Sano's higher-dimensional examples we explicitly determine the failure of the Hodge--Riemann bilinear relations on the middle cohomology. \end{abstract}

\maketitle

\section{Introduction}
\subsection{Statements of Main Theorems}

The $\partial\overline{\partial}$-lemma is one of the principal mechanisms through which the fundamental cohomological features of compact K\"ahler manifolds extend to the non-K\"ahler setting. If a compact complex manifold satisfies the $\partial\overline{\partial}$-lemma, in which case we call it a $\partial\overline{\partial}$-manifold, then its de Rham cohomology admits a canonical Hodge decomposition compatible with complex conjugation, and its Bott--Chern, Aeppli, Dolbeault, and de Rham cohomologies are related in essentially the same way as in the K\"ahler case \cite{AngellaTomassini13}. Moreover, the de Rham complex is formal, and hence all higher Massey products vanish \cite{dgms}. Thus, although the $\partial\overline{\partial}$-lemma is strictly weaker than the existence of a K\"ahler metric, it retains a substantial part of the cohomological package of K\"ahler geometry.

A natural source of non-K\"ahler $\partial\overline{\partial}$-manifolds arises from degenerations and global smoothings. The motivating examples are Clemens manifolds, which are obtained by contracting rational curves on a projective Calabi--Yau threefold to ordinary double points and then taking a global smoothing. By results of Friedman \cite{RF1} and Li \cite{CL}, Clemens manifolds satisfy the $\partial\overline{\partial}$-lemma; moreover, the cup product on $H^{2,1}$ satisfies the second Hodge--Riemann bilinear relation. This suggests the following general problem: if a singular compact complex analytic space admits a K\"ahler-type resolution as well as a global smoothing, under what conditions are its smoothings $\partial\overline{\partial}$-manifolds, and to what extent can one determine the Hodge index of their middle cohomology?

To formulate our main result, consider a one-parameter degeneration
$
f:X\rightarrow\Delta
$
of $m$-dimensional compact complex manifolds, where $X$ is a connected complex manifold, $\Delta$ is the unit disk, and $f$ is a proper holomorphic map which is smooth over the punctured disk $\Delta^*$. After semistable reduction, we may assume that the central fiber
$
X_0:=f^{-1}(0)
$
is a reduced simple normal crossing divisor. We refer to such a family as a semistable degeneration of $m$-dimensional compact complex manifolds.

A central condition in this paper is the following local invariant cycle property.

\begin{definition} \ 
\begin{itemize}
    \item[$(\mathbf{C}_d)$]
    \phantomsection\label{localinvariantcycle}
    We say that $f:X\rightarrow\Delta$ satisfies the local invariant cycle property in degree $d$, denoted by $(\mathbf{C}_d)$, if, after possibly shrinking $\Delta$, the sequence
    $$
    H^d(X,\Q)
    \xrightarrow{R}
    H^d(X_t,\Q)
    \xrightarrow{T-\operatorname{Id}}
    H^d(X_t,\Q)
    $$
    is exact at the middle term for every $t\neq0$, where $X_t:=f^{-1}(t)$, $R$ is the restriction map, and $T$ is the monodromy operator induced by a counter-clockwise loop around the origin.
\end{itemize}
\end{definition}

A classical theorem of Clemens gives this property for K\"ahler degenerations.

\begin{theorem}[\cite{Clemens2}]
Let $f:X\rightarrow\Delta$ be a semistable degeneration of $m$-dimensional compact complex manifolds. If $X$ is K\"ahler, then \hyperref[localinvariantcycle]{$(\mathbf{C}_d)$} holds for every $d$.
\end{theorem}

The main theorem of this paper is a partial converse to Clemens' theorem.

\begin{theorem}\label{Thmg}
Let $f:X\rightarrow\Delta$ be a semistable degeneration of $m$-dimensional compact complex manifolds. Suppose that all irreducible components of the central fiber $X_0$ are K\"ahler or, more generally, that every nonempty $k$-fold intersection of the irreducible components of $X_0$ is a $\partial\overline{\partial}$-manifold. Then the following hold.
\begin{enumerate}
    \item If \hyperref[localinvariantcycle]{$(\mathbf{C}_d)$} holds for every $d$, then $X_t$ is a $\partial\overline{\partial}$-manifold for $|t|$ sufficiently small.

    \item If \hyperref[localinvariantcycle]{$(\mathbf{C}_m)$} holds, then $H^m(X_t)$ admits a canonical Hodge decomposition for $|t|$ sufficiently small, and the signature of the cup-product pairing on each Hodge $(p,q)$-component of $H^m(X_t)$, where $p+q=m$, can be computed explicitly in terms of the cohomology of the strata of $X_0$. See Theorem \ref{Corg2} for the explicit formula.
\end{enumerate}
\end{theorem}

The first part should be compared with a theorem of Steenbrink: without assuming \hyperref[localinvariantcycle]{$(\mathbf{C}_d)$}, the Hodge-to-de Rham spectral sequence of $X_t$ still degenerates at $E_1$ \cite{St}. This conclusion is, however, strictly weaker than the $\partial\overline{\partial}$-lemma. There are many compact complex manifolds whose Hodge-to-de Rham spectral sequences degenerate at $E_1$ but which are not $\partial\overline{\partial}$-manifolds; see \cite{KS,COUV} and the examples therein. Thus Theorem \ref{Thmg} upgrades Steenbrink's $E_1$-degeneration to a full $\partial\overline{\partial}$-statement under the local invariant cycle condition.

This strengthening is particularly relevant for non-K\"ahler Calabi--Yau geometry. Balanced Hermitian metrics often serve as natural substitutes for K\"ahler metrics in this setting. Although balancedness is not preserved under arbitrary small deformations, Wu proved that every sufficiently small deformation of a balanced $\partial\overline{\partial}$-manifold remains a balanced $\partial\overline{\partial}$-manifold \cite{Wu06}; see also \cite{FuYauDeformation} for a related result under a weaker $\partial\overline{\partial}$-type assumption. Balanced Calabi--Yau $\partial\overline{\partial}$-manifolds therefore form a deformation-stable class combining non-K\"ahler Hermitian geometry with robust Hodge-theoretic properties. Their deformation theory also retains several features familiar from polarized Calabi--Yau geometry, including local Torelli-type results and natural Weil--Petersson metrics \cite{Popo}.

The $\partial\overline{\partial}$-lemma also plays an important role in the study of the Hull--Strominger system on non-K\"ahler Calabi--Yau threefolds \cite{Strominger86}. In particular, it simplifies the Bott--Chern and Aeppli cohomological structures underlying the anomaly cancellation equation and related existence obstructions; see, for example, \cite{GarciaFernandezGonzalezMolina}. Thus the conclusion of Theorem \ref{Thmg} provides substantially more geometric information than the degeneration of the Hodge-to-de Rham spectral sequence alone.

The local invariant cycle hypothesis in Theorem \ref{Thmg} cannot in general be omitted. Clemens observed that Kodaira's degeneration of Hopf surfaces, whose smooth fibers are not $\partial\overline{\partial}$-manifolds, satisfies \hyperref[localinvariantcycle]{$(\mathbf{C}_d)$} for $d=0,1,2,4$, but fails to satisfy \hyperref[localinvariantcycle]{$(\mathbf{C}_3)$}; see \cite[p.~229]{Clemens2} or Section \ref{kodairathm}. This shows that the hypothesis in the first part of Theorem \ref{Thmg} is essential. We also construct a threefold example in Proposition \ref{kodaira2} for which \hyperref[localinvariantcycle]{$(\mathbf{C}_3)$} fails and $H^3(X_t)$ admits no canonical Hodge decomposition. Thus the hypothesis in the second part cannot be omitted either.

By Steenbrink's theory of limiting mixed Hodge structures, the condition \hyperref[localinvariantcycle]{$(\mathbf{C}_d)$} can be expressed entirely in terms of the cohomology of the strata of $X_0$. More precisely, it is equivalent to a condition on the limiting mixed Hodge structure which we call the \emph{monodromy weight condition}. We therefore work primarily with this equivalent formulation. The monodromy weight condition provides a bridge between Schmid's abstract theory of degenerations of polarized variations of Hodge structure and Steenbrink's geometric description of the limiting mixed Hodge structure.

Our strategy for proving Theorem \ref{Thmg} is to extend several aspects of Schmid's theory to the unpolarized setting and, more generally, to variations equipped with nondegenerate bilinear forms of indefinite signature. Once the resulting abstract degeneration theorem is established, Steenbrink's description of the limiting mixed Hodge structure allows us to translate the monodromy weight condition into explicit cohomological conditions on the strata of the central fiber. This also makes the signature formula in the second part of Theorem \ref{Thmg} computable in concrete examples.

We apply these results to several classes of degenerations. Under a strong K\"ahler-type condition on the central fiber, we recover the usual Hodge index formula of compact K\"ahler manifolds. For global smoothings of ordinary double points, we generalize the results of Friedman \cite{RF1} and Li \cite{CL} on Clemens manifolds to arbitrary dimension. Finally, we apply Theorem \ref{Thmg} to the non-K\"ahler Calabi--Yau manifolds constructed by Hashimoto--Sano \cite{HaSa} and Sano \cite{Sa}, obtaining new families of $\partial\overline{\partial}$-manifolds and explicit formulas for the signatures of their middle cohomology.

\subsection{Applications and Examples}

We conclude the introduction by describing several applications of Theorem \ref{Thmg}. These examples illustrate both the strength of the $\partial\overline{\partial}$-criterion and the fact that, even when the $\partial\overline{\partial}$-lemma holds, the cup-product pairing on the middle cohomology need not satisfy the Hodge--Riemann bilinear relations.

First, suppose that the central fiber satisfies a stronger K\"ahler-type condition: there exists a class
$
L\in H^2(X_0,\R)
$
whose restriction to every irreducible component of $X_0$ is a K\"ahler class. By a theorem of Guill\'en--Navarro Aznar \cite{guillen}, see also \cite[\S 11.3]{PS}, the Steenbrink weight spectral sequence then carries a polarized Hodge--Lefschetz structure, and the Steenbrink weight filtration agrees with the monodromy weight filtration in every degree. Theorem \ref{Thmg} therefore implies that the nearby fibers are $\partial\overline{\partial}$-manifolds. Moreover, the signature formula of Theorem \ref{Thmg} reduces in this case to the usual Hodge index formula for compact K\"ahler manifolds; see Corollary \ref{kah}.

A second application concerns global smoothings of ordinary double points. Let $X_0$ be an $m$-dimensional compact complex analytic space with at worst ordinary double points, and suppose that $X_0$ admits a resolution of singularities which is a $\partial\overline{\partial}$-manifold. Then every sufficiently small global smoothing of $X_0$ is again a $\partial\overline{\partial}$-manifold. The signature of the cup-product pairing on its middle cohomology can be computed explicitly from the middle cohomology of the resolution and the relations among the exceptional quadrics; see Theorem \ref{thmodp}. In particular, we obtain the following consequence, which generalizes the results of Friedman \cite{RF1} and Li \cite{CL} for Clemens manifolds.

\begin{corollary}
Let $f:X\rightarrow\Delta$ be a one-parameter degeneration of 
$m$-dimensional compact complex manifolds whose central fiber $X_0$ 
has at worst ordinary double points. Suppose that $X_0$ admits a 
resolution of singularities that is a 
$\partial\overline{\partial}$-manifold. Then $X_q$ 
is a $\partial\overline{\partial}$-manifold for $|q|$ sufficiently small.

Moreover, let
$$
\widetilde{X}_0
:=
\operatorname{Bl}_{\operatorname{Sing}(X_0)}X_0
$$
be the blowup of $X_0$ at all its ordinary double points. If $H^m(\widetilde{X}_0,\C)$, together with the cup-product 
pairing, is a polarized Hodge structure, then 
$H^m(X_q,\C)$, together with the cup-product pairing, is also a polarized Hodge structure for $|q|$ 
sufficiently small.
\end{corollary}
Our results also apply to several recently constructed families of non-K\"ahler Calabi--Yau manifolds obtained as smoothings of $d$-semistable SNC Calabi--Yau varieties. Hashimoto--Sano \cite{HaSa} constructed non-K\"ahler Calabi--Yau threefolds $X_3(a)$, while Sano \cite{Sa} constructed non-K\"ahler Calabi--Yau $m$-folds $X_m(a)$ for every $m\geq4$. Here $a$ is a positive integer parameter in the constructions, and in both families $b_2(X_m(a))$ tends to infinity as $a\rightarrow\infty$. Moreover, both families are not contained in Fujiki class $\mathcal{C}$; see \cite{HaSa} and \cite[Proposition 3.5]{Sa}. Nevertheless, Theorem \ref{Thmg} applies to these degenerations and allows us to determine both the $\partial\overline{\partial}$-property of the smooth fibers and the signature of the cup-product pairing on their middle cohomology.

\begin{theorem}[Theorem \ref{HaSa}]
The non-K\"ahler Calabi--Yau threefolds $X_3(a)$ constructed by Hashimoto--Sano \cite{HaSa} are $\partial\overline{\partial}$-manifolds. Moreover,
$$
H^3(X_3(a),\C),
$$
together with the cup-product pairing, is a polarized Hodge structure.
\end{theorem}

For Sano's higher-dimensional examples, the $\partial\overline{\partial}$-lemma still holds, but the Hodge--Riemann bilinear relations generally fail. The failure can nevertheless be computed explicitly. 

Let $Q(\bullet,\bullet)$ denote the cup-product pairing on $H^m(X_m(a),\C)$, and set
$$
S(\bullet,\bullet):=\epsilon(m)Q(\bullet,\bullet),
\qquad
\epsilon(m):=(-1)^{\frac{m(m-1)}{2}}.
$$
Let $C$ denote the Weil operator on $H^m(X_m(a),\C)$, defined by multiplication by
$
(\sqrt{-1})^{p-q}
$
on $H^{p,q}(X_m(a))$.

Recall that in Sano's construction
$
T\subset\mb{P}^1\times\mb{P}^{m-2}
$
is a general hypersurface of bidegree $(1,m-1)$. Set
$$
\nu_k:=h^{k-1,m-k-1}(T),
$$
with the convention that $h^{p,q}(T)=0$ if $p<0$ or $q<0$, and set
$$
A_a:=9(27a^2-2a+5)+a+6.
$$

\begin{theorem}[Theorem \ref{Sa}]
The non-K\"ahler Calabi--Yau $m$-folds $X_m(a)$ constructed by Sano \cite{Sa}, for $m\geq4$, are $\partial\overline{\partial}$-manifolds. For each $k$, let $n_k(a)$ be defined as follows:
$$
n_k(a)=
\begin{cases}
\nu_k+A_a,
& m\equiv1\pmod4
\textup{ and }
k=\dfrac{m-1}{2}\textup{ or }\dfrac{m+1}{2},\\[2mm]
\nu_k,
& m\textup{ is odd and the preceding case does not occur},\\[2mm]
\nu_k,
& m\textup{ is even and }k\neq\dfrac{m}{2},\\[2mm]
\nu_{\frac m2}+a,
& m\geq6\textup{ is even and }k=\dfrac{m}{2},\\[2mm]
\nu_2+a-1,
& m=4\textup{ and }k=2.
\end{cases}
$$
Then the signature of
$$
S(C\bullet,\overline{\bullet})
$$
on $H^{k,m-k}(X_m(a))$ is
$$
\left(
h^{k,m-k}(X_m(a))-n_k(a),\,
n_k(a)
\right).
$$
\end{theorem}

Thus Sano's examples provide $\partial\overline{\partial}$-Calabi--Yau manifolds, even outside Fujiki class $\mathcal{C}$, whose middle cohomology does not in general satisfy the Hodge--Riemann bilinear relations. In particular, the $\partial\overline{\partial}$-lemma alone does not force the canonical Hodge structure on middle cohomology to be polarized by the cup product.

A brief description of the contents of the paper is as follows. In Section $2$, we review the $\partial\overline{\partial}$-lemma, Steenbrink's theory of limiting mixed Hodge structures, local invariant cycle theorems and monodromy weight conditions, and Fujisawa's theory of the cup product on limiting mixed Hodge structures. In Section $3$, we recall some basic results on degenerations of abstract variations of Hodge structure and prove an abstract version of Theorem \ref{Thmg}, assuming two technical estimates for nilpotent orbits. In Section $4$, we deduce Theorem \ref{Thmg} from the abstract results established in Section $3$. In Section $5$, we apply Theorem \ref{Thmg} to several examples, including global smoothings of ordinary double points, global smoothings of Calabi--Yau threefolds with an $O_{16}$ singularity, the non-K\"ahler Calabi--Yau threefolds of Hashimoto--Sano \cite{HaSa}, and the non-K\"ahler Calabi--Yau $m$-folds of Sano \cite{Sa}. In Section $6$, we prove the unpolarized and generalized converse to the nilpotent orbit theorem used in Section $3$. Finally, in the Appendix, we carry out the detailed computations of the limiting mixed Hodge structures for Sano's non-K\"ahler Calabi--Yau $m$-folds.
\begin{remark}
    Some special cases of Theorem \ref{Thmg} for the degenerations with $N^2=0$ have also been obtained by Tsung-Ju Lee \cite{TJ} and Taro Sano \cite{SA3} independently. 

\end{remark}

\noindent\textbf{Acknowledgments.}
The author would like to express his sincere gratitude to Robert Friedman for numerous suggestions, explanations, and clarifications, and for the considerable time he devoted to reading and editing this paper. The author is also grateful to Radu Laza, Aise Johan de Jong, Chiu-Chu Melissa Liu, Giulia Sacc\`a, and Taro Fujisawa for helpful comments and discussions. Thanks are also due to Hung Chiang, Po-Sheng Wu, Yi-Tsung Wang, Ping-Hsun Chung, and Alex Xu for many inspiring discussions.

The author benefited from the 2024 TIMS Summer Program in Geometry, particularly from the lectures by Tristan Collins and S\'ebastien Picard, which deepened his understanding of the role of the $\partial\overline{\partial}$-lemma in the study of non-K\"ahler Calabi--Yau geometry.

The author learned from Chin-Lung Wang that Tsung-Ju Lee had independently obtained related results. He thanks Chin-Lung Wang for bringing this work to his attention and Tsung-Ju Lee for subsequent discussions and helpful suggestions on an earlier version of this paper. The author also thanks Taro Sano and Chi Li for their comments and discussions on the first version of this paper posted on arXiv.

\noindent\textbf{AI Disclosure.} The author used ChatGPT to help refine the wording of this article. ChatGPT also identified a sign error in the computation underlying Theorem 1.6.

\noindent\textbf{Notation:}
For every $a\in \mb{Z}$, let $\epsilon(a):=(-1)^{\frac{a(a-1)}{2}}.$ We have $$\epsilon(a+1)=(-1)^a \epsilon(a)=\epsilon(-a), \ \ \ \epsilon(a+2)=-\epsilon(a),$$
$$\epsilon(a+b)=(-1)^{ab}\epsilon(a)\epsilon(b).$$ 

\section{Preliminaries and Background}
\subsection{Hodge Decomposition and the $\partial\overline{\partial}$-Lemma} \label{ddbar}

The classical $\partial\overline{\partial}$-lemma for compact K\"ahler manifolds states that if we are given a smooth differential form $\alpha$ on a compact K\"ahler manifold which is $\partial$-closed, $\overline{\partial}$-closed and $d$-exact, then $\alpha$ is $\partial\overline{\partial}$-exact. Compact complex manifolds satisfying this condition are called $\partial\overline{\partial}$-manifolds. 

 Let $M$ be a compact complex manifold, not necessarily K\"ahler. Consider the Hodge filtration on the cohomology: $$F^pH^d(M,\C):=\textup{Im}(\mathbb{H}^d(M,\sigma^{\geq p}\Omega_M^{\bullet})\rightarrow \mathbb{H}^d(M,\Omega_M^{\bullet})),$$ where $\sigma^{\geq p}\Omega_M^{\bullet}$ is the trivial filtration on the de Rhm complex. The following theorem is well-known:

 \begin{theorem}[cf.  \cite{deligne2}, \cite{dgms} ] \label{Del}
     The following conditions are equivalent : 
\begin{enumerate}
    \item $M$ is a $\partial\overline{\partial}$-manifold.
    
    \item For all $d$, there is a Hodge decomposition $H^d(M)=\bigoplus_{p} H^{p,d-p}(M),$ such that $H^{p,d-p}(M)\cong H^{d-p}(M,\Omega_M^p),$ where $H^{p,d-p}(M):=F^{p}H^d(M)\cap \overline{F^{d-p}}H^d(M)$.

\item The following two conditions hold: 
\begin{enumerate}
    \item the Hodge to de Rham spectral sequence for $M$ degenerates at the $E_1$ page,
    \item the Hodge filtration $F^{\bullet}H^d(X)$ is $d$-opposed to its complex conjugate $\overline{F}^{\bullet}H^d(M)$ for all $d$, i.e. $$F^{p}H^{d}(M,\C)\oplus \overline{F^{d-p+1}}H^{d}(M,\C)=H^{d}(M,\C)$$ for all $p\geq 0$.
\end{enumerate}

\end{enumerate} 
\end{theorem}
 
 There are two important consequences of Theorem \ref{Del}. 
 
First, since both the degeneration of the Hodge–de Rham spectral sequence and the $d$-opposedness of the Hodge filtration with its complex conjugate are open conditions under deformation, it follows that the $\partial\overline{\partial}$-lemma itself is open under deformation. 

Second, if a compact complex manifold $M$ is dominated by either a compact Kähler manifold or by a $\partial\overline{\partial}$-manifold of the same dimension, then $M$ is also a $\partial\overline{\partial}$-manifold (cf. \cite[Theorem 5.22]{dgms}). In particular, any Moishezon manifold and any compact complex manifold of Fujiki class $\mathcal{C}$ (i.e. bimeromorphic to a compact Kähler manifold) are $\partial\overline{\partial}$-manifolds. Conversely, if $M$ is a $\partial\overline{\partial}$-manifold and $Z \subset M$ is a $\partial\overline{\partial}$-submanifold, then the blowup $Bl_{Z}M$ is again a $\partial\overline{\partial}$-manifold (\cite{Stel}, \cite{RYY}, \cite{YY}, \cite{DTNA}). It follows that the $\partial\overline{\partial}$-lemma is a bimeromorphic invariant in dimensions $\leq 3$, but Liu--Rao recently showed that $\partial\overline{\partial}$-lemma is not a bimeromorphic invariant in every complex dimension at least four \cite{LR26}. Moreover, if $M$ is a compact complex analytic space with at worst isolated singularities, then one resolution of singularities of $M$ is a $\partial\overline{\partial}$-manifold if and only if every resolution is \cite[Remark 5.2]{FrLa}.

There are also several examples of $\partial\overline{\partial}$-manifolds not belonging to Fujiki class $\mathcal{C}$. First, since the $\partial\overline{\partial}$-lemma is open under deformation while being of class $\mathcal{C}$ is not (\cite{Campana}, \cite{CYS}), small deformations of certain Moishezon twistor spaces yield such examples. Second, examples arise among solvmanifolds (\cite{AK1}, \cite{AK2}, \cite{KAS}). Third, \cite{DTNA} construct a simply connected threefold example. Fourth, Friedman \cite{RF1} and Li \cite{CL} proved that Clemens manifolds are $\partial\overline{\partial}$-manifolds. Finally, \cite{KS} produce further simply connected threefold examples, as well as compact complex manifolds that fail the $\partial\overline{\partial}$-lemma.

We will also need the following lemma:

\begin{lemma} \label{rmkg} (cf. \cite[Lemma 1.4]{RF1})
   Let $X$ be a compact complex manifold of dimension $m$ for which the Hodge-de Rham spectral sequence degenerates at $E_1$. Then, the middle cohomology $H^m(X,\C)$ satisfies the first Hodge-Riemann condition with respect to the cup product: $$ Q\left(F^p,F^{m-p+1}\right)=0 $$ for every $p\geq 0$, where $Q(\bullet,\bullet)$ is the cup product pairing on $X$ and $F^{\bullet}$ is the Hodge filtration on $H^m(X,\C)$.
   \end{lemma}

\subsection{Geometric Theory of Limiting Mixed Hodge Structures} \label{steenbrinkintro}
 
Suppose we are given a one-parameter degeneration of $m$-dimensional compact complex manifolds $f:X \rightarrow \Delta$, where $X$ is a connected complex manifold, $\Delta$ is the unit disk and $f$ is a proper holomorphic map that is smooth over the punctured disk $\Delta^*$ and $E:=f^{-1}(0)=X_0$ is a simple normal crossing divisor all of whose components are reduced. 
 The total space is homotopy equivalent to $E$ by a fiber-preserving retraction $r:X \rightarrow E$. So the composition of the inclusion map $i_t: X_t\hookrightarrow X$ and the retraction gives a specialization map $r_t: X_t\rightarrow E$. The complex $\psi_f\underline{\Z}_X:=(Rr_t)_{*}i_t^*\underline{\Z}_X$ is the complex of the nearby cocycle so that $\mathbb{H}^d(E,\psi_f\underline{\Z}_X)=H^d(X_{\infty})$ for all $d\geq 0$, where $X_{\infty}:=X\times_{\Delta^{*}}\mathfrak{H}$ is the canonical fiber.

In the discussion that follows, we will need the following two distinct K\"ahler-type conditions on $E$:
\begin{definition} \ 
\begin{itemize}
     \item[$(\mathbf{K}_1)$]
    \phantomsection\label{WeakKahler} We say $f:X\rightarrow\Delta$ satisfies the weak K\"ahler-type condition $(\mathbf{K}_1)$ if every irreducible component of $E$ is K"ahler or, more generally, if every nonempty $k$-fold intersection of irreducible components of $E$ is a $\partial\overline{\partial}$-manifold for all $k\geq1$.
      \item[$(\mathbf{K}_2)$]
    \phantomsection\label{StrongKahler} We say $f:X\rightarrow\Delta$ satisfies the strong K\"ahler-type condition $(\mathbf{K}_2)$ if there exists a class $L\in H^2(E,\R)$ which restricts to a K\"ahler class on each irreducible component of the normal crossing divisor $E$.
    \end{itemize}
\end{definition}

We first recall and summarize Steenbrink's construction of limiting mixed Hodge structures. Consider the relative log de Rham complex $$\Omega^{\bullet}_{X/\Delta}(\text{log}E):=\Omega_X^{\bullet}(\text{log}E)/f^{*}\Omega^1_{\Delta}(\text{log}0)\wedge \Omega^{\bullet-1}_X(\text{log}E).$$

\begin{theorem}(\cite{St}, or see \cite[ \S 11.2.4-\S 11.2.7]{PS})\label{St} Let $f:X\rightarrow\Delta$ be a semistable degeneration of $m$-dimensional compact complex manifolds. Then 
\begin{enumerate}
\item
The hypercohomology $\mathbb{H}^d(E,\Omega^{\bullet}_{X/\Delta}(\textup{log}E)\otimes \mathcal{O}_E)$ is isomorphic to the cohomology $H^d(X_{\infty};\C)$, where $X_{\infty}:=X\times_{\Delta^{*}}\mathfrak{H}$ is the canonical fiber. The sheaf $\Tilde{\mathcal{H}}^d=\mathbb{R}^d f_{*}\Omega^{\bullet}_{X/\Delta}(\textup{log}E)$ is locally free and satisfies: $\Tilde{\mathcal{H}}^d|_{\Delta^*}$ is isomorphic to the holomorphic vector bundle $\mathcal{H}^d:=(R^d(f|_{f^{-1}(\Delta^*)})_{*}\C)\otimes_{\C}\mathcal{O}_{\Delta^*}$, and 
 $\Tilde{\mathcal{H}}^d$ is Deligne’s canonical extension of $\mathcal{H}^d$. 
 \item The monodromy operator $T$ on $H^d(X_t,\C)$ induced by counter-clockwise loop is unipotent.
\end{enumerate}
Suppose moreover that $f:X\rightarrow\Delta$ satisfies \hyperref[WeakKahler]{$(\mathbf{K}_1)$}. Then
 \begin{enumerate}
  \setcounter{enumi}{2}
\item There exists a marked mixed Hodge complex of sheaves on $E$ $$ 
\psi_f^{\textup{Hdg}}=(\psi_f\underline{\Z}_X,(s(C^{\bullet,
\bullet}),W^{St}),(s(A^{\bullet,\bullet}),W^{St},F_{\lim})),$$ which defines the limiting mixed Hodge structure over $\Q$ on $H^d(X_{\infty})$ so that the nilpotent endomorphism $N=\textup{log}(T)$ acts as a morphism of mixed Hodge structure of type $(-1,-1)$.
In particular, the weight spectral sequence $${}_{W^{St}}E_1^{-r,d+r}=\bigoplus_{k\geq 0,-r} H^{d-r-2k}(E(2k+r+1))(-r-k)\Rightarrow H^d(X_{\infty},\Q)$$ degenerates at $E_2$

\item The spectral sequence with $E_1$ page $$E^{p,q}_1=H^q(E; \Omega^{p}_{X/\Delta}(\textup{log}E)\otimes \mathcal{O}_E) \Rightarrow H^{p+q}(X_{\infty};\C)$$ degenerates  at $E_1$ and the corresponding filtration on $H^{p+q}(X_{\infty};\C)$ is the Hodge filtration. Moreover, possibly after shrinking $\Delta$, the spectral sequence  with $E_1$ page $$E^{p,q}_1=R^qf_{*}\Omega^{p}_{X/\Delta}(\textup{log}E) \Rightarrow \R^{p+q}f_{*}\Omega^{\bullet}_{X/\Delta}(\textup{log}E)=\Tilde{\mathcal{H}}^{p+q}$$ degenerates at $E_1$. Thus, for $t\in\Delta^*$, the Hodge to de Rham spectral sequence for $X_t$ degenerates at $E_1$ and the coherent sheaves $R^qf_{*}\Omega^{p}_{X/\Delta}(\textup{log}E)$ are locally free.

 \end{enumerate}
\end{theorem}

\begin{remark}\label{neg}
  It should be noticed that the nilpotent operator $N$ under Steenbrink's setup is different from the nilpotent operator $N$ in the abstract setup (that is, the nilpotent operator $N$ of Schmid \cite{schmid}, Cattani-Kaplan-Schmid \cite{CKS} and in Section $3$) by a \textit{negative} sign. The reason is that under the abstract setup, we use the inverse of the monodromy operator to trivialize the local system on the upper half-plane, which leads to a negative sign on the nilpotent operator. For some more elaboration on this difference, we can refer to \cite[6.27, 6.28]{Fujisawa2}.
\end{remark}

\subsection{Local Invariant Cycle Properties and Monodromy Weight Conditions}

Throughout this subsection, let $f: X\rightarrow\Delta$ be a semistable degeneration of $m$-dimensional compact complex manifolds, as in the previous subsection, satisfying \hyperref[WeakKahler]{$(\mathbf{K}_1)$}.

\begin{definition} \label{monodromyweight} We say $f:X \rightarrow \Delta$ satisfies the monodromy weight condition $$ W^{\mathrm{St}}=W(N,d) \quad\textup{on}\quad H^d(X_{\infty},\mathbb{Q}), $$ if the weight filtration $W^{\mathrm{St}}$ of Steenbrink's limiting mixed Hodge structure satisfies  $$\textup{Gr}_{d+r}^{W^{St}}H^d(X_{\infty},\Q)\xrightarrow{N^r}\textup{Gr}_{d-r}^{W^{St}}H^d(X_{\infty},\Q)$$ is an isomorphism for all $r\geq0.$
\end{definition}

\begin{theorem} \cite[Proposition 3.22]{Zein} \label{Zeinthm}
    Let $f: X\rightarrow\Delta$ be a semistable degeneration of $m$-dimensional compact complex manifolds satisfying \hyperref[WeakKahler]{$(\mathbf{K}_1)$}. Fix $d\geq 0$. Then the following are equivalent:
    \begin{enumerate}
        \item The local invariant cycle properties  \hyperref[localinvariantcycle]{$(\mathbf{C}_d)$} and \hyperref[localinvariantcycle]{$(\mathbf{C}_{2m-d})$} hold.
        \item  The monodromy weight condition $$ W^{\mathrm{St}}=W(N,d) \quad\textup{on}\quad H^d(X_{\infty},\mathbb{Q}) $$ holds. 
    \end{enumerate}
\end{theorem}

\begin{remark} By \cite[Proposition 3.20]{Zein}, $$ \operatorname{Gr}_{d+r}^{W^{\mathrm{St}}} H^d(X_{\infty},\mathbb{Q}) $$ is Poincar\'e dual to $$ \operatorname{Gr}_{2m-d-r}^{W^{\mathrm{St}}} H^{2m-d}(X_{\infty},\mathbb{Q}). $$ Consequently, $$ W^{\mathrm{St}}=W(N,d) \quad\text{on}\quad H^d(X_{\infty},\mathbb{Q}) $$ if and only if $$ W^{\mathrm{St}}=W(N,2m-d) \quad\text{on}\quad H^{2m-d}(X_{\infty},\mathbb{Q}). $$ \end{remark}

\begin{proof}[Sketch of Proof of Theorem \ref{Zeinthm}]
    Although the original theorem is stated under the assumption that $X$ is K\"ahler, its proof remains valid under the weaker assumption \hyperref[WeakKahler]{$(\mathbf{K}_1)$}.

The implication $(2)\Rightarrow\hyperref[localinvariantcycle]{(\mathbf{C}_d)}$ is due to Deligne; see also \cite[Theorem 11.43]{PS}.  Conversely, the main observation from \cite[Proposition 3.22]{Zein} is that if $$v\in\textup{Gr}_{d+r}^{W^{St}}H^d(X_{\infty},\Q)\quad \textup{and}\quad N^rv=0 \quad \textup{in} \quad \textup{Gr}_{d-r}^{W^{St}}H^d(X_{\infty},\Q),$$ then $$Nv=0 \quad \textup{in} \quad \textup{Gr}_{d+r-2}^{W^{St}}H^d(X_{\infty},\Q).$$ This assertion follows from a diagram-chasing argument involving the $E_1$ page of the spectral sequence of Steenbrink's weight filtration. Once it is established, condition \hyperref[localinvariantcycle]{$(\mathbf{C}_d)$} implies that $v=0$. Therefore, $N^r$ is injective on $\textup{Gr}_{d+r}^{W^{St}}H^d(X_{\infty},\Q)$. Applying the same argument in the complementary degree yields the corresponding injectivity statement on $H^{2m-d}(X_{\infty},\Q)$. Together with Poincar\'e duality, this implies that $N^r$ is also surjective, and hence is an isomorphism.
\end{proof}

One important consequence of the strong K\"ahler-type condition \hyperref[StrongKahler]{$(\mathbf{K}_2)$} is the following:

\begin{theorem} (\cite{guillen}, or see \cite[ Section 11.3.2]{PS}) \label{GNTheorem}
      Let $f: X\rightarrow\Delta$ be a semistable degeneration of $m$-dimensional compact complex manifolds satisfying \hyperref[StrongKahler]{$(\mathbf{K}_2)$}. Then $$ W^{\mathrm{St}}=W(N,d) \quad\textup{on}\quad H^d(X_{\infty},\mathbb{Q}) $$ for every $d\geq 0$.
\end{theorem}

For later use, we give a more precise description of Steenbrink's weight filtration. See \cite[Sections 11.2 and 11.3]{PS} for proofs and further details. Let $E_i$ be irreducible components of $E$. For each $l\geq 1$ and each $I=(i_1,...,i_l)$, we denote $E_I$ the intersection $E_{i_1}\cap...\cap E_{i_l}$ and $E(l)$ the disjoint union of $E_I$ with $|I|=l$. We have natural inclusions $\rho^I_j:E_I\rightarrow E_J$ where $J=(i_1,...,\hat{i_j},...,i_l)$ and $\rho_j^l:\bigoplus_{|I|=l}\rho^I_j:E(l)\rightarrow E(l-1)$.

We denote $\gamma_l$ the alternating sum of the Gysin morphisms: $$\gamma_l=\bigoplus_{j=1}^l(-1)^{j-1}(\rho^l_j)_{!}:H^k(E(l))\rightarrow H^{k+2}(E(l-1))(1)$$ and $\theta_{l}$ the alternating sum of the pullback morphism: $$\theta_{l}=\bigoplus_{j=1}^l(-1)^{j-1}(\rho^{l}_j)^*:H^k(E(l))\rightarrow H^{k}(E(l+1)).$$

The $E_1$ page of the weight spectral sequence is $${}_{W^{St}}E_1^{-r,d+r}=\bigoplus_{k\geq 0,-r} H^{d-r-2k}(E(2k+r+1))(-r-k)$$ and the differential $d_1$ corresponds to the map $-\bigoplus_{k\geq 0, -r}\gamma_{r+2k+1}+\bigoplus_{k\geq 0, -r}\theta_{r+2k+1}$ on ${}_{W^{St}}E^{-r,d+r}_1$ (which will be denoted by $-\gamma+\theta$ for simplicity): $$ \bigoplus_{k\geq 0,-r} H^{d-r-2k}(E(2k+r+1))(-r-k)\rightarrow  \bigoplus_{k\geq 0,-r+1} H^{d-r-2k+2}(E(2k+r))(-r-k+1).$$

For all $r\geq 0$, note that 
\begin{equation*}
    \begin{split}
   {}_{W^{St}}E_1^{r,d-r} &=\bigoplus_{k\geq 0,r} H^{d+r-2k}(E(2k-r+1))(r-k)\\
    &=\bigoplus_{l\geq 0,-r} H^{d-r-2l}(E(2l+r+1))(-l)={}_{W^{St}}E_1^{-r,d+r}(r)
    \end{split}
\end{equation*} by taking $l=k-r$. Let $I$ denote the identity map ${}_{W^{St}}E_1^{-r,d+r}\rightarrow {}_{W^{St}}E_1^{r,d-r}(-r)$ above. Then an important fact is that the morphism $N^r$ from $${}_{W^{St}}E_2^{-r,d+r}=\textup{Gr}_{d+r}^{W^{St}}H^d(X_{\infty})$$ to $${}_{W^{St}}E_2^{r,d-r}=\textup{Gr}_{d-r}^{W^{St}}H^d(X_{\infty})$$ is induced by the map $(2\pi \sqrt{-1})^rI$. Thus, to check that $W^{St}=W(N,d)$ on $H^d(X_{\infty})$, it suffices to show that the isomorphism  $(2\pi \sqrt{-1})^rI$ from ${}_{W^{St}}E_1^{-r,d+r}$ to ${}_{W^{St}}E_1^{r,d-r}$ also induces an isomorphism at the $E_2$ page for every $r\geq 0$.

Summarizing, we have 

\begin{proposition} \label{criterionmonodromy}
    Let $f: X\rightarrow\Delta$ be a semistable degeneration of $m$-dimensional compact complex manifolds satisfying \hyperref[WeakKahler]{$(\mathbf{K}_1)$}. Fix $d,$ $r\geq0$. Then the following are equivalent:
    \begin{enumerate}
        \item The map $$\textup{Gr}_{d+r}^{W^{St}}H^d(X_{\infty},\Q)\xrightarrow{N^r}\textup{Gr}_{d-r}^{W^{St}}H^d(X_{\infty},\Q)$$ is an isomorphism.
        \item The cohomology at the middle term of the complex  $$  \bigoplus_{k\geq 0,-r-1} H^{d-r-2k-2}(E(2k+r+2))\xrightarrow{-\gamma+\theta}\bigoplus_{k\geq 0,-r} H^{d-r-2k}(E(2k+r+1))$$$$\xrightarrow{-\gamma+\theta}  \bigoplus_{k\geq 0,-r+1} H^{d-r-2k+2}(E(2k+r))$$ is isomorphic to the cohomology at the middle term of the complex $$\bigoplus_{k\geq 0,r-1} H^{d+r-2k-2}(E(2k-r+2))\xrightarrow{-\gamma+\theta}\bigoplus_{k\geq 0,r} H^{d+r-2k}(E(2k-r+1))$$$$\xrightarrow{-\gamma+\theta}  \bigoplus_{k\geq 0,r+1} H^{d+r-2k+2}(E(2k-r)),$$ where the isomorphism is induced by the identity map arising from the identification $$\bigoplus_{k\geq 0,-r} H^{d-r-2k}(E(2k+r+1))=\bigoplus_{k\geq 0,r} H^{d+r-2k}(E(2k-r+1)).$$
    \end{enumerate}
\end{proposition}

As observed by Sano \cite[Proposition 2.6 (iii)]{SA3}, if we focus on the middle limiting cohomology when the central fiber $X_0$ has no triple points, the condition that $W^{St}=W(N,m)$ can be reformulated as the non-degeneracy of the cup product on the specific subspace of the middle cohomology of the double locus. In the same spirit, we have the following criterion of $W^{St}=W(N,m)$ on the middle limiting cohomology:

\begin{theorem} \label{criND}
 Let $f: X\rightarrow \Delta$ be a semistable degeneration of $m$-dimensional compact complex manifolds satisfying \hyperref[WeakKahler]{$(\mathbf{K}_1)$}. Then the following are equivalent:

\begin{enumerate}
    \item For each $l> 0$, $\textup{Gr}_{m+l}^{W^{St}}H^{m}(X_{\infty})\xrightarrow{N^l} \textup{Gr}_{m-l}^{W^{St}}H^{m}(X_{\infty})$ is an isomorphism .
    \item For each $l>0$, the null subspace of $$\textup{im}\left(H^{m-l}(E(l))\xrightarrow{\theta_l} H^{m-l}(E(l+1))\right)$$ under cup product is contained in (or equivalently, equal to) $$\textup{im}\left(H^{m-l-2}(E(l+2))\xrightarrow{-\gamma_{l+2}}H^{m-l}(E(l+1))\right).$$
\end{enumerate}
\end{theorem}

\begin{proof}
  For any $l>0$, consider $$V^l:=\textup{im}\left( \bigoplus_{k\geq 1} H^{m-l-2k}(E(2k+l)) \xrightarrow{-\gamma+\theta}\bigoplus_{k\geq 0} H^{m-l-2k}(E(2k+l+1)) \right),$$ $$\hat{V}^l:=\textup{ker}\left( \bigoplus_{k\geq 0} H^{m-l-2k}(E(2k+l)) \xrightarrow{-\gamma+\theta}\bigoplus_{k\geq 0} H^{m-l-2k+2}(E(2k+l)) \right),$$ $$U^l:=\textup{im}\left(\bigoplus_{k\geq 0} H^{m-l-2k}(E(2k+l)) \xrightarrow{-\gamma+\theta} \bigoplus_{k\geq 0} H^{m-l-2k}(E(2k+l+1)) \right),$$ and $$\hat{U}^l:=\textup{ker}\left(\bigoplus_{k\geq 0} H^{m-l-2k}(E(2k+l+1)) \xrightarrow{-\gamma+\theta} \bigoplus_{k\geq 1} H^{m-l-2k+2}(E(2k+l)) \right)$$ We have $V^l\subset U^l$ and $\hat{V}^l\subset \hat{U}^l$. By Proposition \ref{criterionmonodromy}, the map in Theorem \ref{criND} (1) can be written up to Tate twists and reindexing the direct sum as $$\frac{\hat{V}^l}{V^l}\rightarrow \frac{\hat{U}^l}{U^l}.$$ 
  
 We claim that $\hat{V}^l=(U^l)^{\perp}$ and $\hat{U}^l=(V^l)^{\perp}$, where the orthogonal complements are taken with respect to the direct sum of the cup-product pairings on
$$
\bigoplus_{k\geq 0} H^{m-l-2k}(E(2k+l+1)).
$$ Observe that the map defining $V^l$ is Poincar\'e dual to the map defining $\hat{U}^l$ (cf. \cite[Proposition 3.20]{Zein}). Therefore, we have $$\textup{dim}V^l+\textup{dim}\hat{U}^l=\textup{dim}\left(\bigoplus_{k\geq 0} H^{m-l-2k}(E(2k+l+1))\right).$$ On the other hand, since the restriction map and the Gysin map are adjoint to each other under cup product (cf. \cite[Addendum 1.20]{PS}), $V^l$ is orthogonal to $\hat{U}^l$. Since cup product is non-degenerate on $$\bigoplus_{k\geq 0} H^{m-l-2k}(E(2k+l+1)),$$ we have $\hat{U}^l=(V^l)^{\perp}$. A similar argument works for $\hat{V}^l=(U^l)^{\perp}$.

  We thus conclude that the map in Theorem \ref{criND} (1) can be rewritten up to Tate twists as $$\frac{(U^l)^{\perp}}{V^l}\rightarrow \frac{(V^l)^{\perp}}{U^l},$$ where the map is induced by the identity map.  This map is an isomorphism if and only if $(U^l)^{\perp}\cap U^l=V^l$, or equivalently, $V^l$ is equal to the null subspace of $U^l$.
  
   Write $$X^l:=\textup{im}\left(H^{m-l}(E(l))\xrightarrow{\theta_l} H^{m-l}(E(l+1))\right).$$ Then we have $U^l=V^l+X^l$. Now assume (1) holds. For any $\alpha$ in the null subspace of $X^l$, since $V^l\subset (U^l)^{\perp}$, $\alpha$ is also a null vector in $U^l$, and hence $\alpha \in V^l$ by (1). By looking at the $H^{m-l}(E(l+1))$ factor of $V^l$, we immediately see that $\alpha=-\gamma_{l+2}(\beta)$ for some $\beta \in H^{m-l-2}(E(l+2))$. Thus (2) holds. 
   
   Conversely, assume (2) holds. We will prove by reverse induction that (1) holds as well. First, we show that (1) and (2) automatically holds for $l=m-1$ and $l=m$. Note that $H^1(E(m))$ and $H^0(E(m+1))$ are polarized Hodge structures and every sub Hodge structure of a polarized Hodge structure is also polarized. As a result, since $U^l$ is a subspace of $H^1(E(m))$ or $H^0(E(m+1))$ for $l=m-1$ or $l=m$ respectively, we have $U^l\cap (U^l)^{\perp}=0$, and hence (1) holds for $l=m, m-1$. A similar argument shows that (2) holds for $l=m, m-1$, too. Now, for any $\alpha$ in the null subspace of $U^l$, we write $\alpha =\alpha_1+ \alpha_2$, where $\alpha_1\in V^l$ and $\alpha_2\in X^l$. Since $\alpha_2=\alpha-\alpha_1$ is a null vector, by (2), we have $\alpha_2=-\gamma_{l+2}(\beta_2)$ for some $\beta_2\in H^{m-l-2}(E(l+2)).$ Since $-\gamma_{l+2}(\beta_2)+\theta_{l+2}(\beta_2)\in V^l$, we can rewrite $\alpha=\alpha_1'-\theta_{l+2}(\beta_2)$ with $\alpha_1'\in V^l$. Now $-\theta_{l+2}(\beta_2)\in X^{l+2}\subset U^{l+2}$ is also a null vector, and hence by induction we have $-\theta_{l+2}(\beta_2)\in V^{l+2}\subset V^l$. This shows that $\alpha\in V^l$ and the proof is complete.
   
\end{proof}

\subsection{Cup Product on Limiting Mixed Hodge Structures} \label{cupproduct}

Throughout this subsection, $f: X\rightarrow\Delta$ is a semistable degeneration of $m$-dimensional compact complex manifolds, as in the previous subsections, satisfying \hyperref[WeakKahler]{$(\mathbf{K}_1)$}.

 We turn to the product structure on the limiting mixed Hodge structure. Since the nearby cocycle complex is quasi-isomorphic to the relative log de Rham complex, there is a natural cup product defined on $H^{\bullet}(X_{\infty})$, which will be denoted by $Q(\bullet,\bullet)$. More precisely, from the product structure on $\Omega_{X/\Delta}^{\bullet}(\textup{log}(E))$, we have a natural product morphism $$\R^{i}f_{*}\Omega_{X/\Delta}^{\bullet}(\textup{log}(E))\otimes \R^jf_{*}\Omega_{X/\Delta}^{\bullet}(\textup{log}(E))\rightarrow \R^{i+j}f_{*}\Omega_{X/\Delta}^{\bullet}(\textup{log}(E)).$$ Also, we have a relative trace morphism $$\R^{2m}f_{*}\Omega_{X/\Delta}^{\bullet}(\textup{log}(E))=R^{m}f_{*}\omega_{X/\Delta}\rightarrow \mathcal{O}_{\Delta},$$ where $\omega_{X/\Delta}$ is the relative dualizing sheaf of $f$. The cup product $$Q: H^{\bullet}(X_{\infty},\C)\otimes H^{2m-\bullet}(X_{\infty},\C)\rightarrow \C$$ is identified with the fiber over zero of the (relative) cup product $$\R^{\bullet}f_{*}\Omega_{X/\Delta}^{\bullet}(\textup{log}(E))\otimes \R^{2m-\bullet}f_{*}\Omega_{X/\Delta}^{\bullet}(\textup{log}(E))\rightarrow \mathcal{O}_{\Delta}.$$

\begin{remark}
 On a $m$-dimensional compact complex manifold $Z$, the natural trace map in the duality theory is different from the integration map by a sign of $(-1)^m$. To be precise, if we let $$\int_Z : H^{2m}(Z,\C)\rightarrow \C$$ be the integration map on $Z$, then the natural trace map in the duality theory is $(-1)^m\int_Z$ or $\frac{(-1)^m}{(2\pi\sqrt{-1})^m}\int_Z$ depending on whether we fix a choice of $\sqrt{-1}$ or not. We refer to \cite{BS} and \cite{RRV} for the absolute and relative duality theory on complex spaces and \cite{Con} for the clarification of the signs. In our case, we would normalize the relative trace map so that the restriction of $$R^mf_{*}\omega_{X/\Delta}\rightarrow\mathcal{O}_{\Delta}$$ to the general fiber $X_t$ is just the integration map $$\int_{X_t}:H^{2m}(X_t,\C)\rightarrow\C.$$
\end{remark}

Since Steenbrink's limiting mixed Hodge complex $s(A^{\bullet,\bullet})$ doesn't carry a natural product structure, it is not clear from the definition how Steenbrink's weight filtration $W^{St}$ behaves under the cup product. To study the behavior of $W^{St}$ under the cup product, Fujisawa \cite{fujisawa} constructs another mixed Hodge complex of sheaves $(K_{\Q},K_\C,W,F)$ which carries a natural product structure and admits a quasi-isomorphism of bigraded complexes $$(s(C^{\bullet,\bullet}),s(A^{\bullet,\bullet}),W^{St},F)\rightarrow (K_{\Q}, K_\C,W,F).$$ Hence, there is a product morphism  $$\mathbb{H}^i(E,K_{\C})\otimes \mathbb{H}^j(E,K_{\C})\rightarrow \mathbb{H}^{i+j}(E,K_{\C}).$$  Fujisawa also constructs a trace morphism $$\textup{Tr}: \mathbb{H}^{2m}(E,K_{\C})\rightarrow \C$$ and then defines the cup product pairing $$Q_K: \mathbb{H}^{\bullet}(E,K_{\C})\otimes \mathbb{H}^{2m-\bullet}(E,K_{\C})\rightarrow \C$$ in [ibid., Definition 7.11 and  Definition 7.13]. 

 It is not stated in \cite{fujisawa} that the cup product $Q_K$ defined in [ibid., Definition 7.13] coincides with the one we describe above. In \cite{FuNa}, Fujisawa-Nakayama checked the compatibility of the cup product $Q_K$ defined in \cite{fujisawa} and the cup product on the Kato-Nakayama space $E^{\textup{log}}_{\infty}$ of $E$. Using similar tricks, we can check the compatibility of the cup product $Q_K$ defined in \cite{fujisawa} with the cup product $Q$ defined on $H^{\bullet}(X_{\infty},\C)$. We sketch the proof of this compatibility below:

\begin{lemma}
    The cup product $Q_K$ defined in \cite{fujisawa} is compatible with the cup product $Q$ described above.
\end{lemma}
\begin{proof}

 It is not hard to see that the product morphisms of $Q_K$ and $$\Omega^{\bullet}_{X/\Delta}(\textup{log}(E))\otimes \mathcal{O}_{E}$$ are compatible. See the last paragraph of the proof of \cite[ Proposition 4.6]{FuNa} for a detailed explanation.
 
 The nontrivial part is to show that the trace morphism defined in \cite[Definition 7.11]{fujisawa} coincides with the one induced by the dualizing sheaf $\omega_E$ on $E$. To see this, let $E_{\textup{sing}}$ be the subspace consisting of the singular points of $E$ and let $U=E-E_{\textup{sing}}$. The trick is that since $H^m(E_{\textup{sing}},\omega_{E})=0$, there is a surjective morphism $$H^m_c(U,\omega_{E})\twoheadrightarrow H^m(E,\omega_{E})\xrightarrow{\textup{Tr}} \C,$$ where $H^m_c(U,\omega_{E})$ is the cohomology with compact support of $\omega_E$ on $U$. Thus, the trace morphism is determined by the composite $H^m_c(U,\omega_{E})\rightarrow\C$, which is simply the map induced by integration on $U$ because the trace morphism in the duality theory is defined by gluing the trace morphisms on a Stein open cover (cf. \cite[p.262-263]{BS}). Therefore, it suffices to check that the composite of the trace morphism $$\textup{Tr}: \mathbb{H}^{2m}(E, K_{\C})\rightarrow \C$$ defined in \cite[Definition 7.11]{fujisawa} with the morphism $$H^{m}_c(U,\omega_U)\rightarrow H^{2m}_c(U,\C)\rightarrow H^{2m}_c(U, W_0K_{\C})\rightarrow\mathbb{H}^{2m}_c(U, K_{\C})\rightarrow\mathbb{H}^{2m}(E, K_{\C})$$ becomes integration on $U$. This is shown in \cite[Proposition 4.7]{FuNa}.
\end{proof}

As a result, by Fujisawa \cite{fujisawa}, the cup product $Q$ satisfies the following properties:

\begin{proposition}\label{Fu} Suppose that $x\in H^{d}(X_{\infty},\Q)$ and $y\in H^{2m-d}(X_{\infty},\Q)$ and let $N$ denote the nilpotent operator. Then

\begin{enumerate}
    \item $Q(y,x)=(-1)^{d}Q(x,y)$.
    \item $Q(Nx,y)+Q( x,Ny)=0$
    \item For all $p$, $Q( F^pH^{d}(X_{\infty},\C),F^{m-p+1}H^{2m-d}(X_{\infty},\C))=0$
    \item $Q( W_a^{St} H^{d}(X_{\infty},\Q),W_b^{St} H^{2m-d}(X_{\infty},\Q))=0$ for $a+b\leq 2m-1$.
\end{enumerate}
\end{proposition}

 By Proposition \ref{Fu} (4), the cup product induces the morphism $$Gr_{d+r}^{W^{St}}H^{d}(X_{\infty},\Q)\otimes_{\Q}Gr^{W^{St}}_{2m-d-r}H^{2m-d}(X_{\infty},\Q)\rightarrow \Q$$ for each $d$ and $r$, still denoted by $Q$.  Consider  $\langle\bullet,\bullet\rangle=\epsilon(d-2m)Q(\bullet,\bullet)$. To calculate $\langle\bullet,\bullet\rangle$, we introduce another linear pairing $\psi$ as follows: note that $\langle\bullet,\bullet\rangle$ is a map $${}_{W^{St}}E_2^{-r,d+r}\otimes_{\C}{}_{W^{St}}E_2^{r,2m-d-r}\rightarrow \Q.$$ At the ${}_{W^{St}}E_1$ page, we have $${}_{W^{St}}E_1^{-r,d+r}=\bigoplus_{k\geq 0,-r} H^{d-r-2k}(E(2k+r+1))(-r-k)$$ and $${}_{W^{St}}E_1^{r,2m-d-r}=\bigoplus_{l\geq 0,r} H^{2m-d+r-2l}(E(2l-r+1))(r-l)$$$$=\bigoplus_{k\geq 0,-r} H^{2m-d-r-2k}(E(2k+r+1))(-k)$$ by taking $k=l-r$. Consider the linear mapping $$\psi:{}_{W^{St}}E_1^{-r,d+r}\otimes_{\C}{}_{W^{St}}E_1^{r,2m-d-r}\rightarrow \Q(-m),$$ defined by $$\psi(\eta,\xi)=
\frac{\epsilon(r+d-2m)}{(2\pi \sqrt{-1})^{m-2k-r}}\int_{E(2k+r+1)}\eta\wedge \xi$$ for $\eta\in H^{d-r-2k}(E(2k+r+1))(-r-k)$ and $\xi\in H^{2m-d-r-2k}(E(2k+r+1))(-k) $ and $$\psi( H^{d-r-2k_1}(E(2k_1+r+1)(-r-k_1)), H^{2m-d-r-2k_2}(E(2k_2+r+1))(-k_2))=0$$ for $k_1\neq k_2$. Then Fujisawa \cite{fujisawa} shows that $\langle\bullet,\bullet\rangle$ can be calculated in the following way: 
\begin{proposition} \label{Fu2}
We have $$\langle\tilde{\eta},\tilde{\xi}\rangle =(2\pi \sqrt{-1})^m\psi (\eta, \xi),$$ where $\eta,\ \xi$ are any representative of $\tilde{\eta},\ \tilde{\xi}$ at the ${}_{W^{St}}E_1$ page.  
\end{proposition}

\begin{proof}[Proof of Proposition \ref{Fu} and Proposition \ref{Fu2}]
The statements (1), (2), (3), (4) in Proposition \ref{Fu} are proved in \cite[\S 7]{fujisawa}. Proposition \ref{Fu2} follows from \cite[Lemma 6.13]{fujisawa} and the computation in the proof of [ibid., Theorem 8.11]
\end{proof}

\begin{remark}
    The linear pairing $\psi$ is first defined by Guill\'en-Navarro Aznar \cite[(3.4)]{guillen} to construct the polarized Hodge-Lefschetz module structure under \hyperref[StrongKahler]{$(\mathbf{K}_2)$} (or see \cite[\S 11.3.2]{PS}). Fujisawa \cite{fujisawa} studies the behavior of the weight filtration under the cup product structure and shows that such a linear pairing $\psi$ comes geometrically from the cup product structure on the limiting mixed Hodge structure.
\end{remark}

\section{Degenerations of Abstract Hodge Filtrations}

\subsection{Canonical Extensions and Limiting Mixed Hodge Structures}

Let $\mathcal{H}$ be a holomorphic vector bundle on the punctured disc $\Delta^*$ with an integrable connection $\nabla$. Consider the canonical extension $\Tilde{\mathcal{H}}$ of $(\mathcal{H},\nabla)$ of $\mathcal{H}$ to the unit disc $\Delta$ so that $\nabla$ extends to a regular-singular connection on $\Delta$. Let $T$ be the monodromy automorphism of the local system of horizontal sections $\mathcal{H}^{\nabla}$ determined by a counter-clockwise loop around $0$, i.e. at each stalk of $\mathcal{H}^{\nabla}$ over $\Delta^*$, $T$ is obtained by the inverse of the pullback along the counter-clockwise loop by parallel translation. The convention for $T$ here is compatible with the convention for the Picard-Lefschetz monodromy $T$ in the next section.  Assume that $T$ is unipotent and let $N=-\text{log}(T)$. Then by \cite{deligne}, we can write $\Tilde{\mathcal{H}}=H\otimes \mathcal{O}_{\Delta}$, where $H$ is the $\C$-vector space of multivalued horizontal sections of $\mathcal{H}$, and the fiber $\Tilde{\mathcal{H}}_0$ of $\Tilde{\mathcal{H}}$ over $0$ is identified with $H$ up to the action of the unipotent group $\{\textup{exp}(\lambda N) : \lambda \in \C\}$. The fiber of $\mathcal{H}$ at any point is identified with $H$ modulo the action of $\{ \textup{exp}(kN): k \in \Z\}$.  The connection $\nabla$ is associated to the connection $1$-form $\frac{N}{2\pi\sqrt{-1}}\frac{dq}{q}$, where $q$ is the coordinate of $\Delta^*.$ 

Write $$q=e^{2\pi\sqrt{-1}z},\  z=\frac{\text{log} q}{2\pi \sqrt{-1}},$$ where $z$ is the uniformizer of $\Delta^*$ and $\varphi: \mathfrak{H}\rightarrow \Delta^*$ the corresponding uniformization. A holomorphic section $v$ of $\mathcal{H}$, viewed as a holomorphic section of the trivial bundle $\varphi^*\mathcal{H}=H\otimes \mathcal{O}_{\mathfrak{H}}$ with the invariance property $v(z-1) = Tv(z)$, extends to a holomorphic section of $\tilde{\mathcal{H}}$ if and only if the section $\textup{exp}(-zN)\varphi^*v$, viewed as a holomorphic
section of $\Tilde{\mathcal{H}}$, extends to a (single-valued) holomorphic function from $\Delta$ to $H$.

We first recall some definitions in the context of the degeneration of Hodge structures.

\begin{definition}\label{Def}
Let $H$ be a finite-dimensional complex vector space, and $N$ a nilpotent endomorphism of $H$ defined over $\R,$
\begin{enumerate}
    \item The monodromy weight filtration of $N$ centered at $d$ is the unique increasing filtration $W=W(N,d)$ of $H$ such that $N(W_i)\subset W_{i-2}$ for $i\geq 2$ and the map $$N^l: \textup{Gr}^W_{d+l}H\rightarrow \textup{Gr}^W_{d-l}H$$ is an isomorphism for all $l$.
    \item (\cite[Definition 2.26]{CKS})  Suppose we are given a non-degenerate bilinear form $S$ on $H$ defined over $\R$ which is $(-1)^d$-symmetric: $S(v,u)=(-1)^dS(u,v)$, and such that $N$ is an infinitesimal isometry: $S(u,Nv)+S(Nu,v)=0$. Then a mixed Hodge structure $(H,W,F)$ on $H$ is polarized by $N$ centered at $d$ if 
    \begin{enumerate}
        \item $N^{d+1}=0$;
        \item $W=W(N,d)$;
        \item $S(F^p,F^{d-p+1})=0$;
        \item $NF^p\subset F^{p-1}$;
        \item the Hodge structure (of weight $d+l$) on the primitive part $$P_{d+l} = \text{ker}(N^{l+1}:\text{Gr}^{W}_{d+l}H\rightarrow \text{Gr}^{W}_{d-l-2}H)$$ is polarized by the form $S(\bullet, N^l\bullet)$.
    \end{enumerate}
\end{enumerate}
\end{definition}

A well-known consequence of Schmid's nilpotent orbit theorem and $SL_2$ orbit theorem (cf. \cite{schmid}) is that if $\mathcal{H}$ underlies an integral polarized variation of Hodge structure of weight $d$, then Hodge bundles $\mathcal{F}^p$ extends to holomorphic subbundles $\Tilde{\mathcal{F}}^p$ of $\Tilde{\mathcal{H}}$ so that the triple $(\Tilde{\mathcal{H}}_0, W(N,d), \Tilde{\mathcal{F}}_0)$ is a polarized mixed Hodge structure, where $W(N,d)$ is the monodromy weight filtration of $N$ centered at $d$.

\subsection{The Abstract Degeneration Theorem}
This section aims to deal with the converse situation of Schmid's theorem. We are interested in the following two situations:

\begin{enumerate}
    \item[(\textbf{A$'$})] We are given a mixed Hodge structure $(H, W, F)$ defined over $\R$ and a nilpotent endomorphism $N$ on $H$ such that \begin{enumerate}
        \item $N$ is a morphism of mixed Hodge structure of type $(-1,-1)$ 
        \item $W=W(N,d)$ is the weight filtration of $N$ centered at $d$.
    \end{enumerate}
    \item[(\textbf{B$'$})] Besides the assumption in (\textbf{A$'$}), we are  given a non-degenerate bilinear form $S$ on $H$ defined over $\R$ which satisfies \begin{enumerate}
    \item $S(v,u)=(-1)^dS(u,v)$,
    \item  $S(u,Nv)+S(Nu,v)=0$,
    \item $S(F^{p},F^{d-p+1})=0$ for all $p$.
\end{enumerate}
\end{enumerate}
 
 In the situation (\textbf{A$'$}), since $W=W(N,d)$, we have $$\textup{dim}_{\C}\textup{Gr}_{F}^p H=\textup{dim}_{\C}\textup{Gr}_{F}^{d-p}H,$$ and hence $$\textup{dim}F^p+\textup{dim}\overline{F}^{d+1-p}=\textup{dim}H$$ for all $p$. Also, in the situation (\textbf{B$'$}), we have $S(W_aH, W_bH)=0$ for $a+b\leq 2d-1$ (see e.g. \cite[Lemma 6.4]{schmid}). Thus, $S(\bullet, N^l\bullet)$ defines a non-degenerate pairing on $\textup{Gr}^{W}_{d+l}H$. Moreover, by the first Hodge-Riemann condition on $F$,  $S(\bullet, N^l\overline{\bullet})$ is non-degenerate on the $(p,d+l-p)$-component of the Hodge structure on $\textup{Gr}^{W}_{d+l}H$ (of weight $d+l$) for all $p$ and $l$. Note that $S(\bullet, N^l\overline{\bullet})$ is also non-degenerate on the $(p,d+l-p)$-component of the Hodge structure (of weight $d+l$) on the primitive part $$P_{d+l} := \textup{ker}(N^{l+1}:\textup{Gr}^{W}_{d+l}H\rightarrow \textup{Gr}^{W}_{d-l-2}H)$$ for all $p$ and $l$. Let $(s_{+}^{p,d+l-p},s_{-}^{p,d+l-p})$ denote the signature of $S(C\bullet, N^l\overline{\bullet})$ on the $(p,d+l-p)$-component of $P_{d+l},$ where $C$ is the Weil operator defined by $(\sqrt{-1})^{p-q}$ on the $(p,q)$-component of a Hodge structure of weight $p+q$. Since $S(u,Nv)+S(Nu,v)=0$, the signature of $S(C\bullet, N^l\overline{\bullet})$ on the $(p,d+l-p)$-component of the Hodge structure on $\textup{Gr}^{W}_{d+l}H$ is $$\left(\sum_{r\geq 0 }s_{+}^{p+2r,d+l-p+2r}+\sum_{r\geq 0 }s_{-}^{p+2r+1,d+l-p+2r+1}, \ \sum_{ r\geq 0} s_{-}^{p+2r,d+l-p+2r}+ \sum_{r\geq 0 }s_{+}^{p+2r+1,d+l-p+2r+1}\right).$$

Given an integer $n$, set $n_{+}=\textup{max}\{n,0\}$. In the situation (\textbf{B$'$}), for $0\leq p\leq d+l$, define the invariants $$\textbf{S}_{+}^{p,d+l-p}=\sum_{r\geq (-l)_{+}}s_{+}^{p+r,d+l-p+r}$$ and $$\textbf{S}_{-}^{p,d+l-p}=\sum_{r\geq (-l)_{+}}s_{-}^{p+r,d+l-p+r},$$ which corresponds to (the $N$-Lefschetz decomposition of) the $(p,d+l-p)$-component of the Hodge structure on $\textup{Gr}^{W}_{d+l}H$. Observe that $\textbf{S}_{+}^{p,d+l-p}+\textbf{S}_{-}^{p,d+l-p}$ is equal to the dimension of the $(p,m+l-p)$-component of the Hodge structure on $\textup{Gr}^W_{d+l}H$, but when $l\geq 0$, $(\textbf{S}_{+}^{p,d+l-p},\textbf{S}_{-}^{p,d+l-p})$ is not the signature of the Hermitian form $S(C\bullet,N^l\overline{\bullet})$ on the $(p,d+l-p)$-component of the Hodge structure on $\textup{Gr}^W_{d+l}H$ since $S(Nu,v)+S(u,Nv)=0$ for all $u$ and $v$.

Our main theorem is:

\begin{theorem}\label{main}
  Suppose we are given a decreasing filtration $\mathcal{F}^{\bullet}$ of holomorphic subbundles of $\mathcal{H}$. Let $\tilde{\mathcal{F}}^{\bullet}$ be the corresponding holomorphic subbundles of $\Tilde{\mathcal{H}}|_{\Delta^*}$ under canonical extension. We suppose that $\tilde{\mathcal{F}}^{\bullet}$ extend to holomorphic subbundles of $\Tilde{\mathcal{H}}$, still denoted by $\tilde{\mathcal{F}}^{\bullet}$. Consider the fiber $\tilde{\mathcal{F}}^{\bullet}_0$ of $\tilde{\mathcal{F}}^{\bullet}$ over $0$.
 \begin{enumerate}
 
 \item If the triple $(\tilde{\mathcal{H}}_0,W(N,d),\tilde{\mathcal{F}}_0)$ satisfies the situation (\textbf{A$'$}), then the fiber $\mathcal{F}_q^{\bullet}$ of $\mathcal{F}^{\bullet}$ is $d$-opposed to its complex conjugate for $|q|$ small. In other words, $\mathcal{F}^{\bullet}_q$ induces a pure Hodge structure on $\mathcal{H}_q$ of weight $d$ for $|q|$ small.
     \item Suppose moreover that we are given a non-degenerate bilinear form $S$ on $H$ defined over $\R$ which makes $(\Tilde{\mathcal{H}}_0,W(N,d),\tilde{\mathcal{F}}_0)$ satisfy the situation (\textbf{B$'$}). We also replace the assumption (c) of the situation (\textbf{B$'$}) by a stronger first Hodge-Riemann bilinear relation: $S(\mathcal{F}^{p},\mathcal{F}^{d-p+1})=0$ for all $p$. 
     
     Then the signature of $S(C\bullet, \overline{\bullet})$ on the $(p,d-p)$-component of the Hodge structure $\mathcal{H}_q$ becomes $$\left(\sum_{ k=0}^d \textup{\textbf{S}}_{+}^{p,k},\sum_{ k=0}^d \textup{\textbf{S}}_{-}^{p,k}\right)=\left(\sum_{\substack{l\geq p-d, \\ r\geq(-l)_{+}} }s_{+}^{p+r,d+l-p+r}, \ \sum_{\substack{l\geq p-d, \\ r\geq(-l)_{+}} }s_{-}^{p+r,d+l-p+r}\right)$$ for $|q|$ sufficiently small.
     
     In particular, if $(\tilde{\mathcal{H}}_0,W(N,d),\tilde{\mathcal{F}}_0)$ is a polarized mixed Hodge structure centered at $d$, (i.e. $s_{-}^{p,d+l-p}=0$ for all $p$ and $l$), then $\mathcal{F}_q$ induces a polarized Hodge structure of weight $d$ on $\mathcal{H}_q$ for $|q|$ small.
 \end{enumerate}
\end{theorem}

\begin{remark}\ 
\begin{enumerate}
    \item The condition that $S(\mathcal{F}^{p},\mathcal{F}^{d-p+1})=0$ for all $p$ implies the assumption (c) in the situation (\textbf{B$'$}) since $$S(\textup{exp}(zN)v,\textup{exp}(zN)v')=S(v,v')$$ for all vectors $v$ and $v'$ in $H$. 
\item In the theorem, it is surprising that we do not assume that the filtration $\mathcal{F}^{\bullet}$ is horizontal: $\nabla_q\mathcal{F}^p\subset \mathcal{F}^{p-1}$, while it is a natural condition in the geometric setting.

\item When the triple $(\tilde{\mathcal{H}}_0,W(N,d),\tilde{\mathcal{F}}_0)$ is a polarized mixed Hodge structure, the consequence of Theorem \ref{main} is proved by Cattani-Kaplan \cite[Theorem 2.8]{CK2} under the assumption of horizontality. In such cases, the theorem also holds for variations in several variables. 
\end{enumerate}
\end{remark}

\subsection{Generalized Converses to the Nilpotent Orbit Theorem}

The idea to prove Theorem \ref{main} is to compare our Hodge filtration $\mathcal{F}_q$ with the nilpotent orbit $exp(zN)F^{\bullet}$. To be precise, the following two theorems are the unpolarized version and the generalized version of the converse of the nilpotent orbit theorem of Cattani-Kaplan-Schmid (\cite{CKS}, Corollary 3.13), which will serve as the leading terms of the estimates on $\mathcal{F}_q$ in the proof of our main theorem. We will prove the following two theorems in the last section of the paper.

\begin{theorem}\label{thmcks2}
For each mixed Hodge structure $(H, W, F)$ in the situation (\textbf{A$'$}), the filtration $exp(zN)F^{\bullet}$ is $d$-opposed to its complex conjugate for $\textup{Im}(z)$ sufficiently large.   
\end{theorem}

\begin{theorem} \label{thmcks}
For each mixed Hodge structure $(H, W, F)$ and non-degenerate bilinear form $S$ on $H$ in the situation (\textbf{B$'$}), 
\begin{enumerate}
    \item The Hodge filtration $F^{\bullet}$ can be refined to a filtration $\hat{F}^{\bullet}$ such that \begin{enumerate}
        \item $\textup{dim}_{\C}\textup{Gr}_{\hat{F}}^kH=1$ for all $k$,
        \item The restriction of the Hermitian form $(\sqrt{-1})^{d}S(\bullet,\overline{\bullet})$ on $$\textup{exp}(zN)\hat{F}^k$$ is non-degenerate for  $\text{Im}(z)$ sufficiently large and for all $k$.
    \end{enumerate}
        
\item The signature of the Hermitian form $S(C\bullet, \overline{\bullet})$ on $$\textup{exp}(zN)F^p\bigcap \overline{\textup{exp}(zN)F^{d-p}}$$ becomes $$\left(\sum_{ k=0}^d \textup{\textbf{S}}_{+}^{p,k},\sum_{ k=0}^d \textup{\textbf{S}}_{-}^{p,k}\right)=\left(\sum_{\substack{l\geq p-d, \\ r\geq(-l)_{+}} }s_{+}^{p+r,d+l-p+r}, \ \sum_{\substack{l\geq p-d, \\ r\geq(-l)_{+}} }s_{-}^{p+r,d+l-p+r}\right)$$ for $\textup{Im}(z)$ sufficiently large.
     
In particular, if $(H,W,F)$ is a polarized mixed Hodge structure centered at $d$, (i.e. $s_{-}^{p,d+l-p}=0$ for all $p$ and $l$), then $exp(zN)F^{\bullet}$ is a nilpotent orbit.

\end{enumerate}
\end{theorem}

\begin{remark}
Theorem \ref{thmcks2} and Theorem \ref{thmcks} are generalizations of the converse of nilpotent orbit theorem of Cattani-Kaplan-Schmid \cite[Corollary 3.13]{CKS}. It is possible to follow their methods to prove the first theorem. More precisely, if we consider the classifying space $D$ of unpolarized Hodge structures according to the assumption in Theorem \ref{thmcks2}, then $D$ is still a homogeneous space of the form $G(\mb{R})/K$, where $G(\mb{R})$ is a real Lie group and $K$ is a subgroup of $G(\mb{R})$, at the expense that $K$ might be non-compact. Since $D$ is still a real open subset of its compact dual $D^{\vee}$, we can follow the line of Cattani-Kaplan \cite[Proposition 2.18]{CK} to deal with the case of the mixed Hodge structure split over $\R$ and then use the same methods as in \cite[Corollary 3.13]{CKS} to reduce the general case to the case of the mixed Hodge structure split over $\R$. In the last section of the paper, we will prove the two theorems above by direct combinatorial calculation. 
\end{remark}

\subsection{Proof of the Abstract Degeneration Theorem}

We now prove the main theorem of this section: 
\begin{proof}[Proof that Theorem \ref{thmcks2} $\plus$ Theorem \ref{thmcks} $\Rightarrow$ Theorem \ref{main}]


First, since $\Tilde{\mathcal{F}}^{\bullet}$ is the corresponding filtration of $\mathcal{F}^{\bullet}$ under canonical extension, by Deligne's construction, this means $\text{exp}(-zN)\varphi^{*}\mathcal{F}^{\bullet}$ viewed as the filtration on $\Tilde{\mathcal{H}}|_{\Delta*}$ coincides with $\Tilde{\mathcal{F}}^{\bullet}|_{\Delta*}$. 

To prove Theorem \ref{main} (1), it suffices to show that  $(\varphi^{*}\mathcal{F}^{k})_z\oplus \overline{(\varphi^{*}\mathcal{F}^{d-k+1})_z}=(\varphi^{*}\mathcal{H})_z$ for all $0\leq k\leq d$ and $\textup{Im}(z)$ sufficiently large. We may assume $0\leq \textup{Re}(z)\leq 1$ and write $z=a+it$. We choose a basis $\{v_i^k\}_{i\in I^k}$ of $\Tilde{\mathcal{F}}_{0}^{k}$ for every $0\leq k\leq d$, where $I^k$ is a finite indexed set for the basis. After possibly shrinking $\Delta$, we can extend $v_i^k$ to a holomorphic section $v_i^k(q)$ of  $\Tilde{\mathcal{F}}^{k}$ so that $\{v_i^k(q)\}_{i\in I^k}$ is a basis of $\Tilde{\mathcal{F}}_{q}^{k}.$ Then $(\varphi^{*}\mathcal{F}^{k})_z$ has a basis of the form $$u_i^k(z)=e^{zN}v_i^k(q).$$ We have $|q|=O(e^{-2\pi t})$ and $|z|=O(t)$ since $0\leq \textup{Re}(z)\leq 1$. And since $N^{d+1}=0$, we have $$e^{zN}v_i^k(q)=e^{zN}(v_i^k+O(e^{-2\pi t}))=e^{zN}v_i^k+O(t^de^{-2\pi t}).$$ Therefore, the top wedge of the basis of $(\varphi^{*}\mathcal{F}^{k})_z$ and $ \overline{(\varphi^{*}\mathcal{F}^{d-k+1})_z}$ would be $$\Phi^k:=\left(\bigwedge_{i\in I^k}u_i^k(z) \right)\wedge\left( \bigwedge_{j\in I^{d-k+1}} \overline{u_j^{d-k+1}(z)}\right)$$
$$=\left(\bigwedge_{i\in I^k} e^{zN}v_i^k+O(t^de^{-2\pi t}) \right)\wedge\left( \bigwedge_{j\in I^{d-k+1}} e^{\overline{z}N}\overline{v_j^{d-k+1}}+O(t^de^{-2\pi t})\right)$$ $$=\left(\bigwedge_{i\in I^k} e^{zN}v_i^k\right)\wedge\left( \bigwedge_{j\in I^{d-k+1}} e^{\overline{z}N}\overline{v_j^{d-k+1}}\right)+O\left(t^{d(\text{dim}_{\C}H)}e^{-2\pi t}\right)$$

By our assumption, $(\Tilde{\mathcal{H}}_0, W(N,d),\tilde{\mathcal{F}}_0)$ is a mixed Hodge structure. So by Theorem \ref{thmcks2}, $\text{exp}(zN)\tilde{\mathcal{F}}_0^{\bullet}$ is $d$-opposed to its complex conjugate for $\text{Im}(z)$ sufficiently large. Therefore, the leading term $$\left(\bigwedge_{i\in I^k} e^{zN}v_i^k\right)\wedge\left( \bigwedge_{j\in I^{d-k+1}} e^{\overline{z}N}\overline{v_j^{d-k+1}}\right)$$ is a nonvanishing top form for $t$ sufficiently large and we can rewrite it as $$C_k(a,t)\left(\bigwedge_{i\in I^0} v_i^0\right)$$
where $C_k(a,t)$ is a polynomial in $a$ and $t$ and is nonzero for $t$ sufficiently large. Then we must have $|C_k(a,t)|\geq C_kt^{L_k}$ for some $C_k> 0$ and nonnegative integer $L_k$, for $t$ sufficiently large. So we can conclude that $\Phi^k$ is nonvanishing for $t$ sufficiently large, and hence $(\varphi^{*}\mathcal{F}^{k})_z$ is $d$-opposed to its complex conjugate for $t$ large.

To prove Theorem \ref{main} (2), we need to show that $S$ has the corresponding signature on $(\varphi^{*}\mathcal{F}^{k})_z \bigcap \overline{(\varphi^{*}\mathcal{F}^{d-k})_z}$. Due to the $d$-opposedness condition that we have just proved and the first Hodge-Riemann bilinear relation, it is equivalent to show that for $0\leq k\leq d,$ the signature of the Hermitian form  $(\sqrt{-1})^d S(\bullet,\overline{\bullet})$ on $(\varphi^{*}\mathcal{F}^k)_z$ is 
\begin{equation}\label{eq1}
\mathlarger{\mathlarger{\sum}}_{j=k}^{d}(-1)^{d-j}\left(\sum_{\substack{l\geq j-d, \\ r\geq(-l)_{+}} }s_{+}^{j+r,d+l-j+r}-\sum_{\substack{l\geq j-d, \\ r\geq(-l)_{+}} }s_{-}^{j+r,d+l-j+r}\right)
\end{equation} for $t$ sufficiently large. Here, by abuse of terminology, given a non-degenerate bilinear form $S$ (or Hermitian form) with signature $(s_{+},s_{-})$, we call the number $s_{+}-s_{-}$ the signature of $S$ as well.

 Replace the vectors $v_i^k$ with the basis compatible with the refined filtration from Theorem \ref{thmcks2}.  We have $$(\sqrt{-1})^dS(u_i^k(z),\overline{u_{i'}^k(z)})$$
$$= (\sqrt{-1})^d S(e^{zN}v_i^k+O(t^de^{-2\pi t}),\overline{e^{zN}v_{i'}^k}+O(t^de^{-2\pi t}))$$
$$=(\sqrt{-1})^d S(e^{zN}v_i^k,e^{\overline{z}N}\overline{v_{i'}^k})+O(t^{2d}e^{-2\pi t}).$$ The last equality follows from $e^{zN}v_i^k=O(t^d)$ and $\overline{e^{zN}v_{i'}^k}=O(t^d)$, which follows from  $N^{d+1}=0$ and $|z|=O(t)$.

Since the mixed Hodge structure $(\Tilde{\mathcal{H}}_0, W(N,d),\tilde{\mathcal{F}}_0)$ and $S(\bullet,\bullet)$ satisfy all of the assumptions in Theorem \ref{thmcks}, the Hermitian form  $(\sqrt{-1})^d S(\bullet,\overline{\bullet})$ has signature as in (\ref{eq1}) on the filtration $\textup{exp}(zN)\tilde{\mathcal{F}}_0^k.$ Thus the leading term $$\left((\sqrt{-1})^d S(e^{zN}v_i^k,e^{\overline{z}N}\overline{v_{i'}^k})\right)_{i,i'}$$ of the matrix $\left((\sqrt{-1})^dS(u_i^k(z),\overline{u_{i'}^k(z)})\right)_{i,i'}$ is a nondegenerate Hermitian matrix with signature the same as (\ref{eq1}).

We recall the following easy linear algebra fact:
\begin{lemma} \label{linear}
    Suppose that $A$ is a non-degenerate Hermitian matrix such that the leading principal $l\times l$ submatrices are non-degenerate for all $l$. Then $A$ is $*$-congruent to the diagonal matrix $D=\textup{diag}(M_0, M_1/M_0, M_2/M_1,...)$,  where $M_l$ is the leading principal $l\times l$ minor of $A$. (In other words, there exists an invertible matrix $P$ such that $A=P^{*}DP.$)
\end{lemma}
\begin{proof}
   Given a Hermitian matrix $A=\left(\begin{array}{cc}
   G  & B \\
    B^* & C
\end{array}\right)$ with $G$ invertible, we can perform the block diagonalization  $$\left(\begin{array}{cc}
   G  & B \\
    B^* & C
\end{array}\right)=\left(\begin{array}{cc}
   I  & 0 \\
   B^{*}G^{-1}  & I
\end{array}\right)\left(\begin{array}{cc}
   G  & 0 \\
    0 & C-B^*G^{-1}B
\end{array}\right)\left(\begin{array}{cc}
   I  & G^{-1}B \\
    0 & I
\end{array}\right).$$ Inductively performing this operation to $A$ with respect to the leading principal submatrices, one gets the diagonal matrix $$D=\textup{diag}(M_0, M_1/M_0, M_2/M_1,...).$$
\end{proof}

If $A$ is the matrix $\left((\sqrt{-1})^d S(e^{zN}v_i^k,e^{\overline{z}N}\overline{v_{i'}^k})\right)_{i,i'}$, then by our choice of basis, the leading principal $l\times l$ submatrices are non-degenerate. In particular, if we let $P_l(a,t)$ be the leading principal $l\times l$ minor of $A$, then $P_l(a,t)$ is a polynomial in $a$ and $t$ which must be nonvanishing for $t$ sufficiently large, Therefore, we have $|P_l(a,t)|\geq C_lt^{K_l}$ for some $C_l>0$ and nonnegative integer $K_l$, for $t$ sufficiently large.

Let $B=\left((\sqrt{-1})^dS(u_i^k(z),\overline{u_{i'}^k(z)})\right)_{i,i'}$. By our computation, we can see that the upper left $l \times l$ minor $M_l$ of $B$ is of the form $$P_l(a,t)+O(t^{n_l}e^{-2\pi t}),$$ where $n_l$ is a non-negative integer. Hence, $M_l$ is nonzero for $t$ sufficiently large, and the leading principal $l\times l$ submatrices of $B$ are also non-degenerate. Now we have  $$M_l^{-1}=\frac{1}{P_l(a,t)}\left(1+\sum_{j\geq 0}\left(-P_l(a,t)O(t^{n_l}e^{-2\pi t})\right)^j\right)=\frac{1}{P_l(a,t)}+O(t^{n_l}e^{-2\pi t})$$ and $$\frac{M_{l+1}}{M_{l}}=\frac{P_{l+1}(a,t)}{P_{l}(a,t)}+O(t^{n'_l}e^{-2\pi t})$$ for some nonnegative integer $n'_l$. Therefore, $M_{l+1}/M_{l}$ must have the same sign as 
$P_{l+1}(a,t)/P_{l}(a,t)$ for $t$ sufficiently large. So we can conclude that the Hermitian matrix $\left(\sqrt{-1})^dS(u_i^k(z),\overline{u_{i'}^k(z)}\right)_{i,i'}$ also has signature as in (\ref{eq1}) for $t$ sufficiently large. The proof is complete.

\end{proof}

\begin{remark}
In section $6$, we will show that $C_k(a,t)=C_k t^{L_k}+O(t^{L_k-1})$ with $C_k\neq 0$ and $L_k\geq 0$, $P_{l+1}(a,t)/P_{l}(a,t)=C'_lt^{L'_l}+ O(t^{L'_l-1})$ with $C'_l\neq 0$ and $k-d\leq L'_l\leq d$. The asymptotic orders $L_k$ and $L'_l$ are explicitly computable (cf. Remark \ref{rmk'}).


\end{remark}

\section{Degenerations of Geometric Hodge Filtrations}

\subsection{Main Results on Geometric Hodge Filtrations}
In this section, let $f: X\rightarrow\Delta$ be a semistable degeneration of $m$-dimensional compact complex manifolds, satisfying \hyperref[WeakKahler]{$(\mathbf{K}_1)$}. 

Then, the first theorem of the section is:

\begin{theorem} \label{Corg}
Fix $d\geq0$. Suppose that the monodromy condition $$W^{St}=W(N,d)\quad \textup{on} \quad H^d(X_{\infty},\Q)$$ holds. Then the Hodge filtration $F_q^{\bullet}$ of $H^{d}(X_q,\C)$ on the nearby fiber $X_q:=f^{-1}(q)$ is $d$-opposed to its complex conjugate for $|q|$ sufficiently small.
\end{theorem}

\begin{remark}
    As we mentioned in the previous section, the condition $$W^{St}=W(N,d)\quad \textup{on} \quad H^d(X_{\infty},\Q)$$ implies the Hodge numbers symmetry on the general fiber: $$h^{p,d-p}(X_q)=h^{d-p,p}(X_q)$$ for all $p$. However, by \cite{COUV}, there exists a compact complex threefold such that the Hodge to de Rham spectral sequence degenerates at $E_1$ and that the Hodge numbers symmetry $h^{p,r}=h^{r,p}$ holds for all $p$ and $r$ but is not a $\partial\overline{\partial}$-manifold. Thus it is not obvious that $d$-opposedness condition on the Hodge filtration of the general fiber follows from the monodromy weight condition.
\end{remark}

 Now, we focus on the case $d=m$ and assume that the monodromy weight condition $$W^{St}=W(N,m)\quad \textup{on} \quad  H^m(X_{\infty},\Q)$$ holds. 
 
 Let $Q(\bullet,\bullet)$ denote the natural cup product on $H^m(X_{\infty},\C)$ defined in Section \ref{cupproduct}. Write $S(\bullet,\bullet)=\epsilon(m)Q(\bullet,\bullet)$, where $\epsilon(a):=(-1)^{\frac{a(a-1)}{2}}$. When no confusion is likely, we also use $Q(\bullet,\bullet)$ to denote the natural cup-product pairing on any compact complex manifold and define $S(\bullet,\bullet)$ by the same formula. Let $C$ denote the Weil operator, which is defined as $(\sqrt{-1})^{p-r}$ on the $(p,r)$-component of a Hodge structure of weight $p+r$. Since $$W^{St}=W(N,m)\quad \textup{on} \quad  H^m(X_{\infty},\Q),$$ according to the discussion of the situation (\textbf{B}$'$) in the previous section, the Hermitian form $$(u,v)\mapsto S(Cu,(-N)^{l}\ \overline{v})$$ is nondegenerate on $\textup{Gr}^{W^{St}}_{m+l}H^m(X_{\infty},\C)$ for all $l\geq 0.$ See Remark \ref{neg} for why we use $-N$ instead of $N$. Here, we recall that $\textup{Gr}^{W^{St}}_{m+l}H^m(X_{\infty},\C)$ is isomorphic to the cohomology at the middle term of the complex $$  \bigoplus_{k\geq 0,-l-1} H^{m-l-2k-2}(E(2k+l+2))\xrightarrow{-\gamma+\theta}\bigoplus_{k\geq 0,-l} H^{m-l-2k}(E(2k+l+1))$$$$\xrightarrow{-\gamma+\theta}  \bigoplus_{k\geq 0,-l+1} H^{m-l-2k+2}(E(2k+l)),$$ and the cup-product pairing $Q(\bullet,\bullet)$ is induced by the pairing $\psi$ in Proposition \ref{Fu2}.
 
 Let $(s^{p,m+l-p}_{+},s^{p,m+l-p}_{-})$ be the signature of this Hermitian form on the $(p,m+l-p)$-component of the Hodge structure (of weight $d+l$) on the primitive part $$P_{m+l} = \text{ker}\left(N^{l+1}:\textup{Gr}^{W^{St}}_{m+l}H^m(X_{\infty},\C)\rightarrow \textup{Gr}^{W^{St}}_{m-l-2}H^m(X_{\infty},\C)\right).$$ 

Given an integer $n$, set $n_{+}=\textup{max}\{n,0\}$. Since $$W^{St}=W(N,m)\quad \textup{on} \quad  H^m(X_{\infty},\Q),$$ for all $l$ with $-m\leq l \leq m$, there is an $N$-Lefschetz decomposition: $$\textup{Gr}^{W^{St}}_{m+l}H^m(X_{\infty},\C)=\bigoplus_{r\geq (-l)_{+}} N^r P_{m+l+2r}.$$ Thus for $0\leq p\leq m+l$, define the invariants $$\textbf{S}_{+}^{p,m+l-p}=\sum_{r\geq (-l)_{+}}s_{+}^{p+r,m+l-p+r}$$ and $$\textbf{S}_{-}^{p,m+l-p}=\sum_{r\geq (-l)_{+}}s_{-}^{p+r,m+l-p+r},$$ which corresponds to (the $N$-Lefschetz decomposition of) the $(p,m+l-p)$-component of the Hodge structure on $\textup{Gr}^{W^{St}}_{m+l}H^m(X_{\infty},\C)$.

Then, the second theorem of the section is:

\begin{theorem} \label{Corg2}
 Suppose that $$W^{St}=W(N,m)\quad \textup{on} \quad  H^m(X_{\infty},\Q)$$ holds. Then the signature of $S(C\bullet, \overline{\bullet})$ on the $(p,m-p)$-component of the Hodge structure on $H^m(X_q,\C)$ is
$$\left(\sum_{ k=0}^m \textup{\textbf{S}}_{+}^{p,k},\sum_{ k=0}^m \textup{\textbf{S}}_{-}^{p,k}\right)$$ for $|q|$ sufficiently small.

In particular, if $(H^m(X_{\infty},\C), W^{St}, F_{\lim})$ is a polarized mixed Hodge structure with respect to the cup product $S(\bullet,\bullet)$ (i.e., $s^{p,r}_{-}=0$ for all $p$ and $r$), then the middle cohomology $H^m(X_q,\C)$ together with the cup product $S(\bullet,\bullet)$ is a polarized Hodge structure for $|q|$ sufficiently small.

\end{theorem}

\subsection{Proof of the Main Geometric Results}

\begin{proof}[Proof of Theorem \ref{Corg} and Theorem \ref{Corg2}]
By Theorem \ref{St}, there are holomorphic subbundles $\Tilde{\mathcal{F}}^{\bullet}$ of $\Tilde{\mathcal{H}}^{d}$ that correspond to the Hodge subbundles $\mathcal{F}^{\bullet}$ of $\mathcal{H}^{d}$ under Deligne's canonical extension.

Taking $ S(\bullet,\bullet)=\epsilon(m)Q(\bullet,\bullet),$ it therefore suffices, under the respective assumptions of Theorems \ref{Corg} and \ref{Corg2}, to verify the conditions in situations $(\mathbf{A}')$ and $(\mathbf{B}')$, together with the stronger form of the first Hodge--Riemann bilinear relation $$ S\left(\mathcal{F}^p,\mathcal{F}^{d-p+1}\right)=0 $$ for every $p$. Once these conditions are verified, the conclusions follow from Theorem \ref{main}. 

Condition (a) in situation $(\mathbf{A}')$ follows from Theorem \ref{St}. Conditions (a) and (b) in situation $(\mathbf{B}')$ follow from Proposition \ref{Fu}, while the stronger form of the first Hodge--Riemann bilinear relation follows from Lemma \ref{rmkg}. This completes the proof. \end{proof}

\begin{proof}[Proof of Theorem \ref{Thmg}]
    This is a reformulation of Theorems~\ref{Corg} and~\ref{Corg2} in light of Theorem~\ref{Zeinthm}.
\end{proof}

 As an immediate consequence, if $f: X\rightarrow\Delta$ satisfies \hyperref[StrongKahler]{$(\mathbf{K}_2)$}, then the general fibers are  $\partial\overline{\partial}$-manifolds, and the Hodge indices on the middle cohomology of general fibers are the same as those of compact K\"ahler manifolds: 

\subsection{Degenerations of Strong K\"ahler Type}
\begin{corollary}\label{kah}
 Let $f: X\rightarrow\Delta$ be a semistable degeneration of $m$-dimensional compact complex manifolds satisfying \hyperref[StrongKahler]{$(\mathbf{K}_2)$}. Then the general fiber $X_q$ is a $\partial\overline{\partial}$-manifold for $|q|$ sufficiently small. Moreover, if we let $$h^{a,b}_{q}=\textup{dim}_{\C}H^{b}(X_q,\Omega_{X_q}^a),$$ then the signature of $S(C\bullet,\overline{\bullet})$ on the $(p,m-p)$-component of the Hodge structure on $H^m(X_q,\C)$ is $$\left(\sum_{r\geq 0}(h^{p-2r,m-p-2r}_q-h^{p-2r-1,m-p-2r-1}_q) ,\sum_{r\geq 0}(h^{p-2r-1,m-p-2r-1}_q-h^{p-2r-2,m-p-2r-2}_q)\right)$$ for $|q|$ small. 
 
 In particular, when $m$ is even, the general fiber $X_q$ satisfies the Hodge index theorem for compact K\"ahler manifolds: the signature of the cup product $Q(\bullet,\bullet)$ on $H^m(X_q,\R)$ is $\sum_{a,b\geq 0}(-1)^{a}h^{a,b}_q$ for $|q|$ small.
\end{corollary}
\begin{proof}
By Theorem \ref{GNTheorem}, under the assumption \hyperref[StrongKahler]{$(\mathbf{K}_2)$},$$W^{St}=W(N,d) \quad \textup{on}\quad H^d(X_{\infty},\Q)$$ for all $d\geq 0$. Hence, by Theorem \ref{Corg}, $X_q$ is a $\partial\overline{\partial}$-manifold for $|q|$ small. 

On the other hand, Proposition \ref{Fu2} asserts that the cup product on $H^m(X_{\infty})$ restricts to the pairing $\psi$ on graded pieces, which coincides with the one used in \cite{guillen} to define the polarized Hodge-Lefchetz module structure on ${}_{W^{St}}E^{-r,d+r}_1\otimes \R$. So, we can compute the Hodge index via the Hodge-Lefchetz decomposition. To be precise, let $L\in H^2(X_0,\R)$ be a cohomology class that restricts to a K\"ahler class on each irreducible component $E_i$. Cup product with the restriction of $L$ to the appropriate intersection of components of $X_0$ define mappings $$\textup{Gr}^{W^{St}}_{r}H^d(X_{\infty},\C)\rightarrow \textup{Gr}^{W^{St}}_{r+2}H^{d+2}(X_{\infty},\C),$$ still denoted by $L.$ For $p+s\geq d$, let $P^{p,s}_d$ be the $(p,s)$-component of the Hodge structure of the $N$-primitive space $$\textup{ker}\left(N^{p+s-d+1}:\textup{Gr}^{W^{St}}_{p+s}H^d(X_{\infty},\C)\rightarrow \textup{Gr}^{W^{St}}_{2d-p-s-2}H^d(X_{\infty},\C)\right),$$ and for $d\leq m$, we define the $L$-primitive subspace of $P^{p,s}_d$ by $$(P^{p,s}_d)_0:=P^{p,s}_d\cap \textup{ker}(L^{m-d+1}).$$ Then $P^{p,s}_m$ has a Lefschetz decomposition: $$P^{p,s}_m=\bigoplus_{r\geq 0}L^r(P^{p-r,s-r}_{m-2r})_0.$$ We claim that the Hermitian form $S(C\bullet,(-N)^{p+s-m}\overline{\bullet})$ on $L^r(P^{p-r,s-r}_{m-2r})_0$ is positive definite if $r$ is even and is negative definite if $r$ is odd. To see this, note that $$S(\bullet,\bullet)=\epsilon(m)Q(\bullet,\bullet)=\epsilon(m)\epsilon(-m)\langle\bullet,\bullet\rangle=(-1)^m\langle\bullet,\bullet\rangle$$ on $H^m(X_{\infty,},\C).$ Let $x\in (P^{p-r,s-r}_{m-2r})_0$, and then we have $$(-1)^r S(CL^rx,(-N)^{p+s-m}L^r\overline{x})$$ $$=(-1)^r (-1)^m \langle CL^rx,(-N)^{p+s-m}L^r\overline{x}\rangle$$ $$=(-1)^r (-1)^m (2\pi \sqrt{-1})^m \psi(CL^rx,(-N)^{p+s-m}L^r\overline{x})$$  $$=(-1)^r (-1)^m (2\pi \sqrt{-1})^m (-1)^{p+s}\psi(L^rx,C(-N)^{p+s-m}L^r\overline{x})$$
$$=(-1)^r (-1)^m (2\pi \sqrt{-1})^m (-1)^{p+s} (-1)^r\psi(x,L^rC(-N)^{p+s-m}L^r\overline{x})$$
$$=(2\pi\sqrt{-1})^m\psi(x,CN^{p+s-m}L^{2r}\overline{x})> 0.$$ The last inequality follows from the definition of polarized Hodge-Lefschetz module in \cite[(4.3)]{guillen}. 

As a result, the signature of $S(C\bullet,(-N)^{p+s-m}\overline{\bullet})$ on $P^{p,s}_m$ is $$\left(\sum_{r\geq 0}\textup{dim}_{\C}(P^{p-2r,s-2r}_{m-4r})_0,\sum_{r\geq 0}\textup{dim}_{\C}(P^{p-2r-1,s-2r-1}_{m-4r-2})_0\right).$$ Then by Theorem \ref{Corg}, for $|q|$ small, the signature of $S(C\bullet,\overline{\bullet})$ on the $(p,m-p)$-component of the Hodge structure on $H^m(X_q,\C)$ is $$\left(\sum_{\substack{l\geq p-m, \\ r\geq(-l)_{+}} }\sum_{s\geq 0}\textup{dim}_{\C}(P^{p+r-2s,m+l-p+r-2s}_{m-4s})_0, \ \sum_{\substack{l\geq p-m, \\ r\geq(-l)_{+}} }\sum_{s\geq 0}\textup{dim}_{\C}(P^{p+r-2s-1,m+l-p+r-2s-1}_{m-4s-2})_0\right).$$ Since for all $a\geq 0$ and $l\geq p-m$, we have the decomposition $$\textup{Gr}_{F}^{p-a}Gr^{W^{St}}_{m+l-2a}H^{m-2a}(X_{\infty},\C)=\bigoplus_{r\geq (-l)_{+}}N^rP^{p+r-a,m+l-p+r-a}_{m-2a},$$ we can see that $$\sum_{\substack{l\geq p-m, \\ r\geq(-l)_{+}} }\textup{dim}_{\C}(P^{p+r-a,m+l-p+r-a}_{m-2a})_0$$
$$=\sum_{l\geq p-m}\textup{dim}_{\C}\textup{Gr}_{F}^{p-a}\left( Gr^{W^{St}}_{m+l-2a}H^{m-2a}(X_{\infty},\C)\right)_0$$
$$=\textup{dim}_{\C}\textup{Gr}_{F}^{p-a}H^{m-2a}(X_{\infty},\C)-\textup{dim}_{\C}\textup{Gr}_{F}^{p-a-1}H^{m-2a-2}(X_{\infty},\C)$$
$$=h^{p-a,m-p-a}_{q}-h_q^{p-a-1,m-p-a-1}$$ for $|q|$ small. Hence, by summing, the signature of $S(C\bullet,\overline{\bullet})$ on the $(p,m-p)$-component of the Hodge structure on $H^m(X_q,\C)$ is $$\left(\sum_{r\geq 0}(h^{p-2r,m-p-2r}_q-h^{p-2r-1,m-p-2r-1}_q) ,\sum_{r\geq 0}(h^{p-2r-1,m-p-2r-1}_q-h^{p-2r-2,m-p-2r-2}_q)\right)$$ for $|q|$ small. 
The last statement in the theorem is a standard consequence of the signature $S(C\bullet,\overline{\bullet})$ on the middle cohomology we just computed. The only thing to notice is that, although we don't have the Hard Lefschetz theorem on the general fiber $X_q$, we still have $h^{p,r}_q=h^{m-p,m-r}_q$ for $|q|$ small since $$h^{p,r}_q=\textup{dim}_{\C}\textup{Gr}_{F}^{p}H^{p+r}(X_{\infty},\C)=\textup{dim}_{\C}\textup{Gr}_{F}^{m-p}H^{2m-p-r}(X_{\infty},\C)=h^{m-p,m-r}_q$$ for $|q|$ small by the existence of $L$ on the central fiber.

\end{proof}

\subsection{Criteria for Degenerations of Threefolds without Triple Points}

 When $X_0$ is a normal crossing threefold without triple points, by combining Theorem \ref{Corg}, Theorem \ref{criND} (or simply Sano \cite[Proposition 2.6 (iii)]{SA3}), and the Hodge index theorem, we obtain several interesting criteria for $H^3(X_t)$ to have the Hodge decomposition:

 \begin{corollary} \label{cri3}
 Let $f: X\rightarrow \Delta$ be a semistable degeneration of $3$-dimensional compact complex manifolds. We assume that the central fiber $X_0$ has no triple points. Then $H^3(X_t)$ admits the canonical Hodge decomposition for $|t|$ small if any of the following conditions holds:
\begin{enumerate}
    \item The dual graph of $X_0$ is a star (i.e., a tree with only one internal vertex), the threefold corresponding to the internal vertex is a $\partial\overline{\partial}$-manifold, and the threefolds corresponding to the leaves are K\"ahler.
    \item The dual graph of $X_0$ is a tree, each irreducible component of $X_0$ is a $\partial\overline{\partial}$-manifold, and for each irreducible component $S$ of the double locus, we have $K_S^2 > 0$, where $K_S$ is the canonical bundle of $S$.
    \item Each irreducible component $E_i$ of $X_0$ is a $\partial\overline{\partial}$-manifold, and for each irreducible component $S\subset E_i\cap E_j$ for some $i$ and $j$, we have $N_{S/E_i}^2 > 0$ (or equivalently, $N_{S/E_j}^2 > 0$), where $N_{S/E_i}$ is the normal bundle of $S$ in $E_i$.
\end{enumerate}
 \end{corollary}

\begin{proof}
By Theorem \ref{Corg} and Theorem \ref{criND}, it suffices to verify that the conditions in Theorem \ref{criND} (2) hold for $m=3$. Since $X_0$ has no triple points, $E(l)=\emptyset$ for $l\geq 3$ Thus, the conditions in Theorem \ref{criND} (2) reduce to showing that $$V:=\textup{im}\left(H^2(E(1))\xrightarrow{\theta_1} H^2(E(2))\right)$$ is a non-degenerate subspace. Let $s$ be the number of irreducible components of $E(2)$. By assumption, each irreducible component of $E(2)$ is K\"ahler in all three cases above. Thus, by the Hodge index theorem, if $V$ contains $s$ orthogonal vectors with positive square, then $V$ must be a non-degenerate subspace of $H^2(E(2))$. We can check that this holds in the cases (1), (2) and (3) above. For (1), the restriction of the $s$ K\"ahler classes in $s$ threefolds corresponding to the leaves gives the $s$ orthogonal vectors in $H^2(E(2))$ with positive square.  For (3), the cohomology classes of $s$ normal bundles corresponding to $s$ components of $E(2)$ are $s$orthogonal vectors in $H^2(E(2))$ with positive square, where we use the fact that $[N_{S/E_i}]=[S]|_S$ is contained in $V$. For (2), consider the restriction of the dualizing sheaf $\omega_{X_0}$ to each irreducible component of $X_0$, and the corresponding restriction to $E(2)$. If $E$ is an irreducible component of $X_0$ and $S_1$, $S_2$,..., $S_l$ are irreducible components of the double locus, then $\omega_{X_0}|_{E}=K_{E}+S_1+...+S_l$, and $\omega_{X_0}|_{S_i}=K_{S_i}$. Since $K_{S_i}^2>0$ by assumption, $\theta_1(\omega_{X_0}|_E)$ is a vector with positive square in $V$. To produce $s$ orthogonal vectors with positive square, we use the assumption that the dual graph of $X_0$ is tree. If $S$ an irreducible components contained in a threefold corresponding to a leaf of the tree, then $K_S\in V$ since a leaf is a degree $1$ vertex. Then we can use the induction on the farthest distance to a leaf to see that the $s$ canonical classes of irreducible components of $E(2)$ are in $V$. This completes the proof.

\end{proof}
 
\begin{remark}
 In Proposition \ref{kodaira2}, we give a nonexample $X_0$ whose dual graph is a circle and which violates the condition condition (2) in Theorem \ref{criND}. There is also a nonexample with four irreducible components whose dual graph is a tree. See \cite[Section 5.5]{thesis} for this example.
\end{remark}

\section{Examples}

Throughout the section, for any $m$-dimensional compact complex manifold $M$, we write $Q_M(\bullet,\bullet)$, or simply $Q(\bullet,\bullet)$ when there is no ambiguity, for the cup-product pairing on $H^m(M)$. We write $S(\bullet,\bullet)$ for the pairing $\epsilon(m)Q(\bullet,\bullet)$. If $H^m(M)$ admits a canonical Hodge decomposition, we denote by $C$ the Weil operator, which acts as multiplication by $(\sqrt{-1})^{p-q}$ on the $(p,q)$-component of $H^m(M)$, and use $S(C\bullet,\overline{\bullet})$ to denote the associated Hermitian pairing.

Before proceeding to the examples, we recall several classical facts concerning the Gysin and restriction maps that will be used in the computation of the monodromy weight condition.

\begin{proposition}[{\cite[(1.5), (1.6)]{Voi2}}] \label{tool1}
Let $X$ be a smooth compact complex manifold and let $Y\subset X$ be a smooth hypersurface. Let
$$
\theta:H^k(X)\rightarrow H^k(Y) \quad \textup{and} \quad \gamma:H^k(Y)\rightarrow H^{k+2}(X)
$$
denote the restriction map and the Gysin map, respectively. Then the following hold.
\begin{enumerate}
    \item The composition
    $$
    H^k(X)\xrightarrow{\theta}H^k(Y)\xrightarrow{\gamma}H^{k+2}(X)
    $$
    coincides with cup product with the cohomology class $[Y]\in H^2(X)$.

    \item The composition
    $$
    H^k(Y)\xrightarrow{\gamma}H^{k+2}(X)\xrightarrow{\theta}H^{k+2}(Y)
    $$
    coincides with cup product with the first Chern class of the normal bundle $N_{Y/X}$.
\end{enumerate}
\end{proposition}

\begin{proposition}[{\cite[Theorem 7.31]{Voi1},\cite[(2.7)]{Voi2}}]\label{tool2}
Let $X$ be an $m$-dimensional compact complex manifold, and let $B\subset X$ be a compact complex submanifold of codimension $2$. Let
$
\widetilde{X}:=\operatorname{Bl}_B X
$
be the blowup of $X$ along $B$. Then the following hold.
\begin{enumerate}
    \item For every $k$, there is an isomorphism
    $$
    H^k(\widetilde{X})
    \cong
    H^k(X)\oplus H^{k-2}(B).
    $$

    \item Suppose that there are inclusions
    $$
    B\xrightarrow{\iota_2}Y\xrightarrow{\iota_1}X,
    $$
    where $B$, $Y$, and $X$ are compact complex manifolds of dimensions $m-2$, $m-1$, and $m$, respectively. Let
    $
    \widetilde{Y}:=\operatorname{Bl}_B Y.
    $
    Then $\widetilde{Y}$ naturally embeds into $\widetilde{X}$ via
    $
    \widetilde{\iota}_1:\widetilde{Y}\hookrightarrow\widetilde{X}.
    $
    Under the blowup decompositions
    $
    H^k(\widetilde{X})
    \cong
    H^k(X)\oplus H^{k-2}(B)
    $
    and
    $
    H^k(\widetilde{Y})
    \cong
    H^k(Y),
    $
    the restriction map
    $$
    H^k(\widetilde{X})
    \xrightarrow{\widetilde{\iota}_1^*}
    H^k(\widetilde{Y})
    $$
    can be identified with
    $$
    H^k(X)\oplus H^{k-2}(B)
    \xrightarrow{(\iota_1^*,\,(\iota_2)_!)}
    H^k(Y),
    $$
    where $\iota_1^*$ is the restriction map and $(\iota_2)_!$ is the Gysin morphism.

    Dually, under the same identifications, the Gysin morphism
    $$
    H^k(\widetilde{Y})
    \xrightarrow{(\widetilde{\iota}_1)_!}
    H^{k+2}(\widetilde{X})
    $$
    can be identified with
    $$
    H^k(Y)
    \xrightarrow{((\iota_1)_!,\,-\iota_2^*)}
    H^{k+2}(X)\oplus H^k(B),
    $$
    where $(\iota_1)_!$ is the Gysin morphism and $\iota_2^*$ is the restriction map.

    \item If $k=m$, then, under the decomposition
    $$
    H^m(\widetilde{X})
    \cong
    H^m(X)\oplus H^{m-2}(B),
    $$
    the pairing
    $$
    \epsilon(m)Q_{\widetilde{X}}(\bullet,\bullet)
    $$
    is the orthogonal direct sum of
    $$
    \epsilon(m)Q_X(\bullet,\bullet)
    \qquad\textup{and}\qquad
    \epsilon(m-2)Q_B(\bullet,\bullet).
    $$
\end{enumerate}
\end{proposition}
\begin{proposition}[{\cite[Proposition 3.20]{Zein}}] \label{tool3}
Let $f:X\rightarrow\Delta$ be a semistable degeneration of $m$-dimensional compact complex manifolds satisfying \hyperref[WeakKahler]{$(\mathbf{K}_1)$}. Suppose moreover that the central fiber $X_0=E$ has no triple intersections, i.e. $E(3)=\emptyset$. Then the following hold.
\begin{enumerate}
    \item The space
    $$
    \operatorname{Gr}^{W^{St}}_{k+1}H^k(X_{\infty})
    =
    \operatorname{ker}\left(
    H^{k+1}(E(2))
    \xrightarrow{-\gamma}
    H^{k}(E(1))
    \right)
    $$
    is Poincar\'e dual to
    $$
    \operatorname{Gr}^{W^{St}}_{2m-k-1}H^{2m-k}(X_{\infty})
    =
    \operatorname{coker}\left(
    H^{2m-k}(E(1))
    \xrightarrow{\theta}
    H^{2m-k}(E(2))
    \right).
    $$

       More precisely, the Gysin and restriction morphisms are adjoint with respect to the Poincar\'e pairings: for
    $$
    \alpha\in H^{k-1}(E(2))
    \qquad\textup{and}\qquad
    \beta\in H^{2m-k-1}(E(1)),
    $$
    we have
    $$
    Q_{E(1)}(\gamma\alpha,\beta)
    =
    Q_{E(2)}(\alpha,\theta\beta).
    $$
    Consequently,
    $$
    \operatorname{im}(\gamma)
    =
    \operatorname{ker}(\theta)^\perp
    \qquad\textup{and}\qquad
    \operatorname{ker}(\gamma)
    =
    \operatorname{im}(\theta)^\perp,
    $$
    under the corresponding Poincar\'e duality identifications. In particular, the Poincar\'e pairing descends to a perfect pairing between
    $$
    \operatorname{Gr}^{W^{St}}_{k+1}H^k(X_{\infty})
    \qquad\textup{and}\qquad
    \operatorname{Gr}^{W^{St}}_{2m-k-1}H^{2m-k}(X_{\infty}).
    $$

    \item In particular,
    $$
    W^{St}=W(N,k)
    \quad\textup{on}\quad
    H^k(X_{\infty})
    $$
    if and only if
    $$
    W^{St}=W(N,2m-k)
    \quad\textup{on}\quad
    H^{2m-k}(X_{\infty}).
    $$
    Moreover, when $k=m$, the equality
    $$
    W^{St}=W(N,m)
    \quad\textup{on}\quad
    H^m(X_{\infty})
    $$
    holds if and only if
    $$
    N:
    \operatorname{Gr}^{W^{St}}_{m+1}H^m(X_{\infty})
    \longrightarrow
    \operatorname{Gr}^{W^{St}}_{m-1}H^m(X_{\infty})
    $$
    is injective.
\end{enumerate}
\end{proposition}

\subsection{ A non-example for degeneration of non-K\"ahler surfaces} \label{kodairathm}

In \cite{kodaira}, Kodaira constructs an example of the degeneration of Hopf surfaces into a rationally ruled surface with an ordinary double curve. More precisely, for every $m\geq 1$, there exists a one-parameter complex analytic family $f:X\rightarrow \Delta$, while $X_0$ is obtained from the Hirzebruch surface $F_m$ by identifying its zero and infinity sections so that their common image is an ordinary double curve. Since $H^1(X_t,\Q)$ is one-dimensional, it cannot carry a pure Hodge structure of weight one. On the other hand, we can prove:

\begin{proposition} \label{kodaira}
    Let $\tilde{f}:\widetilde{Y}\rightarrow \Delta$ be the semistable reduction of Kodaira's degeneration of Hopf surfaces obtained by blowing up $X$ along the double curve and then performing the quadratic base change. Then the following statements hold:  
    \begin{enumerate}
        \item The Steenbrink weight filtration and the monodromy weight filtration are distinct on $H^1(\widetilde{Y}_\infty,\Q)$: $$W^{St} \neq W(N,1)\quad \textup{on}\quad H^1(\widetilde{Y}_\infty,\Q).$$
        \item  The local invariant cycle property \hyperref[localinvariantcycle]{$(\mathbf{C}_1)$} holds, whereas \hyperref[localinvariantcycle]{$(\mathbf{C}_3)$} fails.
    \end{enumerate}
\end{proposition}

Thus, this degeneration shows that the conclusion of Theorem~\ref{Corg}~(1) may fail if the monodromy weight conditions, or equivalently the local invariant cycle properties, are not imposed in every degree.

\begin{proof}
We first describe a semistable model of the family $X\rightarrow\Delta$. Let $Y$ be the blowup of $X$ along the ordinary double curve, and consider the quadratic base change $\Delta\rightarrow\Delta$ given by $q\mapsto q^2$. Let $\widetilde{Y}\rightarrow\Delta$ denote the resulting pullback family. Then $\widetilde{Y}_0$ is a reduced simple normal crossing surface, which can be described explicitly as follows.

Let $S_1$ and $S_2$ be two copies of the Hirzebruch surface $F_m$, and let $(D_0)_i$ and $(D_\infty)_i$ denote the zero and infinity sections of $S_i$, respectively, for $i=1,2$. Their self-intersection numbers are $(D_0)_i^2=-m$ and $(D_{\infty})^2_i=m$. Then $$\Tilde{Y}_0\cong S_1 \amalg  S_2/ \sim,$$ where the equivalence relation is generated by identifying $(D_0)_1$ with $(D_\infty)_2$ and $(D_0)_2$ with $(D_\infty)_1$ via suitable automorphisms of $\mathbb{P}^1$. We denote the images of these two pairs of identified curves in $\widetilde{Y}_0$ by $D_1=(D_0)_1=(D_{\infty})_2$ and $D_2=(D_0)_2=(D_{\infty})_1$ respectively.

Now we can compute Steenbrink's limiting mixed Hodge structure associated to the family $\tilde{Y}\rightarrow \Delta $ using Proposition \ref{criterionmonodromy}, Note that ${}_{W^{St}}E_1^{-r,d+r}=0$ for $|r|\geq 2$ and any $d$. Also, we have $$Gr^{W^{St}}_{0}H^1(\widetilde{Y}_\infty,\Q)=\textup{coker}\left(H^0(S_1)\oplus H^0(S_2)\xrightarrow{\theta} H^0(D_1)\oplus H^0(D_2)\right)\cong \Q,$$ $$Gr^{W^{St}}_{1}H^1(\widetilde{Y}_\infty,\Q)=\textup{ker}\left(H^1(S_1)\oplus H^1(S_2)\xrightarrow{\theta} H^1(D_1)\oplus H^1(D_2)\right)=0,$$ and $$Gr^{W^{St}}_{2}H^1(\widetilde{Y}_\infty,\Q)=\textup{ker}\left(H^0(D_1)\oplus H^0(D_2)\xrightarrow{-\gamma} H^2(S_1)\oplus H^2(S_2)\right)= 0.$$
Thus, $$W^{St} \neq W(N,1)\quad \textup{on}\quad H^1(\widetilde{Y}_\infty,\Q),$$ and (1) is proved.

To see (2), we apply \cite[Lemma 9.14]{Clemens2} (or see \cite[Proof of Theorem 11.43]{PS}), which states that \hyperref[localinvariantcycle]{$(\mathbf{C}_d)$} holds if and only if $$\textup{ker}(N)\subset W^{St}_dH^d(\widetilde{Y}_\infty,\Q).$$ Thus, \hyperref[localinvariantcycle]{$(\mathbf{C}_1)$} holds. Now, by Poincar\'e duality \cite[Proposition 3.20]{Zein}, $$Gr^{W^{St}}_{4}H^3(\widetilde{Y}_\infty,\Q)\cong \Q,$$ and hence \hyperref[localinvariantcycle]{$(\mathbf{C}_3)$} fails. 

\end{proof}

\begin{proposition} \label{kodaira2}
   Let $\tilde{f}\colon \widetilde{Y}\rightarrow \Delta$ be as in
Proposition \ref{kodaira}, and consider the induced degeneration $g:Z:=\widetilde{Y}\times\mathbb{P}^1\rightarrow \Delta$. Then
 \begin{enumerate}
     \item $\textup{dim}\textup{Gr}^{W^{St}}_{4}H^3(Z_\infty,\Q)=\textup{dim}\textup{Gr}^{W^{St}}_{2}H^3(Z_\infty,\Q).$
     \item The morphism $$\textup{Gr}^{W^{St}}_{4}H^3(Z_\infty,\Q)\xrightarrow{N}\textup{Gr}^{W^{St}}_{2}H^3(Z_\infty,\Q)$$ is the zero map. Consequently, $$W^{St} \neq W(N,3)\quad \textup{on}\quad H^3(Z_\infty,\Q).$$
 \end{enumerate}
\end{proposition}

\begin{proof}
   Part (1) follows from Poincar\'e duality
\cite[Proposition 3.20]{Zein}. We prove (2).

By Proposition \ref{criterionmonodromy} and Proposition \ref{criND}, it suffices to prove that $$I:=\textup{im}\left(H^2(S_1\times\mathbb{P}^1)\oplus H^2(S_2\times\mathbb{P}^1)\xrightarrow{\theta}H^2(D_1\times\mathbb{P}^1)\oplus H^2(D_2\times\mathbb{P}^1)\right)$$ is a codimension $1$ degenerate subspace, and then the morphism in Proposition \ref{kodaira2}(2) is $\mb{Q}\xrightarrow{0}\Q.$  Write $$H^2(D_i\times \mathbb{P}^1)=\mb{Q}\delta_i\oplus \mathbb{Q}h_i,$$ where $\delta_i$ and $h_i$ are the pullbacks of the classes of a point
on $D_i$ and $\mathbb{P}^1$, respectively. A direct computation of the restriction morphism $\theta$ gives $$I=\textup{Span}\{\delta_1,\delta_2,h_1-h_2\}.$$ In particular, $I$ has codimension one. With respect to the ordered basis $\{\delta_1,\delta_2,h_1-h_2\},$
the restriction of the intersection pairing to $I$ is represented by
the matrix $$\left(\begin{array}{ccc}
        0 &0 & 1 \\
        0 &0 & -1\\
        1 & -1& 0
    \end{array}\right),$$ which is a degenerate bilinear form.
  
\end{proof}

Since
$$
Z_t=\widetilde{Y}_t\times\mathbb{P}^1
$$
and $\widetilde{Y}_t$ is a Hopf surface, the Hodge filtration on
$H^3(Z_t,\mathbb{C})$ is not opposed to its complex conjugate and
therefore does not define a canonical pure Hodge decomposition. This
example shows that the hypothesis
\hyperref[localinvariantcycle]{$(\mathbf{C}_3)$} in Theorem
\ref{Thmg}(2) cannot be omitted.

\subsection{ Ordinary Double Points}

In this section, we study global smoothings of an $m$-dimensional compact complex analytic space $X_0$ with at worst ordinary double points. We assume that $X_0$ admits a resolution of singularities that is a $\partial\overline{\partial}$-manifold. Since $X_0$ has only isolated singularities, this is equivalent to requiring that every resolution of singularities of $X_0$ be a $\partial\overline{\partial}$-manifold; see, for example, \cite[Remark 5.2]{FrLa}.

Let $\left\{x_i,\middle|,1\leq i\leq l\right\}$
be the set of ordinary double points of $X_0$, and let $\widetilde{X}_0$ be the blow-up of $X_0$ at all these points. The exceptional divisor $Q_i\subset\widetilde{X}_0$ lying over $x_i$ is a smooth $(m-1)$-dimensional quadric. Suppose first that $m$ is odd. In this case, each $Q_i$ is an even-dimensional quadric and contains two families of maximal linear subspaces of dimension $\frac{m-1}{2}$. Choose one such linear subspace from each family, and denote their middle-dimensional homology classes by $A_i$ and $B_i$. By abuse of notation, we also use $A_i$ and $B_i$ to denote their Poincar\'e dual cohomology classes. We first recall the cohomology of even-dimensional quadrics $Q_i\subset\mathbb{P}^m$; see \cite{Lai74} for more details:

$$H^{r}(Q_i,\mathbb{Z})\cong\left\{
\begin{array}{cc}
  \mathbb{Z}, & r\textup{ is even},\ 0\leq r\leq 2(m-1)\textup{ and }r\neq m-1,\\
    \mathbb{Z}A_i\oplus\mathbb{Z}B_i, & r=m-1.  
\end{array}\right.$$ Let $h$ be the hyperplane class on $Q_i$. Then  $$h^{\frac{m-1}{2}}=A_i+B_i, $$ and the intersection products are given by $$
A_i^2=B_i^2=\frac{1+(-1)^{m-1}}{2},
\qquad
A_iB_i=\frac{1-(-1)^{m-1}}{2}.
$$
Therefore, $$(A_i+B_i)^2=2,\quad  (A_i+B_i)(A_i-B_i)=0, \quad \textup{and}$$$$ (A_i-B_i)^2=2(-1)^{\frac{m-1}{2}}=2\epsilon(m-1).$$
It follows that with respect to the polarization $h_i$, the primitive cohomology $$H^{m-1}_{\textup{prim}}(Q_i,\C)=\C\left(A_i-B_i\right).$$

 Define the number $$R=\textup{dim}_{\C}\textup{ker}\left(\bigoplus_{i=1}^l H^{m-1}_{\textup{prim}}(Q_i)\xrightarrow{-\gamma} H^{m+1}(\widetilde{X}_{0})\right)$$ $$=\textup{dim}_{\C}\left\{(a_1,...,a_l)\in\C^l\middle|\sum_{i=1}^la_i(A_i-B_i)=0 \quad\textup{in}\quad H_{m-1}(\widetilde{X}_0)\right\}$$ where $\gamma$ is the Gysin morphism, and the second equality follows from the Poincar\'e duality. Thus, the number $R$ represents the number of linear relations of $A_i-B_i$ in $H_{m-1}(\tilde{X_{0}})$.

 If $m$ is even, consider the kernel of the restriction map $$V^m:=\textup{ker}\left(H^{m}(\tilde{X_0})\rightarrow\bigoplus_{i=1}^{l}H^{m}(Q_i) \right).$$ 

 Let $$\left(h^{k,m-k}_+(\widetilde{X}_0),\ h^{k,m-k}_{-}(\widetilde{X}_0)\right)$$ be the signature of $S(C\bullet,\overline{\bullet})$ on $H^{k,m-k}(\widetilde{X}_0)$. When $m$ is even, let $$\left(\widehat{h}^{\frac{m}{2},\frac{m}{2}}_+(\widetilde{X}_0),\ \widehat{h}^{\frac{m}{2},\frac{m}{2}}_{-}(\widetilde{X}_0)\right)$$ be the signature of $S(C\bullet,\overline{\bullet})$ on the $(m/2,m/2)$ part of $V^m$.

Generalizing the results of Friedman \cite{RF1} and Li \cite{CL} on the existence of polarized Hodge structures on the middle cohomology of Clemens threefolds, we obtain the following theorem.

\begin{theorem} \label{thmodp}
    Let $f: X\rightarrow \Delta$ be a one-parameter degeneration of $m$-dimensional compact complex manifolds such that the central fiber $X_0$ has at worst ordinary double points. Suppose that $X_0$ admits a resolution of singularities that is a $\partial \overline{\partial}$-manifold. 
    
    Then   
    
   \begin{enumerate}
\item The nearby fiber $X_q$ is a $\partial\overline{\partial}$-manifold for every sufficiently small $q\neq 0$.

\item Suppose that $m$ is odd. For $k\neq\frac{m+1}{2},\frac{m-1}{2}$, the signature of $S(C\bullet,\overline{\bullet})$ on $H^{k,m-k}(X_q)$ is the same as that on $H^{k,m-k}(\widetilde{X}_0)$. For
$$
k=\frac{m+1}{2}
\qquad\textup{or}\qquad
k=\frac{m-1}{2},
$$
the signature is
$$
\left(
h^{k,m-k}_+(\widetilde{X}_0)+R,\,
h^{k,m-k}_-(\widetilde{X}_0)
\right)
$$
for every sufficiently small $q\neq 0$.

\item Suppose that $m$ is even. For $k\neq\frac{m}{2}$, the signature of $S(C\bullet,\overline{\bullet})$ on $H^{k,m-k}(X_q)$ is the same as that on $H^{k,m-k}(\widetilde{X}_0)$. For $k=\frac{m}{2}$, the signature is
$$
\left(
\widehat{h}^{\frac{m}{2},\frac{m}{2}}_+(\widetilde{X}_0)+l,\,
\widehat{h}^{\frac{m}{2},\frac{m}{2}}_-(\widetilde{X}_0)
\right)
$$
for every sufficiently small $q\neq 0$.

\end{enumerate}
   In particular, if $H^m(\widetilde{X}_0,\mathbb{C})$, equipped with its Hodge decomposition and the cup-product pairing $S(\bullet,\bullet)$, is a polarized Hodge structure, then $H^m(X_q,\mathbb{C})$, equipped with its Hodge decomposition and the cup-product pairing $S(\bullet,\bullet)$, is also a polarized Hodge structure for every sufficiently small $q\neq 0$.

\end{theorem}
\begin{proof}
The computation below is similar to the proof of \cite[Theorem 2.12]{RF1} for Calabi--Yau threefolds with ordinary double points. We begin by describing a semistable model of the family. Let
$\left\{x_i,\middle|,1\leq i\leq l\right\}$
be the set of ordinary double points of $X_0$, and let $Y$ be the blow-up of $X$ along this set. Consider the double cover of $\Delta$ defined by $q\mapsto q^2$, and let $\widetilde{Y}\rightarrow\Delta$ be the pullback family. Then the central fiber $\widetilde{Y}_0$ is a reduced simple normal crossing divisor.

The central fiber $\widetilde{Y}_0$ can be described more explicitly as follows. Let $\widetilde{X}_0$ be the blow-up of $X_0$ at the ordinary double points ${x_i}$. By assumption, $\widetilde{X}_0$ is a $\partial\overline{\partial}$-manifold. The exceptional divisor over each $x_i$ is a smooth $(m-1)$-dimensional quadric $Q_i\subset\mathbb{P}^m.$
For each $i$, let
$E_i\subset\mathbb{P}^{m+1}$
be a smooth $m$-dimensional quadric containing $Q_i$ as a hyperplane section. Then
$$
\widetilde{Y}_0
\cong
\left(
\widetilde{X}_0
\amalg
\coprod_{i=1}^l E_i
\right)\big/\sim,
$$
where, for each $i$, the equivalence relation identifies $Q_i\subset\widetilde{X}_0$ with a hyperplane section of $E_i$.

Consider Steenbrink's limiting Hodge structure associated to the family $\widetilde{Y} \rightarrow \Delta $. By Theorem \ref{Corg}, to show that $X_q$ is a $\partial\overline{\partial}$-manifold, we need to prove that $$W^{St}=W(N,d) \quad \textup{on} \quad H^d(X_{\infty},\Q)$$ for all $d$. Since $\widetilde{Y}$ has no triple points, Proposition \ref{criterionmonodromy} reduces this assertion to showing that $$\textup{ker}\left(\bigoplus_{1\leq i\leq l}H^{d-1}(Q_i)\xrightarrow{-\gamma} H^{d+1}(\widetilde{X}_{0})\oplus\bigoplus_{1\leq i\leq l} H^{d+1}(E_i) \right)$$ and $$\textup{coker}\left( H^{d-1}(\widetilde{X}_{0})\oplus\bigoplus_{1\leq i\leq l} H^{d-1}(E_i)\xrightarrow{\theta}\bigoplus_{1\leq i\leq l}H^{d-1}(Q_i)\right)$$ are identified under the map induced by the identity on
$$\bigoplus_{1\leq i\leq l}H^{d-1}(Q_i,\mathbb{Q}).$$ By the Lefschetz hyperplane theorem, the Gysin morphism
$$
H^{d-1}(Q_i)
\rightarrow
H^{d+1}(E_i)
$$
is injective unless $d=m$, whereas the restriction morphism
$$
H^{d-1}(E_i)
\rightarrow
H^{d-1}(Q_i)
$$
is surjective unless $d=m$. Therefore, when $d\neq m$, both the kernel and the cokernel above vanish, and the required assertion follows immediately.

We now consider the case $d=m$. If $m$ is even, then each $Q_i$ is an odd-dimensional quadric, and hence its middle cohomology group $H^{m-1}(Q_i)$ vanishes. Therefore, the assertion is immediate.

Suppose now that $m$ is odd. By Proposition \ref{criND}, it suffices to prove that the subspace $$V:=\textup{im}\left(H^{m-1}(\tilde{X}_0)\oplus\bigoplus_{1\leq i\leq l} H^{m-1}(E_i)\xrightarrow{\theta}\bigoplus_{1\leq i\leq l} H^{m-1}(Q_i)\right)$$ is nondegenerate with respect to the direct sum of the cup-product pairings. 

If $\frac{m-1}{2}$ is even, then the cup-product pairing on each $H^{m-1}(Q_i)$ is positive definite. Consequently, its restriction to the subspace $V$ is also positive definite, and hence $V$ is nondegenerate.

Suppose instead that $\frac{m-1}{2}$ is odd. For each $i$, the restriction morphism $  H^{m-1}(E_i)\rightarrow H^{m-1}(Q_i)$ maps the class $h_i^{\frac{m-1}{2}}$ to $A_i+B_i$. The direct sum of the cup-product pairings on $\bigoplus_{1\leq i\leq l} H^{m-1}(Q_i)$ has signature $(l,l)$, while the classes $A_1+B_1,\ldots,A_l+B_l$
are mutually orthogonal and have positive square. They therefore span an $l$-dimensional maximal positive-definite subspace contained in $V$. By Sylvester's law of inertia, we can conclude that $V$ is a non-degenerate subspace. 

To compute the Hodge index on the middle cohomology of the general fiber, Theorem \ref{Corg2} reduces the problem to computing the signature of $S(C\bullet, \overline{(-N)^{l} \bullet})$ on the graded pieces $Gr^{W^{St}}_{m+l}H^m(\widetilde{Y}_{\infty})$ of the limiting mixed Hodge structure on the middle cohomology for each $l$. These signatures can be computed explicitly using Proposition \ref{Fu2}.

Suppose first that $m$ is odd. We have $$Gr^{W^{St}}_{m}H^m(\widetilde{Y}_{\infty})=H^m(\widetilde{X}_0).$$  Under this identification, Proposition \ref{Fu2} shows that the Hermitian form $S(C\bullet,\overline{\bullet})$ on $Gr^{W^{St}}_{m}H^m(X_{\infty})$ agrees with the Hermitian form $S(C\bullet,\overline{\bullet})$ on $H^m(\widetilde{X}_0)$. By assumption, its signature on $H^{k,m-k}(\widetilde{X}_0)$ is $$\left(h^{k,m-k}_+(\tilde{X_0}),\ h^{k,m-k}_{-}(\tilde{X_0})\right).$$ Let $\eta =\sum_{i=1}^l \frac{m_i}{2\pi \sqrt{-1}}(A_i-B_i)$ be an element of  $$Gr^{W^{St}}_{m+1}H^m(\widetilde{Y}_{\infty})=\textup{ker}\left(\bigoplus_{1\leq i\leq l}H^{m-1}(Q_i)(-1)\xrightarrow{-\gamma} H^{m+1}(\widetilde{X}_{0})\oplus\bigoplus_{1\leq i\leq l} H^{m+1}(E_i) \right)$$
$$=\left\{\sum_{i=1}^l \frac{m_i}{2\pi \sqrt{-1}}(A_i-B_i)\in\bigoplus_{1\leq i\leq l}H^{m-1}(Q_i)(-1)\middle|\sum_{i=1}^l m_i(A_i-B_i)=0 \quad\textup{in}\quad H_{m-1}(\widetilde{X}_0)\right\},$$
 by Proposition \ref{Fu2} again, we have 
\begin{equation*}
\begin{split}
    (\sqrt{-1})^{\frac{m+1}{2}-\frac{m+1}{2}}S(\eta, -N\eta )&=\epsilon(m)Q(\eta,-N\eta) \\
    &=(-1)\epsilon(m)\epsilon(-m)\epsilon(1-m)\sum_{i=1}^{l}m_i^2(A_i-B_i)^2\\
    &=2\sum_{i=1}^{l}m_i^2>0.
\end{split}
\end{equation*}
 Consequently, the signature of $S(C\bullet,\overline{(-N)\bullet})$ on the Hodge $(k,m+1-k)$-component of $Gr^{W^{St}}_{m+1}H^m(\widetilde{Y}_{\infty})$is $(R,0)$ when $k=\frac{m+1}{2}$ and $(0,0)$ for all other values of $k$. Therefore, by Theorem \ref{Corg2}, the signature of $S(C\bullet,\overline{\bullet})$ on the $(k,m-k)$-component of $H^m(X_q)$ is $$\left\{\begin{array}{cc}
        \left(h^{k,m-k}_+(\tilde{X_0}),\ h^{k,m-k}_{-}(\tilde{X_0})\right),   & k\neq \frac{m+1}{2},\frac{m-1}{2}, \\
          \left(h^{k,m-k}_+(\tilde{X_0})+R,\ h^{k,m-k}_{-}(\tilde{X_0})\right),   & k= \frac{m+1}{2}\textup{ or }\frac{m-1}{2}
        \end{array}\right.$$ for $|q|$ sufficiently small.

The positivity of $S(\eta,-N\eta)$ above also follows from the Picard--Lefschetz formula. In fact, regard $\eta$ as an element of $H^m(\widetilde{Y}_{\infty})\cong H^m(X_q)$. If we let $\delta_i$ denote the vanishing cycle associated with each ODP $x_i$, then we have $$T(\eta)=\eta+2\sum_{i=1}^l\epsilon(m+2)Q_{X_q}(\eta,\delta_i)\delta_i,$$ where the multiplicity $2$ reflects the $2$-to-$1$ base change used in the construction of the semistable family. Since $N^2=0$ in this situation, we have $N=\log T=T-\operatorname{id}$. Therefore, combined with the fact that $N\eta\neq 0$, we obtain $$\epsilon(m)Q(\eta,-N\eta)=-\epsilon(m)\epsilon(m+2)\sum_{i=1}^l 2(Q_{X_q}(\eta,\delta_i))^2 > 0.$$ 

Next, we discuss the case where $m$ is even. In this case, we have $$Gr^{W^{St}}_{m}H^m(\widetilde{Y}_{\infty})=\frac{\textup{ker}\left(H^m(\widetilde{X}_0)\oplus\bigoplus_{i=1}^l H^m(E_i)\xrightarrow{\theta} \bigoplus_{i=1}^l H^m(Q_i)\right)}{\textup{im}\left(\bigoplus_{i=1}^l H^{m-2}(Q_i)(-1)\xrightarrow{-\gamma} H^m(\widetilde{X}_0)\oplus\bigoplus_{i=1}^l H^m(E_i)\right)}.$$ It suffices to show that every class in $Gr^{W^{St}}_{m}H^m(\widetilde{Y}_{\infty})$ admits a representative in $V^m\oplus H^m_{\textup{prim}}(E_i).$ To see this, let $A_i$ and $B_i$ denote the homology classes of the two families of standard $\frac{m}{2}$-dimensional planes on $E_i$. Then the image of the Gysin morphism $H^{m-2}(Q_i)\rightarrow H^m(E_i)$ is generated by $A_i+B_i$, while  $$H^m_{\textup{prim}}(E_i)\cong\textup{ker}\left(\bigoplus_{i=1}^l H^m(E_i)\xrightarrow{\theta} \bigoplus_{i=1}^l H^m(Q_i)\right)\cong \textup{Span}\{A_1-B_1,...,A_l-B_l\}.$$ Therefore, after taking the quotient by the image of the Gysin morphism, every class in $Gr^{W^{St}}_{m}H^m(\widetilde{Y}_{\infty})$ admits a representative in $V^m\oplus H^m(E_i)_{\textup{prim}}.$ This finishes the proof.

\end{proof}

\begin{remark}
 We can see that the argument above doesn't rely on any topological constraints in deformation theory to ensure the existence of smoothing. For Calabi-Yau threefolds with ODPs, by the work of Clemens, Friedman \cite{RF2}, Tian \cite{Tian}, Kawamata \cite{Kawamata}, Ran \cite{Ran}, a necessary and sufficient condition for the existence of smoothing is that there must be some linear relation $\sum_{i=1}^lm_i(A_i-B_i)=0$ in $H^{2}(\Tilde{X}_0)$ with $m_i\neq 0$ for all $i$. In higher dimension, Rollenske-Thomas \cite{RT} shows that a necessary condition for an odd-dimensional Calabi-Yau with ODPs to have a smoothing is that there must be some non-linear relation between $A_i-B_i$ in $H^{m-1}(\Tilde{X}_0)$. It is possible that for more complicated singularities than ODPs, the topological constraints in deformation theory should come in to ensure that the middle limiting cohomology is endowed with a polarized mixed Hodge structure.
\end{remark}

\subsection{ $O_{16}$ Singularities (Cones over Cubic Surfaces)}

In this section, we study threefolds with an $O_{16}$ singularity, that is, a singularity given by the vertex of the affine cone over a smooth cubic surface in $\mathbb{P}^3$.

\begin{theorem}
Let $f:X\rightarrow\Delta$ be a one-parameter degeneration of compact complex threefolds such that the central fiber $X_0$ has a unique $O_{16}$ singularity. Suppose that $X_0$ admits a resolution of singularities which is a $\partial\overline{\partial}$-manifold. Then the following hold. 

\begin{enumerate}
    \item The nearby fiber $X_q$ is a $\partial\overline{\partial}$-manifold for $|q|$ sufficiently small.

    \item Let $\widetilde{X}_0$ be the blowup of $X_0$ at the $O_{16}$ singularity. If $H^3(\widetilde{X}_0,\mathbb{C})$, together with the cup-product pairing $S(\bullet,\bullet)$, is a polarized Hodge structure, then $H^3(X_q,\mathbb{C})$, together with the cup-product pairing $S(\bullet,\bullet)$, is a polarized Hodge structure for $|q|$ sufficiently small.
\end{enumerate}

\end{theorem}

\begin{proof}
We first describe a semistable model of the family. Let $Y$ be the blowup of $X$ along the $O_{16}$ singularity. Consider the $3$-to-$1$ base change of $\Delta$ defined by $q\mapsto q^3$, and let $\widetilde{Y}\rightarrow \Delta$ be the resulting pullback family. Then $\widetilde{Y}_0$ is a reduced normal crossing variety.

We can describe $\widetilde{Y}_0$ more explicitly as follows. Let $\widetilde{X}_0$ be the blowup of $X_0$ at the $O_{16}$ singularity. Then $\widetilde{X}_0$ is a compact complex manifold of Fujiki class $\mathcal{C}$, and hence is a $\partial\overline{\partial}$-manifold. The exceptional divisor of $\widetilde{X}_0\rightarrow X_0$ is a cubic surface $Q$ in $\mathbb{P}^3$. Let $E$ be a cubic threefold in $\mathbb{P}^4$ containing $Q$ as a hyperplane section. Then
$$
\widetilde{Y}_0\cong \widetilde{X}_0\amalg E/\sim,
$$
where the equivalence relation identifies $Q\subset \widetilde{X}_0$ with the corresponding hyperplane section $Q\subset E$.

Consider Steenbrink's limiting mixed Hodge structure associated with the family $\widetilde{Y}\rightarrow\Delta$. As in the proof of the ODP case, the mixed Hodge structure on $H^d(\widetilde{Y}_{\infty})$ is pure unless $d=3$, by the Lefschetz hyperplane theorem. Thus, it suffices to show that
$$
W^{St}=W(N,3)
$$
on the middle cohomology $H^3(\widetilde{Y}_{\infty})$, and this follows from Corollary \ref{cri3}. In particular, the nearby fiber $X_q$ is a $\partial\overline{\partial}$-manifold for $|q|$ sufficiently small.

It remains to prove the assertion concerning polarization. The argument is similar to the ODP case. First, since
$$
\operatorname{Gr}_4^{W^{St}}H^3(\widetilde{Y}_{\infty})
=
\operatorname{ker}\left(
H^2(Q)(-1)
\xrightarrow{-\gamma}
H^4(\widetilde{X}_0)\oplus H^4(E)
\right)
\subset H^2_{\operatorname{prim}}(Q)(-1),
$$
the Hodge structure $\operatorname{Gr}_4^{W^{St}}H^3(\widetilde{Y}_{\infty})$ is polarized by
$S(C\bullet,\overline{(-N)\bullet})$. On the other hand,
$$
\operatorname{Gr}_3^{W^{St}}H^3(\widetilde{Y}_{\infty})
=
H^3(\widetilde{X}_0)\oplus H^3(E).
$$
By assumption, $H^3(\widetilde{X}_0)$ is polarized by
$S(C\bullet,\overline{\bullet})$, while the same holds for $H^3(E)$ since $E$ is a smooth projective cubic threefold. Hence
$\operatorname{Gr}_3^{W^{St}}H^3(\widetilde{Y}_{\infty})$ is polarized by
$S(C\bullet,\overline{\bullet})$. Therefore, Theorem \ref{Corg2} implies that $H^3(X_q,\mathbb{C})$, together with the cup-product pairing $S$, is a polarized Hodge structure for $|q|$ sufficiently small.
\end{proof}
\begin{remark}
Similar to the ODP case, the theorem above doesn't rely on the deformation theory of $O_{16}$ singularities either. We discuss the deformation theory for Calabi-Yau threefolds with an $O_{16}$ singularity here, though it is not needed in the proof.

We first review the general deformation theory of Calabi-Yau threefold with isolated rational singularities. The general reference is \cite[Section 5]{FrLa2}.

In the following, let $X$ be a canonical Calabi-Yau threefold in the sense of \cite[Definition 5.1]{FrLa2} such that $H^1(X,\mathcal{O}_X)=0$. Let $Z$ be the singular locus of $X$ and $\pi: \widehat{X}\rightarrow X$ a good equivariant resolution of $X$ so that $E=\pi^{-1}(Z)$ is a divisor with simple normal crossing. For each $x\in Z$, let $E_x=\pi^{-1}(x)$ and $\widehat{X}_x$ be the germ of $\widehat{X}$ around $E_x$. Consider the following spaces: $$S:=\bigoplus_{x\in Z}\textup{ker}(H^4_{E_x}(\widehat{X}_x)\rightarrow H^4(E_x)), $$ $$T:=\textup{ker}(S\rightarrow H^4(\widehat{X})),$$ $$K':=\bigoplus_{x\in Z}\textup{ker}(H^{2}_{E_x}(\Omega^2_{\widehat{X}_x})\rightarrow H^4_{E_x}(\widehat{X}_x)),$$ $$K:=\textup{ker}(H^2_E(\widehat{X},\Omega^2_{\widehat{X}})\rightarrow H^0(R^2\pi_{*}\Omega^2_{\widehat{X}}))\cong K'\oplus S.$$ Then by \cite[(5.1) and Lemma 5.6]{FrLa2}, there exists a commutative diagram of exact sequences 
 \begin{center}
     \begin{tikzcd}
\mathbb{T}^1_{X} \arrow[d] \arrow[r] & K'\oplus T \arrow[r] \arrow[d]                      & 0 \arrow[d]                           \\
{H^0(X,T^1_X)} \arrow[d] \arrow[r]   & K \arrow[r] \arrow[d]                               & 0 \arrow[d]                           \\
{H^2(X,T_X^0)} \arrow[r]             & {H^2(\widehat{X},\Omega^2_{\widehat{X}})} \arrow[r] & H^0(R^2\pi_{*}\Omega^2_{\widehat{X}})
\end{tikzcd}
 \end{center}

We define the \textit{defect} $\sigma(X)$ of $X$ to be $\textup{dim}_{\C}\textup{im}(S\rightarrow H^4(\widehat{X}))=\textup{dim}_{\C}S -\textup{dim}_{\C}T.$ By the above diagram, we can see that if the restriction of global deformations $\mathbb{T}_X^1$ to local deformations $H^0(X, T^1_X)$ is surjective, then $T\cong S$, or equivalently $\sigma(X)=0$.

The invariant $\sigma(X)$ was introduced by Kawamata \cite[p.27]{Kawamata} for a normal projective variety $X$ to measure the failure of $\Q$-factoriality on $X$. If $X$ is a compact complex threefold with only isolated rational singularities, then the defect $\sigma(X)$ also measures the failure of Poincar\'e duality on $X$ (cf. \cite[Section 3]{NaSt} and \cite[Remark 2.2]{FrLa}). 

Now, if $X$ is a quintic threefold with an $O_{16}$ singularity, it is easy to see that any deformation of the $O_{16}$ singularity can be extended to a global embedded deformation of $X$, and hence $\sigma(X)=0$. As a result, it is natural to expect that a Calabi-Yau threefold with an $O_{16}$ singularity and with "good deformation theory" has no defect. In this case,  by the Thom isomorphism theorem, we have $$S:=\textup{ker}(H^4_{Q}(\widetilde{X}_0)\rightarrow H^4(Q))\cong\textup{ker}(H^2(Q)\rightarrow H^4(Q))=H^2(Q)_{\textup{prim}}.$$ Since the defect $\sigma(X_0)$ is zero,  $$\operatorname{Gr}_4^{W^{St}}H^3(\widetilde{Y}_{\infty})=\textup{ker}\left(H^{2}_{\textup{prim}}(Q)(-1)\xrightarrow{-\gamma} H^{4}(\widetilde{X}_{0})\oplus H^{4}(E) \right)=H^{2}_{\textup{prim}}(Q)(-1),$$ and the map $$\operatorname{Gr}_4^{W^{St}}H^3(\widetilde{Y}_{\infty})\xrightarrow{N}\operatorname{Gr}_2^{W^{St}}H^3(\widetilde{Y}_{\infty})$$ is simply  $$H^{2}(Q)_{\textup{prim}}(-1)\xrightarrow{(2\pi\sqrt{-1})\operatorname{Id}}H^{2}(Q)_{\textup{prim}}.$$
\end{remark}

\subsection{Non-K\"ahler Calabi-Yau Threefolds with Arbitrarily Large $b_2$}

In \cite[Theorem 1.1]{HaSa}, Hashimoto--Sano construct a family of non-K\"ahler Calabi--Yau threefolds $X_3(a)$, indexed by positive integers $a$, such that $b_2(X_3(a))$ tends to infinity as $a$ tends to infinity.

We first recall the construction of Hashimoto--Sano. Let
$$
\mathbb{P}(3):=\mathbb{P}^1\times\mathbb{P}^1\times\mathbb{P}^1,
$$
and let $S$ be a very general hypersurface of multidegree $(2,2,2)$ in $\mathbb{P}(3)$. Then $S$ is a $K3$ surface, and, by the Noether--Lefschetz theorem,
$$
\operatorname{Pic}(S)=\bigoplus_{i=1}^3\mathbb{Z}e_i,
$$
where $e_1,e_2,e_3$ are the fiber classes of the three elliptic fibrations $S\rightarrow\mathbb{P}^1$ induced by projection onto the three factors. By \cite[\S 3.1]{HaSa}, for every positive integer $a$, there exists an automorphism $\iota^a\in\operatorname{Aut}(S)$ such that
$$
(\iota^a)^*
=
\left(
\begin{array}{ccc}
1 & 4a^2-2a & 4a^2+2a \\
0 & 1-2a & -2a \\
0 & 2a & 1+2a
\end{array}
\right)
\in
\operatorname{Aut}(\mathbb{Z}^3)
\cong
\operatorname{Aut}(\operatorname{Pic}(S)).
$$

The $d$-semistable SNC Calabi--Yau variety $X_0$ is constructed as follows. Let $X_1'$ be the blowup of $\mathbb{P}(3)$ along
$f_1,\ldots,f_a\in|\mathcal{O}_S(1,0,0)|$, and let $X_1$ be the blowup of $X_1'$ along the strict transform of a general member
$$
C_a\in
|\mathcal{O}_S(16a^2-a+4,4-8a,4+8a)|.
$$
Let $X_2:=\mathbb{P}(3)$, and let $S_1:=S\subseteq X_1$ and $S_2:=S\subseteq X_2$. Then $X_0$ is obtained by gluing $X_1$ and $X_2$ along $S_1$ and $S_2$ via the automorphism
$$
\iota^a:S_1\longrightarrow S_2.
$$
By the theorem of Namikawa--Kawamata \cite[Theorem 4.2]{NaKa}, $X_0$ admits a global smoothing
$
f:X\rightarrow\Delta
$
whose central fiber is isomorphic to $X_0$. The threefold $X_3(a)$ is defined to be a nearby fiber $X_q$ for $|q|$ sufficiently small.

As a consequence of Theorems \ref{Corg} and \ref{Corg2}, we obtain the following.

\begin{theorem}\label{HaSa}
The non-K\"ahler Calabi--Yau threefolds $X_3(a)$ constructed by Hashimoto--Sano \cite[Theorem 1.1]{HaSa} are $\partial\overline{\partial}$-manifolds. Moreover, $$H^3(X_3(a),\mathbb{C}),$$ together with the cup-product pairing $S(\bullet,\bullet)$, is a polarized Hodge structure.
\end{theorem}

\begin{proof}
As in the ODP and $O_{16}$ cases, the mixed Hodge structure on $$H^d(X_{\infty},\C)$$ is pure unless $d=3$, by the Lefschetz hyperplane theorem. Moreover, by Corollary \ref{cri3}, we have
$$
W^{St}=W(N,3)
$$
on $H^3(X_{\infty},\C)$. Hence, by Theorem \ref{Corg}, $X_3(a)$ is a $\partial\overline{\partial}$-manifold.

Next, we prove the assertion concerning polarization. We have
$$
\operatorname{Gr}^{W^{St}}_3H^3(X_{\infty},\Q)
\cong
H^1(C_a,\Q)(-1)\oplus \bigoplus_{i=1}^a H^1(f_i,\Q)(-1),
$$
which, together with the pairing induced by the cup product $S$, is a polarized Hodge structure by Proposition \ref{tool2}. 

On the other hand, we have
$$
\operatorname{Gr}^{W^{St}}_4H^3(X_{\infty},\Q)
=
\operatorname{ker}\left(
H^2(S_1,\Q)(-1)
\xrightarrow{-\gamma}
H^4(X_1,\Q)\oplus H^4(X_2,\Q)
\right).
$$
In particular,
$$
\operatorname{Gr}^{W^{St}}_4H^3(X_{\infty},\Q)
\subseteq
\operatorname{ker}\left(
H^2(S_1,\Q)(-1)
\xrightarrow{-\gamma}
H^4(X_2,\Q)
\right).
$$
The latter kernel is contained in the primitive subspace of $H^2(S_1,\Q)(-1)$ with respect to the ample line bundle
$(\iota^a)^*N_{S_2/X_2}$. More precisely, by Proposition \ref{tool1} it consists of those classes
$\eta\in H^2(S_1,\Q)(-1)$ satisfying
$$
\eta\wedge c_1\bigl((\iota^a)^*N_{S_2/X_2}\bigr)=0,
$$
where $N_{S_2/X_2}$ denotes the normal bundle of $S_2$ in $X_2$. Therefore,
$$
\operatorname{Gr}^{W^{St}}_4H^3(X_{\infty},\Q)
\subseteq
H^2_{\operatorname{prim}}(S_1,\Q)(-1),
$$
and hence $\operatorname{Gr}^{W^{St}}_4H^3(X_{\infty},\Q)$ is polarized by
$$
S(C\bullet,\overline{(-N)\bullet})
$$
by Proposition \ref{Fu2}. It then follows from Theorem \ref{Corg2} that
$H^3(X_3(a),\C)$, together with the cup-product pairing $S(\bullet,\bullet)$, is a polarized Hodge structure.
\end{proof}

\begin{remark}
By a similar computation, one can show that another class of non-K\"ahler Calabi--Yau threefolds with arbitrarily large $b_2$, constructed by Sano \cite{Sa2}, are also $\partial\overline{\partial}$-manifolds, and that their middle cohomology, together with the cup-product pairing $S$, is a polarized Hodge structure. It is therefore natural to ask whether there exists a non-K\"ahler Calabi--Yau threefold which does not satisfy the $\partial\overline{\partial}$-lemma, or one which satisfies the $\partial\overline{\partial}$-lemma but whose middle cohomology, together with the cup-product pairing $S$, is not a polarized Hodge structure.
\end{remark}

\subsection{Non K\"ahler Calabi-Yau $m$-folds with Arbitrarily Large $b_2$}

In this section, we consider the non-K\"ahler Calabi-Yau $m$-folds with arbitrarily large $b_2$, for $m\geq 4$, constructed by Sano in \cite[Theorem 1.1]{Sa}. These examples are not even in Fujiki class $\mathcal{C}$; see \cite[Proposition 3.5]{Sa}. We first recall Sano's construction, which is based on projective Calabi-Yau manifolds of Schoen type.

Let
$
S\subset \mb{P}^2\times\mb{P}^1
$
be a general hypersurface of bidegree $(3,1)$, so that $S$ is a general rational elliptic surface, and let
$
T\subset \mb{P}^1\times\mb{P}^{m-2}
$
be a general hypersurface of bidegree $(1,m-1)$. Then the fiber product
$
S\times_{\mb{P}^1}T
$
is a projective Calabi--Yau manifold of Schoen type.

The surface $S$ can be regarded as the blowup of $\mb{P}^2$ at nine points. Let $H$ denote the pullback of $\mathcal{O}_{\mb{P}^2}(1)$ to $S$, and let $h$ be its cohomology class. Let $E_i$, $i=1,\ldots,9$, denote the exceptional divisors, and let $e_i$, $i=1,\ldots,9$, be their cohomology classes. Then $H^2(S)$ is generated by
$
h,e_1,\ldots,e_9.
$
Using quadratic transformations of $\mb{P}^2$, Sano proves in \cite[Proposition 2.6]{Sa} that, for every positive integer $a$, there exists an automorphism over $\mb{P}^1$
$$
\phi_a:S\longrightarrow S
$$
such that
$$
\phi_a^*H
=
(27a^2+1)H
-(9a^2-3a)(E_1+E_2+E_3)
-9a^2(E_4+E_5+E_6)$$$$
-(9a^2+3a)(E_7+E_8+E_9).
$$

For $i=1,2,$ let
$
Y_i:=\mb{P}^2\times T$ , and let
$
D_i:=(S\times\mb{P}^{m-2})\cap Y_i
=
S\times_{\mb{P}^1}T.
$
Write
$$
p_S:S\times_{\mb{P}^1}T\longrightarrow S
$$
for the projection. Choose distinct smooth fibers
$$
f_i\in |-K_S|,\qquad i=1,\ldots,a,
$$
of the elliptic fibration $S\rightarrow\mb{P}^1$, and set
$$
F_i:=p_S^{-1}(f_i),\qquad i=1,\ldots,a.
$$
Then each $F_i$ is a smooth fiber of the induced fibration
$
S\times_{\mb{P}^1}T\xrightarrow{g}\mb{P}^1.
$

Next, choose a smooth irreducible curve
$$
C_a\in |3H+3\phi_a^*H+(a-2)K_S|.
$$
By \cite[Proposition 2.6 (iii)]{Sa}, this linear system is ample and basepoint-free. Moreover, by the proof of \cite[Proposition 3.6]{Sa}, the genus of $C_a$ is
$$
g(C_a)=9(27a^2-2a+5)-1,
$$
and the morphism
$
C_a\xrightarrow{g'}\mb{P}^1
$
induced by the elliptic fibration $S\xrightarrow{g'}\mb{P}^1$ has degree $18$. Let
$$
\Gamma_a:=p_S^{-1}(C_a)
\cong
C_a\times_{\mb{P}^1}T.
$$

Let $\widetilde{Y}_1$ be the blowup of $Y_1$ along $F_1,\ldots,F_a$, and let $X_1$ be the blowup of $\widetilde{Y}_1$ along the strict transform of $\Gamma_a$. We continue to denote by $D_1$ its strict transform in $X_1$, so that $D_1\subset X_1$. Let
$$
X_2:=Y_2=\mb{P}^2\times T.
$$
The $d$-semistable simple normal crossing Calabi--Yau variety
$$
X_0=X_1\amalg X_2/\sim
$$
is obtained by gluing $D_1$ and $D_2$ via
$$
X_1\supseteq S\times_{\mb{P}^1}T
\xrightarrow{\phi_a\times\operatorname{id}}
S\times_{\mb{P}^1}T
\subseteq X_2.
$$
By the theorem of Namikawa--Kawamata \cite[Theorem 4.2]{NaKa}, $X_0$ admits a global smoothing
$$
f:X\longrightarrow\Delta
$$
whose central fiber is isomorphic to $X_0$. The non-K\"ahler Calabi--Yau $m$-fold $X_m(a)$ is then defined to be a nearby fiber $X_q$ for $|q|$ sufficiently small.

Our main theorem is:

\begin{theorem}\label{Sa}
For every $k$, set
$$
\nu_k:=h^{k-1,m-k-1}(T),
$$
with the convention that $h^{p,q}(T)=0$ if $p<0$ or $q<0$. Then the following hold.

\begin{enumerate}
    \item The non-K\"ahler Calabi-Yau $m$-folds $X_m(a)$ are $\partial\overline{\partial}$-manifolds.

    \item Suppose that $m$ is odd.

    If
    $
    m\equiv 3\pmod 4,
    $
    then, for every $k$, the signature of
    $
    S(C\bullet,\overline{\bullet})
    $
    on $H^{k,m-k}(X_m(a))$ is
    $$
    \left(
    h^{k,m-k}(X_m(a))-\nu_k,\,
    \nu_k
    \right).
    $$

    If
    $
    m\equiv 1\pmod 4,
    $
    then, for
    $$
    k\neq\frac{m-1}{2},\frac{m+1}{2},
    $$
    the signature of
    $S(C\bullet,\overline{\bullet})$
    on $H^{k,m-k}(X_m(a))$ is
    $$
    \left(
    h^{k,m-k}(X_m(a))-\nu_k,\,
    \nu_k
    \right),
    $$
    while for
    $$
    k=\frac{m-1}{2}
    \qquad\textup{or}\qquad
    k=\frac{m+1}{2},
    $$
    the signature is
    $$
    \left(
    h^{k,m-k}(X_m(a))
    -\nu_k
    -9(27a^2-2a+5)-a-6,\,
    \nu_k
    +9(27a^2-2a+5)+a+6
    \right).
    $$

    \item Suppose that $m$ is even.

    If
    $
    k\neq\frac{m}{2},
    $
    then the signature of
    $S(C\bullet,\overline{\bullet})$
    on $H^{k,m-k}(X_m(a))$ is
    $$
    \left(
    h^{k,m-k}(X_m(a))-\nu_k,\,
    \nu_k
    \right).
    $$

    If $m=4$ and $k=2$, then the signature is
    $$
    \left(
    h^{2,2}(X_4(a))-\nu_2-a+1,\,
    \nu_2+a-1
    \right).
    $$

    If $m\geq 6$ is even and
    $
    k=\frac{m}{2},
    $
    then the signature is
    $$
    \left(
    h^{\frac{m}{2},\frac{m}{2}}(X_m(a))
    -\nu_{\frac{m}{2}}-a,\,
    \nu_{\frac{m}{2}}+a
    \right).
    $$
\end{enumerate}
\end{theorem}

The main difficulty in applying Theorem \ref{Corg} to $X_m(a)$ is to understand the cohomology ring of the fiber product $S\times_{\mb{P}^1}T$. For this purpose, we use the following theorem due to the author. Its proof will appear in a forthcoming article; see also \cite[Chapter 5]{thesis}.

\begin{theorem}\label{fib}
Let
$$
\widetilde{X}_i\xrightarrow{f_i}\mb{P}^1,\qquad i=1,2,
$$
be two Lefschetz fibrations, each obtained by blowing up the base locus of a Lefschetz pencil contained in an ample linear system on a smooth projective variety, and let
$
m_i=\operatorname{dim}_{\C}\widetilde{X}_i.
$
Suppose that the sets of critical values of $f_1$ and $f_2$ are disjoint in $\mb{P}^1$. Then the canonical morphism of graded rings
$$
H^{\bullet}(\widetilde{X}_1,\Q)
\otimes_{H^{\bullet}(\mb{P}^1,\Q)}
H^{\bullet}(\widetilde{X}_2,\Q)
\longrightarrow
H^{\bullet}(\widetilde{X}_1\times_{\mb{P}^1}\widetilde{X}_2,\Q)
$$
is an isomorphism in degrees at most $m_1+m_2-2$ and is injective in degree $m_1+m_2-1$.
\end{theorem}

Assuming Theorem \ref{fib}, Theorem \ref{Sa} follows from Theorems \ref{Corg} and \ref{Corg2}, together with a lengthy but standard computation of the relevant cohomology groups, restriction maps, and Gysin morphisms using Propositions \ref{tool1}, \ref{tool2}, and \ref{tool3}. We defer these computations to the Appendix, where we prove Propositions \ref{tedious}, \ref{tedious1}, and \ref{tedious2}.

\begin{proof}[Proof of Theorem \ref{Sa} assuming Propositions \ref{tedious}, \ref{tedious1}, and \ref{tedious2}]

In order to prove that $X_m(a)$ is a $\partial\overline{\partial}$-manifold, by Theorem \ref{Corg} and Proposition \ref{criterionmonodromy}, it suffices to show that
$$
\operatorname{ker}\left(
H^{k-1}(S\times_{\mb{P}^1}T)
\xrightarrow{-\gamma}
H^{k+1}(X_1)\oplus H^{k+1}(X_2)
\right)
$$
and
$$
\operatorname{coker}\left(
H^{k-1}(X_1)\oplus H^{k-1}(X_2)
\xrightarrow{\theta}
H^{k-1}(S\times_{\mb{P}^1}T)
\right)
$$
are isomorphic via the map induced by the identity for all $k\geq 0$. By Proposition \ref{tool3}, it suffices to verify this for $k\leq m$.

We denote by $h$ the cohomology class of the pullback of $\mathcal{O}_{\mb{P}^2}(1)$ to $S$, and by $f'$ the cohomology class of the pullback of $\mathcal{O}_{\mb{P}^1}(1)$ to $S$. Then
$$
f'=[-K_S]=3h-\sum_{i=1}^9e_i.
$$
We also denote by $\eta$ the cohomology class of the pullback of $\mathcal{O}_{\mb{P}^{m-2}}(1)$ to $T$, and by $f''$ the cohomology class of the pullback of $\mathcal{O}_{\mb{P}^1}(1)$ to $T$. Then
$$
H^\bullet(S)\otimes_{H^\bullet(\mb{P}^1)}H^\bullet(T)
$$
is the quotient of $H^\bullet(S)\otimes_{\Q}H^\bullet(T)$ obtained by identifying $f'$ with $f''$. We denote their common image in the quotient by $f$. For simplicity, we also denote the corresponding classes on $S\times_{\mb{P}^1}T$ by $h$, $f$, and $\eta$.

By Theorem \ref{fib}, the canonical morphism
$$
H^\bullet(S)\otimes_{H^\bullet(\mb{P}^1)}H^\bullet(T)
\longrightarrow
H^\bullet(S\times_{\mb{P}^1}T)
$$
is an isomorphism in degrees at $\leq m-2$ and is an injection in degree $m-1$. Thus, for $k\leq m-1$, the space $H^{k-1}(S\times_{\mb{P}^1}T)$ is generated by
$$
\left\{
\begin{array}{ll}
\eta^{\frac{k-1}{2}},\
H^2(S)\cup\eta^{\frac{k-3}{2}},\
h^2\cup\eta^{\frac{k-5}{2}},
& \textup{if $k$ is odd and $k\neq m-1$,}\\[2mm]
H^{m-2}_{\operatorname{prim}}(T),\
\eta^{\frac{k-1}{2}},\
H^2(S)\cup\eta^{\frac{k-3}{2}},\
h^2\cup\eta^{\frac{k-5}{2}},
& \textup{if $k$ is odd and $k=m-1$,}\\[2mm]
0,
& \textup{if $k$ is even and $k\neq m-1$,}\\[2mm]
H^{m-2}(T),
& \textup{if $k$ is even and $k=m-1$.}
\end{array}
\right.
$$

Here and below, terms involving negative powers of $\eta$ are understood to be absent.

We will prove the following.

\begin{proposition}\label{tedious}
Let
$
R=H^\bullet(S)\otimes_{H^\bullet(\mb{P}^1)}H^\bullet(T),
$
and denote by $R^k$ its degree-$k$ part. Then the following hold.
\begin{enumerate}
    \item For $k\leq m-1$,
    $$
    \operatorname{ker}\left(
    H^{k-1}(S\times_{\mb{P}^1}T)
    \xrightarrow{-\gamma}
    H^{k+1}(X_1)\oplus H^{k+1}(X_2)
    \right)
    $$$$=
    \left\{
    \begin{array}{ll}
    \langle h,\phi_a^*h,f\rangle^\perp
    \cup\eta^{\frac{k-3}{2}},
    & \textup{if $k$ is odd,}\\[2mm]
    0,
    & \textup{if $k$ is even.}
    \end{array}
    \right.
    $$

    \item For $k=m$,
    $$
    \operatorname{ker}\left(
    R^{m-1}
    \rightarrow
    H^{m-1}(S\times_{\mb{P}^1}T)
    \xrightarrow{-\gamma}
    H^{m+1}(X_1)\oplus H^{m+1}(X_2)
    \right)
    $$$$=
    \left\{
    \begin{array}{ll}
    \langle h,\phi_a^*h,f\rangle^\perp
    \cup\eta^{\frac{m-3}{2}},
    & \textup{if $m$ is odd,}\\[2mm]
    0,
    & \textup{if $m$ is even.}
    \end{array}
    \right.
    $$

    \item For $k\leq m$, the image
    $$
    \operatorname{im}\left(
    H^{k-1}(X_1)\oplus H^{k-1}(X_2)
    \xrightarrow{\theta}
    H^{k-1}(S\times_{\mb{P}^1}T)
    \right)
    $$
    is generated by
    $$
    \left\{
    \begin{array}{ll}
    \eta^{\frac{k-1}{2}},\
    \langle h,\phi_a^*h,f\rangle\cup\eta^{\frac{k-3}{2}},\
    h^2\cup\eta^{\frac{k-5}{2}},
    & \textup{if $k$ is odd and $k\neq m-1$,}\\[2mm]
    H^{m-2}_{\operatorname{prim}}(T),\
    \eta^{\frac{k-1}{2}},\
    \langle h,\phi_a^*h,f\rangle\cup\eta^{\frac{k-3}{2}},\
    h^2\cup\eta^{\frac{k-5}{2}},
    & \textup{if $k$ is odd and $k=m-1$,}\\[2mm]
    0,
    & \textup{if $k$ is even and $k\neq m-1$,}\\[2mm]
    H^{m-2}(T),
    & \textup{if $k$ is even and $k=m-1$.}
    \end{array}
    \right.
    $$

\end{enumerate}
Here $\langle h,\phi_a^*h,f\rangle^\perp$ denotes the orthogonal complement with respect to the intersection pairing on $H^2(S)$.
\end{proposition}

Assuming Proposition \ref{tedious}, parts (1) and (3) show that, for $k\leq m-1$,
$$
\operatorname{ker}\left(
H^{k-1}(S\times_{\mb{P}^1}T)
\xrightarrow{-\gamma}
H^{k+1}(X_1)\oplus H^{k+1}(X_2)
\right)
$$
and
$$
\operatorname{im}\left(
H^{k-1}(X_1)\oplus H^{k-1}(X_2)
\xrightarrow{\theta}
H^{k-1}(S\times_{\mb{P}^1}T)
\right)
$$
are linearly independent and together span
$H^{k-1}(S\times_{\mb{P}^1}T)$. Therefore,
$$
\operatorname{ker}\left(
H^{k-1}(S\times_{\mb{P}^1}T)
\xrightarrow{-\gamma}
H^{k+1}(X_1)\oplus H^{k+1}(X_2)
\right)
$$
and
$$
\operatorname{coker}\left(
H^{k-1}(X_1)\oplus H^{k-1}(X_2)
\xrightarrow{\theta}
H^{k-1}(S\times_{\mb{P}^1}T)
\right)
$$
are isomorphic via the map induced by the identity.

It remains to consider $k=m$. By Theorem \ref{criND}, it suffices to prove that
$$
\operatorname{im}\left(
H^{m-1}(X_1)\oplus H^{m-1}(X_2)
\xrightarrow{\theta}
H^{m-1}(S\times_{\mb{P}^1}T)
\right)
$$
is non-degenerate with respect to the cup-product pairing on
$H^{m-1}(S\times_{\mb{P}^1}T)$. If $m$ is even, this image is zero by Proposition \ref{tedious}(4), and hence is non-degenerate. We may therefore assume that $m$ is odd.

Set
$
D=S\times_{\mb{P}^1}T.
$
We regard $D$ as a hypersurface of
$
S\times\mb{P}^{m-2}
$
whose cohomology class is
$$
[D]=f+(m-1)\eta.
$$
Let
$$
G_2=
\operatorname{im}\left(
H^{m-1}(X_1)\oplus H^{m-1}(X_2)
\xrightarrow{\theta}
H^{m-1}(D)
\right).
$$
By Proposition \ref{tedious}(3),
$$
G_2=
\operatorname{Span}
\left\{
\eta^{\frac{m-1}{2}},
h\cup\eta^{\frac{m-3}{2}},
\phi_a^*h\cup\eta^{\frac{m-3}{2}},
f\cup\eta^{\frac{m-3}{2}},
h^2\cup\eta^{\frac{m-5}{2}}
\right\}.
$$
Moreover,
$$
G_2\subseteq
\operatorname{im}\left(
H^{m-1}(S\times\mb{P}^{m-2})
\xrightarrow{\theta}
H^{m-1}(D)
\right).
$$ 
Hence, by Proposition \ref{tool1}, for classes $\alpha,\beta\in G_2$, their cup-product pairing is computed in $S\times\mb{P}^{m-2}$ by
$$
Q_D(\alpha,\beta)
=
\int_{S\times\mb{P}^{m-2}}
(f+(m-1)\eta)\cup\alpha\cup\beta.
$$

Recall that
$$
h^2=1,\qquad h\cup f=3,\qquad f^2=0.
$$
Since $\phi_a$ is an automorphism over $\mb{P}^1$, we also have
$$
(\phi_a^*h)^2=1,
\qquad
\phi_a^*h\cup f=3.
$$
Furthermore, from
$$
\phi_a^*h
=
(27a^2+1)h
-(9a^2-3a)(e_1+e_2+e_3)
-9a^2(e_4+e_5+e_6)
-(9a^2+3a)(e_7+e_8+e_9),
$$
we obtain
$$
h\cup\phi_a^*h=27a^2+1.
$$
Therefore, with respect to the ordered basis
$$
\left\{
\eta^{\frac{m-1}{2}},
h\cup\eta^{\frac{m-3}{2}},
\phi_a^*h\cup\eta^{\frac{m-3}{2}},
f\cup\eta^{\frac{m-3}{2}},
h^2\cup\eta^{\frac{m-5}{2}}
\right\},
$$
the cup-product pairing on $G_2$ is represented by
$$
\left(
\begin{array}{ccccc}
0 & 3 & 3 & 0 & m-1 \\
3 & m-1 & (m-1)(27a^2+1) & 3(m-1) & 0 \\
3 & (m-1)(27a^2+1) & m-1 & 3(m-1) & 0 \\
0 & 3(m-1) & 3(m-1) & 0 & 0 \\
m-1 & 0 & 0 & 0 & 0
\end{array}
\right).
$$
Its determinant is
$$
-486a^2(m-1)^5,
$$
which is nonzero for every positive integer $a$. Hence $G_2$ is non-degenerate with respect to the cup-product pairing. This proves the monodromy weight condition in degree $m$, and therefore finishes the proof.

To compute the Hodge index on the middle cohomology of $X_m(a)$, we apply Theorem \ref{Corg2}. Recall that $$\operatorname{Gr}^{W^{St}}_{m}H^m(X_{\infty})==\frac{\textup{ker}\left(H^m(X_1)\oplus H^m(X_2)\xrightarrow{\theta} H^m(S\times_{\mb{P}^1}T)\right)}{\textup{im}\left(H^{m-2}(S\times_{\mb{P}^1}T)\xrightarrow{-\gamma}H^m(X_1)\oplus H^m(X_2)\right)}.$$ We will write $$H^{m}(X_1)\oplus H^{m}(X_2)= H^{m}(\mb{P}^2\times T)\oplus H^{m-2}(\Gamma_a)\oplus \left(\bigoplus_{i=1}^{a}H^{m-2}(F_i)\right)\oplus H^{m}(X_2)$$

We will prove

\begin{proposition} \label{tedious1}
Suppose $m$ is odd.
    \begin{enumerate}
        \item There is an orthogonal decomposition
$$
\operatorname{Gr}^{W^{St}}_{m+1}H^m(X_{\infty})
=
G_1\oplus W^\perp,
$$
where
$$
G_1
=
\langle h,\phi_a^*h,f\rangle^\perp
\cup\eta^{\frac{m-3}{2}},
$$
and $W^\perp$ denotes the orthogonal complement of
$$
W:=
\operatorname{im}\left(
R^{m-1}\rightarrow
H^{m-1}(S\times_{\mb{P}^1}T)
\right)
$$
with respect to the cup-product pairing on
$H^{m-1}(S\times_{\mb{P}^1}T)$. 

\item Moreover, the Hermitian pairing
$$
S(C\bullet,\overline{(-N)\bullet})
$$
is positive definite on $W^\perp$. Its restriction to $G_1$ is positive definite if
$
m\equiv 3\pmod 4,
$
and negative definite if
$
m\equiv 1\pmod 4.
$

\item  There is an orthogonal decomposition
$$
\operatorname{Gr}^{W^{St}}_{m}H^m(X_{\infty})
=
U'\oplus U
\oplus L^{\frac{m-3}{2}}H^1(\Gamma_a)
\oplus \bigoplus_{i=1}^{a}H_{\operatorname{prim}}^{m-2}(F_i)
$$$$\oplus \bigoplus_{i=1}^{a}L^{\frac{m-3}{2}}H^1(F_i).
$$
Here
$$
U'
\subset
\left(H^2(\mb{P}^2)\otimes H^{m-2}(T)\right)
\oplus
\left(H^0(C_a)\otimes H^{m-2}(T)\right)
\subset H^m(X_1)
$$
is the subspace
$$
U'
=
\left\{
(-6h\cup\eta,\eta)
\ \middle|\
\eta\in H^{m-2}(T)
\right\}.
$$
Moreover, $L$ denotes the Lefschetz operators on
$H^\bullet(\Gamma_a)$ and $H^\bullet(F_i)$ with respect to the polarizations induced from $X_1$. $U$ denotes the orthogonal complement of
$
H^0(C_a)\otimes H^{m-2}(T)
$
in $H_{\operatorname{prim}}^{m-2}(\Gamma_a)$,

\item The Hermitian pairing
$
S(C\bullet,\overline{\bullet})
$
is positive definite on
$$
 U\oplus
\bigoplus_{i=1}^{a}H_{\operatorname{prim}}^{m-2}(F_i), 
$$ and is negative definite on $U'.$
On the other hand, its restriction to
$$
L^{\frac{m-3}{2}}H^1(\Gamma_a)
\oplus
\bigoplus_{i=1}^{a}L^{\frac{m-3}{2}}H^1(F_i)
$$
is positive definite if
$
m\equiv 3\pmod 4,
$
and negative definite if
$
m\equiv 1\pmod 4.
$
    \end{enumerate}
\end{proposition}        
  
 \begin{proposition} \label{tedious2}
Suppose that $m$ is even. Then the following hold.
\begin{enumerate}
    \item We have
    $$
    \operatorname{Gr}^{W^{St}}_{m+1}H^m(X_{\infty})
    =
    H^{m-1}(S\times_{\mb{P}^1}T)
    =
    H^{m-1}_{\operatorname{prim}}(S\times_{\mb{P}^1}T).
    $$
    In particular, the Hermitian pairing
    $
    S(C\bullet,\overline{(-N)\bullet})
    $
    is positive definite on
    $\operatorname{Gr}^{W^{St}}_{m+1}H^m(X_{\infty})$.

    \item There is an orthogonal decomposition
    $$
    \operatorname{Gr}^{W^{St}}_{m}H^m(X_{\infty})
    =
    U_1\oplus U_2\oplus U_3\oplus U_4\oplus U_5,
    $$
    where $U_1$ is the orthogonal complement of
    $
    H^0(C_a)\otimes H^{m-2}(T)
    $
    in $H^{m-2}(\Gamma_a)$, and
    $$
    U_2
    =
    \bigoplus_{i=1}^{a}H^{m-2}_{\operatorname{prim}}(F_i).
    $$
    Moreover,
    $$
    U_3
    \subset
    \left(H^2(\mb{P}^2)\otimes H^{m-2}(T)\right)
    \oplus
    \left(H^0(C_a)\otimes H^{m-2}(T)\right)
    \subset H^m(X_1)
    $$
    is the subspace
    $$
    U_3
    =
    \left\{
    (-6h\cup\alpha,\alpha)
    \ \middle|\
    \alpha\in H^{m-2}_{\operatorname{van}}(T)
    \right\},
    $$
    where
    $
    H^{m-2}_{\operatorname{van}}(T)
    :=
    \operatorname{ker}\left(
    H^{m-2}(T)
    \rightarrow
    H^m(\mb{P}^1\times\mb{P}^{m-2})
    \right).
    $

    Let
    $
    \phi_j:F_j\hookrightarrow S\times_{\mb{P}^1}T,
    \  j=1,\ldots,a,
    $
    denote the natural embeddings. Then
    $$
    U_4
    \subset
    \bigoplus_{i=1}^{a}H^{m-2}(F_i)
    $$
    is the subspace
    $$
    U_4
    =
    \operatorname{Span}_{\C}
    \left\{
    \begin{array}{l}
    \phi_j^*\eta^{\frac{m-2}{2}}
    -
    \phi_k^*\eta^{\frac{m-2}{2}},\\[1mm]
    \phi_j^*\!\left(h\cup\eta^{\frac{m-4}{2}}\right)
    -
    \phi_k^*\!\left(h\cup\eta^{\frac{m-4}{2}}\right)
    \end{array}
    \ \middle|\
    j\neq k
    \right\}.
    $$

    Finally,
    $$
    U_5
    \subset
    H^m(\mb{P}^2\times T)
    \oplus
    \left(H^0(C_a)\otimes H^{m-2}(T)\right)
    \oplus
    H^m(X_2).
    $$
    If $m\geq 6$, then $U_5$ is the six-dimensional subspace
    $$
    \begin{aligned}
    U_5=\operatorname{Span}_{\C}\bigg\{&
    \left(\eta^{\frac{m}{2}},0,\eta^{\frac{m}{2}}\right),
    \left(f\cup\eta^{\frac{m-2}{2}},0,
    f\cup\eta^{\frac{m-2}{2}}\right),\left((3h-f)\cup h\cup\eta^{\frac{m-4}{2}},0,0\right),\\
    &\left(0,0,(3h-f)\cup h\cup\eta^{\frac{m-4}{2}}\right),\left(h^2\cup f\cup\eta^{\frac{m-6}{2}},0,0\right),\\
    &\left((3h-(a-2)f)\cup\eta^{\frac{m-2}{2}},
    -\eta^{\frac{m-2}{2}},
    -3h\cup\eta^{\frac{m-2}{2}}\right)
    \bigg\}.
    \end{aligned}
    $$
    If $m=4$, then $U_5$ is the four-dimensional subspace
    $$
    \begin{aligned}
    U_5=\operatorname{Span}_{\C}\bigg\{&
    \left(\eta^{\frac{m}{2}},0,\eta^{\frac{m}{2}}\right),\left((3h-f)\cup h\cup\eta^{\frac{m-4}{2}},0,0\right),\left(0,0,(3h-f)\cup h\cup\eta^{\frac{m-4}{2}}\right),\\ &\left((3h-(a-2)f)\cup\eta^{\frac{m-2}{2}},
    -\eta^{\frac{m-2}{2}},
    -3h\cup\eta^{\frac{m-2}{2}}\right)
    \bigg\}.
    \end{aligned}
    $$

    \item The Hermitian pairing
    $
    S(C\bullet,\overline{\bullet})
    $
    is positive definite on
    $$
    U_1\oplus U_2,
    $$ and is negative definite on $U_3.$
    Its restriction to $U_4$ has signature
    $
    (a-1,a-1).
    $
    Moreover, its restriction to $U_5$ has signature $(3,3)$ if
    $m\geq 6$, and signature $(2,2)$ if $m=4$.
\end{enumerate}
\end{proposition}

Assuming Propositions \ref{tedious1} and \ref{tedious2}, the negative directions of the Hermitian pairing can be read off directly from the above decompositions.

Suppose first that $m$ is odd. The negative-definite summand
$$
U'\cong H^{m-2}(T)(-1)
$$
contributes
$$
\nu_k:=h^{k-1,m-k-1}(T)
$$
negative directions to $H^{k,m-k}(X_m(a))$ for every $k$. If, in addition,
$
m\equiv1\pmod4
$
and
$
k=\frac{m\pm1}{2},
$
then Proposition \ref{tedious1}(2) contributes $7$ further negative directions from $G_1$, while the Lefschetz summands in Proposition \ref{tedious1}(4) contribute
$
g(C_a)+a
$
negative directions. Hence, in these two Hodge components, the total number of negative directions is
$$
\nu_k+g(C_a)+a+7
=
\nu_k+9(27a^2-2a+5)+a+6.
$$
For all other $k$, the only negative contribution is $\nu_k$.

Suppose now that $m$ is even. The negative-definite summand $U_3$ contributes
$$
\nu_k^{\operatorname{van}}
:=
h_{\operatorname{van}}^{k-1,m-k-1}(T)
$$
negative directions to $H^{k,m-k}(X_m(a))$. If
$
k\neq\frac{m}{2},
$
then
$
\nu_k^{\operatorname{van}}=\nu_k.
$
For the middle Hodge component, we have
$
\nu_{\frac m2}^{\operatorname{van}}
=
\nu_{\frac m2}-2.
$
Moreover, Proposition \ref{tedious2}(3) shows that $U_4$ has signature
$
(a-1,a-1),
$
while $U_5$ has signature $(3,3)$ if $m\geq6$ and $(2,2)$ if $m=4$. Thus, when $m\geq6$, the total number of negative directions in the middle Hodge component is
$$
\left(\nu_{\frac m2}-2\right)+(a-1)+3
=
\nu_{\frac m2}+a,
$$
whereas for $m=4$ it is
$$
(\nu_2-2)+(a-1)+2
=
\nu_2+a-1.
$$
This completes the proof.
 \end{proof}

\section{Proof of Theorem \ref{thmcks2} and \ref{thmcks}}

\subsection{Setup and a Key Lemma}

Recall that in Section $3$, we consider the following two situations:
\begin{enumerate}
    \item[(\textbf{A$'$})] We are given a mixed Hodge structure $(H, W, F)$ defined over $\R$ and a nilpotent endomorphism $N$ on $H$ such that \begin{enumerate}
        \item $N$ is a morphism of mixed Hodge structure of type $(-1,-1)$, 
        \item $W=W(N,d)$ is the weight filtration of $N$ centered at $d$.
    \end{enumerate}
    \item[(\textbf{B$'$})] Besides the assumption in (\textbf{A$'$}), we are  given a non-degenerate bilinear form $S$ on $H$ defined over $\R$, which satisfies \begin{enumerate}
    \item $S(v,u)=(-1)^dS(u,v)$,
    \item  $S(u,Nv)+S(Nu,v)=0$,
    \item $S(F^{p},F^{d-p+1})=0$ for all $p$.
\end{enumerate}
\end{enumerate}

In the situation (\textbf{B$'$}), since $W=W(N,d),$ we must have $S(W_a,W_b)=0$ for $a+b\leq 2d-1$ (see e.g. \cite[Lemma 6.4]{schmid}).  Therefore, for $p+q\geq d$, we can define $(s_{+}^{p,q},s_{-}^{p,q})$ to be the signature of $S(C\bullet, N^{p+q-d}\overline{\bullet})$ on the $(p,q)$-component of the Hodge structure (of weight $p+q$) on the primitive part $$P_{p+q} = \text{ker}(N^{p+q-d+1}:\text{Gr}^{W}_{p+q}H\rightarrow \text{Gr}^{W}_{2d-p-q-2}H),$$ where $C$ is the Weil operator.

Consider Deligne's weak splitting of real mixed Hodge structure (cf. \cite[Lemma-Definition 3.4]{PS}):
$$I^{p,q}:=F^p\cap W_{p+q}\cap (\overline{F^q}\cap W_{p+q}+\sum_{j\geq 2}\overline{F^{q-j+1}}\cap W_{p+q-j}).$$
\begin{proposition}
    In the situation (\textbf{A$'$}), this splitting satisfies the following conditions:
\begin{enumerate}
    \item $W_l=\bigoplus_{p+q\leq l} I^{p,q},\ F^r=\bigoplus_{p\geq r}I^{p,q}$.
    \item $I^{p,q}\equiv \overline{I^{q,p}} \text{ mod } \bigoplus_{r\leq p-1,s\leq q-1} I^{r,s}$.
    \item $N I^{p,q}\subset I^{p-1,q-1}.$
\end{enumerate}

In the situation (\textbf{B$'$}), Deligne's splitting also satisfies
\begin{enumerate}
 \setcounter{enumi}{3}
    \item $S(I^{p,q}, I^{r,s})=0$ unless $r=d-p, \ s=d-q.$ 
    \item If we define the primitive part of $I^{p,q}$ to be $$I_{\text{prim}}^{p,q}:=I^{p,q}\cap \text{ker}N^{p+q-d+1},$$ then $(\sqrt{-1})^{p-q}S(\bullet, N^{p+q-d}\overline{\bullet})$ has signature $(s_{+}^{p,q},s_{-}^{p,q})$ on $I^{p,q}_{\text{prim}}$.
\end{enumerate}

\end{proposition}

\begin{proof}

For (1), (2), see e.g. \cite{CKS}[(2.11)-(2.14)]. (3) is an immediate consequence of $NW_l\subset W_{l-2} $ and $NF^{r}\subset F^{r-1}$. (5) follows from (1) and the definition of the signature $(s_{+}^{p,q},s_{-}^{p,q})$. 

Let's prove (4) here: Since $S(W_a,W_b)=0$ for $a+b\leq 2d-1$, we have $S(I^{p,q}, I^{r,s})=0$ if $r+s+p+q\leq 2d-1$. Also, since  $S(F^a,F^b)=0$ for $a+b\geq d+1$ by assumption, we have $S(I^{p,q}, I^{r,s})=0$ if $r+p\geq d+1$. Now, suppose that $r+p\leq d$, we claim that $S(I^{p,q}, I^{r,s})=0$ for $q+s\geq d+1$ too. Were this proved, the only nonzero possibility is $r=d-p$ and $s=d-q$. To see this claim, write $$I^{p,q}:=F^p\cap W_{p+q}\cap (\overline{F^q}\cap W_{p+q}+\sum_{j\geq 2}\overline{F^{q-j+1}}\cap W_{p+q-j})$$ and $$I^{r,s}:=F^r\cap W_{r+s}\cap (\overline{F^s}\cap W_{r+s}+\sum_{i\geq 2}\overline{F^{s-i+1}}\cap W_{r+s-i}).$$ Since $S$ is defined over $\mathbb{R}$, we have $S(\overline{F^a},\overline{F^b})=0$ whenever $a+b\geq d+1$. Combining this with the identity $S(W_a,W_b)=0$ for $a+b\leq 2d-1$, we have $$S(\overline{F^{q-j+1}}\cap W_{p+q-j},\overline{F^{s-i+1}}\cap W_{r+s-i})=0$$ if either $q-j+s-i\geq d-1$ or $p+q+r+s-i-j\leq 2d-1$; or equivalently $$S(\overline{F^{q-j+1}}\cap W_{p+q-j},\overline{F^{s-i+1}}\cap W_{r+s-i})$$ is nonzero only if $p+r\geq d+2$, which is empty under our assumption. Similarly, we can see that $$S(\overline{F^{q}}\cap W_{p+q},\overline{F^{s-i+1}}\cap W_{r+s-i})$$ and $$S(\overline{F^{q-j+1}}\cap W_{p+q-j},\overline{F^{s}}\cap W_{r+s})$$ are nonzero only if $p+r\geq d+1$, which is also empty. Finally, we have $$S(\overline{F^{q}}\cap W_{p+q},\overline{F^{s}}\cap W_{r+s})=0$$ for $q+s\geq d+1$. The claim is proved.
\end{proof}

The key point of our calculation is the following combinatorial lemma: 

\begin{lemma}\label{lmac}
Let $u\in V$ be a vector and $N$ be a nilpotent operator on $V$. Suppose that $N^{n+1}u=0$ but $N^{n}u\neq 0$. Then for all $0\leq k\leq n+1$ we have $$\left( e^{zN}u\wedge e^{zN}Nu\wedge...\wedge e^{zN}N^{n-k}u \right)\wedge \left(e^{\overline{z}N}u\wedge e^{\overline{z}N}Nu\wedge...\wedge e^{\overline{z}N}N^{k-1}u\right)$$
$$=\frac{C_{n-k+1,k}}{((n-k+1)k)!}(\overline{z}-z)^{(n-k+1)k}u\wedge Nu\wedge...\wedge N^nu,$$ where $C_{n-k+1,k}$ is the number of standard Young tableaux of the shape $k \times (n-k+1)$ squares.

\end{lemma}

\begin{proof}
Let $$\Phi=\left( e^{zN}u\wedge e^{zN}Nu\wedge...\wedge e^{zN}N^{n-k}u \right)\wedge \left(e^{\overline{z}N}u\wedge e^{\overline{z}N}Nu\wedge...\wedge e^{\overline{z}N}N^{k-1}u\right).$$ We first show that $\Phi$ is a homogeneous polynomial in $z$ and $\overline{z}$ of degree $(n-k+1)k$. To see this, note that the difference of the degree in $z$ (resp. $\overline{z}$) and the degree in $N$ in each term of $e^{zN}N^lu$ (resp. $e^{\overline{z}N}N^lu$) is $l$, so the difference of the degree in $z,\ \overline{z}$ and the degree in $N$ in $\Phi$ will be $\frac{(n-k+1)(n-k)}{2}+\frac{k(k-1)}{2}$. Also, $\Phi$ is a polynomial in $z$ and $\overline{z}$ over $u\wedge Nu \wedge...\wedge N^nu$, and the degree of $u\wedge Nu \wedge...\wedge N^nu$ in $N$ is $\frac{(n+1)n}{2}$. So the degree of $\Phi$ in $z$ and $\overline{z}$ should be $$\frac{(n+1)n}{2}-\frac{(n-k+1)(n-k)}{2}-\frac{k(k-1)}{2}=(n-k+1)k.$$ 
Next we compute the limit $\lim_{\overline{z}\rightarrow z}\frac{\Phi}{(\overline{z}-z)^{k(n-k+1)}}$ with respect to the basis $u\wedge Nu\wedge...\wedge N^nu$. It suffices to show that this limit is finite and has a coefficient equal to $\frac{C_{n-k+1,k}}{((n-k+1)k)!}$. In fact, since this limit is in indefinite form, we may apply the L'Hospital rule (with respect to a partial derivative in $\overline{z}$). Observe that at each time we take the partial derivative to $\Phi$ with respect to $\overline{z}$, we replace one $e^{\overline{z}N}N^lu$ by $e^{\overline{z}N}N^{l+1}u$. Thus, the derivative remains in the indefinite form until $k(n-k+1)$-th partial derivative. Therefore, it suffices to prove that $$\left(\frac{\partial}{\partial\overline{z}}\right)^{k(n-k+1)}\left(e^{\overline{z}N}u\wedge e^{\overline{z}N}Nu\wedge...\wedge e^{\overline{z}N}N^{k-1}u\right)$$$$=C_{n-k+1,k}\left(N^{n-k+1}u\wedge N^{n-k+2}u\wedge...\wedge N^{n}u\right)$$

To see this equality, note that at the stage of $k(n-k+1)$-th partial derivative, the original factor  $$e^{\overline{z}N}u\wedge e^{\overline{z}N}Nu\wedge...\wedge e^{\overline{z}N}N^{k-1}u$$ has been replaced by $$e^{\overline{z}N}N^{n-k+1}u\wedge e^{\overline{z}N}N^{n-k+2}u\wedge...\wedge e^{\overline{z}N}N^{n}u=N^{n-k+1}u\wedge N^{n-k+2}u\wedge...\wedge N^{n}u,$$ and the coefficient is the number of ways to transform the index set $\{1,...,k-1\}$ into $\{n-k+1,...,n\}$ by successively replacing an index $l$ by $l+1$, subject to the condition that no repeated indices occur at any intermediate step. Such sequences of replacements are naturally identified with standard Young tableaux of rectangular shape $k \times (n-k+1)$ squares. This proves the lemma.

\end{proof}

\begin{corollary}\label{cor1}
Consider the $(n+1)\times (n+1)$ upper triangular matrix $A(x)=A_{0,n+1}(x)$ which is defined using the Taylor expansion of $e^x$ by 
$$ A_{0,n+1}(x) = \left(
\begin{array}{ccccc}
1 & \frac{x}{1!} &  \cdots & \frac{x^{n-1}}{(n-1)!} & \frac{x^n}{n!} \\ 
& 1 & \frac{x}{1!} & \cdots & \frac{x^{n-1}}{(n-1)!} \\
 & &  \ddots & \ddots &  \vdots\\
  &  &  & 1 & \frac{x}{1!} \\
 &  &  &  & 1  \\
 \end{array}
 \right).$$ For $0\leq k\leq n+1$, let $A_{n-k+1,k}(x)$ be the upper right $k\times k$ submatrix of $A(x)$, i.e. $$ A_{n-k+1,k}(x) = \left(
\begin{array}{ccccc}
\frac{x^{n-k+1}}{(n-k+1)!} & \frac{x^{n-k+2}}{(n-k+2)!} &  \cdots & \frac{x^{n-1}}{(n-1)!} & \frac{x^n}{n!} \\ 
\frac{x^{n-k}}{(n-k)!} & \frac{x^{n-k+1}}{(n-k+1)!} & \cdots &\frac{x^{n-2}}{(n-2)!} & \frac{x^{n-1}}{(n-1)!} \\
 \vdots &  \ddots&  \ddots & \ddots &  \vdots\\
  &  & \cdots  & \frac{x^{n-k+1}}{(n-k+1)!} & \frac{x^{n-k+2}}{(n-k+2)!} \\
 &  &  & \cdots  & \frac{x^{n-k+1}}{(n-k+1)!}  \\
 \end{array}
 \right).$$ Then the preceding lemma says that $$\textup{det}(A_{n-k+1,k}(x))=\frac{C_{n-k+1,k}}{((n-k+1)k)!}x^{(n-k+1)k},$$ where $C_{n-k+1,k}$ is the number of standard Young tableaux of the shape $k \times (n-k+1)$ squares.
\end{corollary}

\begin{proof} The variables $z$ and $\overline{z}$ in Lemma \ref{lmac} are formal variables, so we may specialize them to $z=0$ and $\overline{z}=x$. Expressing the vectors $$ e^{\overline{z}N}u,\ldots,e^{\overline{z}N}N^{k-1}u $$ as linear combinations of the ordered basis $$ u,Nu,\ldots,N^n u $$ and making this specialization, the formula in Lemma \ref{lmac} becomes $$ \det\begin{pmatrix} I_{(n-k+1)\times(n-k+1)} & 0_{(n-k+1)\times k} \\ \multicolumn{2}{c}{B_{k\times(n+1)}(x)} \end{pmatrix} = \frac{C_{n-k+1,k}}{\bigl((n-k+1)k\bigr)!} x^{(n-k+1)k}, $$ where $B_{k\times(n+1)}(x)$ is the submatrix formed by the first $k$ rows of $A_{0,n+1}(x)$. By the definition of $A_{n-k+1,k}(x)$, the matrix on the left-hand side is precisely $A_{n-k+1,k}(x)$. Therefore, $$ \det\bigl(A_{n-k+1,k}(x)\bigr) = \frac{C_{n-k+1,k}}{\bigl((n-k+1)k\bigr)!} x^{(n-k+1)k}. $$ \end{proof}

\subsection{Proof of Theorem \ref{thmcks2}}
\begin{proof}[Proof of Theorem \ref{thmcks2}]

For $p+q\geq d$, we choose a basis $\{u_i^{p,q}|i\in J_{\textup{prim}}^{p,q}\}$ of $I_{\textup{prim}}^{p,q}$ for each $p,q$, where $J^{p,q}_{\textup{prim}}$ is an (ordered) indexed set, such that $J^{p,q}_{\textup{prim}}=J^{q,p}_{\textup{prim}}$ and $\overline{u_i^{p,q}}=u_i^{q,p}+s_i^{q,p}$ with $s_i^{q,p}\in W_{q+p-2}$. Using the weight filtration of $N$, we may extend this basis to a basis $\{u_i^{p,q}|i\in J^{p,q}\}$ of $I^{p,q}$, where $J^{p,q}=\coprod_{j\geq 0}J_{\textup{prim}}^{p+j,q+j}$ for $p+q\geq d$ and $J^{p,q}=\coprod_{j\geq d-p-q}J_{\textup{prim}}^{p+j,q+j}$ for $p+q<d$, such that $$ N^lu_i^{p,q}= u_i^{p-l,q-l}$$ for all $i\in J^{p,q}_{\textup{prim}}$ and $0\leq l\leq p+q-d$, and such that $$\overline{u_i^{p,q}}=u_i^{q,p}+s_i^{q,p},\ \ \ s_i^{q,p}\in W_{q+p-2}$$ for all $p,q$ and $i\in J^{p,q}$.

We would like to show that $\textup{exp}({zN})F^{k}\oplus \overline{\textup{exp}(zN)F^{d-k+1}}=H$ for $\textup{Im}(z)$ sufficiently large. Without loss of generality, we may assume $k\leq d-k+1$. Also, since $$\textup{exp}({(z+1)N})F^{k}\oplus \overline{\textup{exp}((z+1)N)F^{d-k+1}}$$$$=\textup{exp}(N)\left(\textup{exp}({zN})F^{k}\oplus \overline{\textup{exp}(zN)F^{d-k+1}}\right),$$ we may and do assume $z$ lies in a domain with bounded real parts. It suffices to show that $$\frac{(\bigwedge_{p\geq k,\  i\in J^{p,q}}e^{zN}u_i^{p,q})\wedge(\bigwedge_{p\geq d-k+1,\  i\in J^{p,q}}\overline{e^{zN}u_i^{p,q}})}{\bigwedge_{p,q\geq 0,\ i\in J^{p,q}} u_i^{p,q}}\neq 0$$ for $\textup{Im}(z)$ sufficiently large.

Let $u_i^{p,q}\in I^{p,q}_{\textup{prim}}$ with $i\in J^{p,q}_{\textup{prim}}$ and $p\geq k,\ q\geq d-k+1$, we first compute $$\Phi_i^{p,q}:=\left(\left(\bigwedge_{0\leq l\leq p-k} e^{zN}u^{p-l,q-l}_i\right) \wedge \left(\bigwedge_{0\leq l\leq q-d+k-1} e^{\overline{z}N}\overline{u^{q-l,p-l}_i}\right)\right).$$ 
$$=\left(\bigwedge_{0\leq l\leq p-k} e^{zN}N^lu^{p,q}_i\right) \wedge \left(\bigwedge_{0\leq l\leq q-d+k-1} e^{\overline{z}N}(N^lu_i^{p,q}+s_i^{p-l,q-l})\right)$$
Then by the preceding lemma, we can write $$\Phi^{p,q}_i=\Psi^{p,q}_i+\Xi_{i}^{p,q},$$ where $$\Psi^{p,q}_i=\frac{C_{p-k+1,q-d+k}}{((p-k+1)(q-d+k))!}(\overline{z}-z)^{(p-k+1)(q-d+k)}u_i^{p,q}\wedge u_i^{p-1,q-1}\wedge...\wedge u_i^{d-q,d-p},$$  and

$$\Xi_{i}^{p,q}= \left(\bigwedge_{0 \leq l\leq p-k }e^{zN}u_i^{p-l,q-l}\right)\wedge\left( \sum_{0\leq l\leq q-d+k-1} e^{\overline{z}N}u_i^{p,q} \wedge  \right....\wedge e^{\overline{z}N}u_i^{p-l+1,q-l+1}\wedge $$$$  e^{\overline{z}N}s_i^{p-l,q-l}\wedge  e^{\overline{z}N}(u_i^{p-l-1,q-l-1}+s_i^{p-l-1,q-l-1}) \wedge ...\wedge$$$$\left.  e^{\overline{z}N}(u_i^{p-q+d-k+1,d-k+1}+s_i^{p-q+d-k+1,d-k+1})\right) $$

On the other hand, for $u_i^{p,q}\in I^{p,q}_{\textup{prim}}$ with $i\in J^{p,q}_{\textup{prim}}$ and $p\geq k,\ q< d-k+1$. We have $$\Phi_i^{p,q}:=\bigwedge_{0\leq l\leq p+q-d} e^{zN}N^lu^{p,q}_i=u_i^{p,q}\wedge u_i^{p-1,q-1}\wedge...\wedge u_i^{d-q,d-p}.$$

If the mixed Hodge structure splits over $\R$, i.e. $s_i^{p,q}=0$ for all $i,p,q$, then $\Xi_{i}^{p,q}=0$ too. Then we have $$\frac{(\bigwedge_{p\geq k,\  i\in J^{p,q}}e^{zN}u_i^{p,q})\wedge(\bigwedge_{p\geq d-k+1,\  i\in J^{p,q}}\overline{e^{zN}u_i^{p,q}})}{\bigwedge_{p,q\geq 0,\ i\in J^{p,q}} u_i^{p,q}}$$$$= C (\overline{z}-z)^{\sum_{p\geq k, q\geq d-k+1 }|J^{p,q}_{\textup{prim}}|(p-k+1)(q-d+k)},\ \ C\neq 0,$$ which is nonzero for $\text{Im}(z)\neq 0$ and the theorem is proved.

In the general case, we must estimate the remainder terms arising from $\Xi_i^{p,q}$. We first expand the exponential factors. Before rewriting the resulting vectors in the chosen basis of $H$, each contribution is a wedge product of factors of the form $N^l u_i^{p',q'}$ or $N^l s_i^{p',q'}$, multiplied by a monomial in $z$ and $\overline{z}$. We then express these factors as linear combinations of the basis elements of $H$ and decompose
$$
\Xi_i^{p,q}=\sum_{j_i}\Xi_{i,j_i}^{p,q},
$$
where each $\Xi_{i,j_i}^{p,q}$ is a wedge product of basis elements $u_{i'}^{p',q'}$ multiplied by a monomial in $z$ and $\overline{z}$. Here, $j_i$ records both the original wedge product of the factors $N^l u_i^{p',q'}$ and $N^l s_i^{p',q'}$ obtained before the basis expansion and the choice of a summand produced by that basis expansion. Consequently, if two different original wedge products yield the same wedge product of basis elements, we retain them as distinct summands $\Xi_{i,j_i}^{p,q}$ rather than combining them.

Consider two numbers associated to $\Xi_{i,j_{i}}^{p,q}$. The first number is the degree difference $$\sigma(\Xi_{i,j_{i}}^{p,q}):=\textup{deg}_{z,\ \overline{z}}(\Xi_{i,j_{i}}^{p,q})-\textup{deg}_{z,\ \overline{z}}(\Psi_i^{p,q}),$$ and the second number is the weight difference $$\tau(\Xi_{i,j_{i}}^{p,q}):=\textup{wt}(\Xi_{i,j_{i}}^{p,q})-\textup{wt}(\Psi_i^{p,q}),$$ where the weight $\textup{wt}$ of a wedge product of basis $u_{i'}^{p',q'}$ is defined to be the sum of the corresponding integers $p'+q'$. The key observation is that $$\tau+2\sigma \leq -1.$$ Roughly speaking, each increase of one in the degree with respect to $z$ and $\overline{z}$ results from an application of $N$, which lowers the weight by two. The additional decrease by at least one arises from the presence of a term $s_i^{p',q'}$.

We now make this computation explicit. By Lemma \ref{lmac}, $$\textup{wt}(\Psi^{p,q}_i)=d(p+q-d+1)$$ and $$\textup{deg}_{z,\overline{z}}(\Psi^{p,q}_i)=(p-k+1)(q-d+k).$$ Moreover, we have $$ \deg_{z,\overline{z}}\left(\Xi_{i,j_i}^{p,q}\right) = \deg_N\left(\Xi_{i,j_i}^{p,q}\right), $$  where $\deg_N\left(\Xi_{i,j_i}^{p,q}\right)$ denotes the total number of applications of $N$ in the term from which $\Xi_{i,j_i}^{p,q}$ arises after expanding the exponential factors. More precisely, after expanding the exponential factors of $\Xi_{i}^{p,q}$, each summand is obtained as a wedge product of factors of the form $N^l u_i^{p',q'}$ or $N^l s_i^{p',q'},$ and $\deg_N$ is the sum of the corresponding exponents $l$, and $\deg_N\left(\Xi_{i,j_i}^{p,q}\right)$ is defined to be $\deg_N$ of such summand it originated from. Also, $$\textup{wt}(\Xi^{p,q}_{i,j_{i}})\leq \sum_{0\leq l\leq p-k}\textup{wt}(u_i^{p-l,q-l})+  \sum_{0\leq l\leq q-d+k-1}\textup{wt}(u_i^{p-l,q-l})-2\textup{deg}_N(\Xi^{p,q}_{i,j_{i}})-1 $$$$= (k+q)(p-k+1)+(p+d-k+1)(q-d+k)-2\textup{deg}_N(\Xi^{p,q}_{i,j_{i}})-1.$$ Here, $-1$ appears due to the existence of $s_i^{p',q'}$, and that after expressing $s_i^{p',q'}$ as a linear combination of the basis $u_{i''}^{p'',q''}$, each summand has weight at most $p'+q'-1.$ Therefore, we have $$\tau(\Xi_{i,j_{i}}^{p,q})+2\sigma(\Xi_{i,j_{i}}^{p,q})\leq (k+q)(p-k+1)+(p+d-k+1)(q-d+k)-1 $$$$-d(p+q-d+1)-2(p-k+1)(q-d+k)=-1.$$

Finally, note that the terms in the top wedge formed by the wedge of some $\Xi_{i,j_{i}}^{p,q}$ and some $\Psi_{i}^{p,q}$ with at least one $\Xi_{i,j_{i}}^{p,q}$ must satisfy $$\sum \tau(\Xi_{i,j_{i}}^{p,q})=0.$$ Using the inequalities $$\tau(\Xi_{i,j_{i}}^{p,q})+2\sigma(\Xi_{i,j_{i}}^{p,q})\leq -1,$$ it implies that $$\sum \sigma(\Xi_{i,j_{i}}^{p,q})\leq -1.$$ So these terms have degree in $z,\ \overline{z}$ strictly less than the top wedge of $\Psi_{i}^{p,r}$. Therefore, if we write $z=a+\sqrt{-1}t$, we can conclude that 
 $$\frac{(\bigwedge_{p\geq k,\  i\in J^{p,q}}e^{zN}u_i^{p,q})\wedge(\bigwedge_{p\geq d-k+1,\  i\in J^{p,q}}\overline{e^{zN}u_i^{p,q}})}{\bigwedge_{p,q\geq 0,\ i\in J^{p,q}} u_i^{p,q}}$$
 $$= C (\overline{z}-z)^{\sum_{p\geq k, q\geq d-k+1 }|J^{p,q}_{\textup{prim}}|(p-k+1)(q-d+k)}+O(|t|^{-1+\sum_{p\geq k, q\geq d-k+1 }|J^{p,q}_{\textup{prim}}|(p-k+1)(q-d+k)}),$$ where $C\neq 0$ and $z$ lies in a domain with bounded real parts by our assumption.

This proves the theorem.

\end{proof}

\subsection{Proof of Theorem \ref{thmcks}}

\begin{proof}[Proof of Theorem \ref{thmcks}]
We adopt the notation in the proof of Theorem \ref{thmcks2}. In this case, since $$S(e^{(z+1)N}v,\overline{e^{(z+1)N}u})=S(e^{zN}v,\overline{e^{zN}u}),$$ we may and do assume $z$ lies in a domain with bounded real parts.

We choose the basis $u_i^{p,q}$ of $I^{p,q}_{\textup{prim}}$ again so that the Hermitian form $S(C\bullet, N^{p+q-d}\overline{\bullet})$ on $I^{p,q}_{\textup{prim}}$ is diagonalized under this basis. We also make the basis satisfies $$\overline{u_i^{p,q}}=u_i^{q,p}+s_i^{q,p},\ s_i^{q,p}\in W_{p+q+2},$$ which is possible by the condition (4) of Deligne's splitting. In this case, we also separate $J^{p,q}_{\textup{prim}}$ into the disjoint union of $J^{p,q}_{\textup{prim},+}$ and $J^{p,q}_{\textup{prim},-}$, where $u_i^{p,q}$ is a positive (resp. negative) eigenvector of $S(C\bullet, N^{p+q-d}\overline{\bullet})$ for $i\in J^{p,q}_{\textup{prim},+}$ (resp. $J^{p,q}_{\textup{prim},-}$). By assumption, $|J^{p,q}_{\textup{prim},\pm}|=s_{\pm}^{p,q}$. 

Consider the basis $\left\{N^r u_i^{p,q}\mid 0\leq p,q\leq d,\ i\in J^{p,q}_{\textup{prim}},\ 0\leq r\leq p+q-d  \right\}.$ We give a well-ordering on this indexed set $(p,q,i,r)$ by $(p,q,i,r)>(p',q',i',r')$ if 
\begin{enumerate}
    \item $p-r> p'-r'$,
    \item $p-r= p'-r'$ and $q-r> q'-r'$,
    \item $p-r= p'-r'$ and $q-r= q'-r'$ and $r> r'$,
    \item $p-r= p'-r'$ and $q-r= q'-r'$ and $r= r'$ and $i>i'$.
\end{enumerate}
Here, the well-ordering on $i\in J^{p,q}_{\textup{prim}}$ is chosen arbitrarily. We will prove that the decreasing filtration $\hat{F}^{\bullet}$ constructed using the well-ordered basis above satisfies the desired property in Theorem \ref{thmcks} (1).

On the other hand, we need to show that the signature of the Hermitian form $S(C\bullet, \overline{\bullet})$ on $\textup{exp}(zN)F^k\bigcap \overline{\textup{exp}(zN)F^{d-k}}$ becomes $$\left(\sum_{\substack{l\geq k-d, \\ r\geq(-l)_{+}} }s_{+}^{k+r,d+l-k+r}, \ \sum_{\substack{l\geq k-d, \\ r\geq(-l)_{+}} }s_{-}^{k+r,d+l-k+r}\right)$$ for $\textup{Im}(z)$ sufficiently large.

Besides the opposedness condition proved in Theorem \ref{thmcks2}, we notice that $\textup{exp}(zN)F^{\bullet}$ satisfies the first Hodge-Riemann bilinear relation: $$S(\textup{exp}(zN)F^{k},\textup{exp}(zN)F^{d-k+1})=S(F^k,F^{d-k+1})=0,$$ where we use the first Hodge-Riemann bilinear relation for $F^{\bullet}$ and the fact that $$S(e^{zN}u,e^{zN}v)=S(u,v)$$ for all $u,v$.

As a result, it suffices to prove that for $0\leq k\leq d,$ the signature of the Hermitian form  $(\sqrt{-1})^d S(\bullet,\overline{\bullet})$ on $\textup{exp}(zN)F^{k}$  is 
\begin{equation}\label{eq2}
\mathlarger{\mathlarger{\sum}}_{j=k}^{d}(-1)^{d-j}\left(\sum_{\substack{l\geq j-d, \\ r\geq(-l)_{+}} }s_{+}^{j+r,d+l-j+r}-\sum_{\substack{l\geq j-d, \\ r\geq(-l)_{+}} }s_{-}^{j+r,d+l-j+r}\right)
\end{equation} for $\textup{Im}(z)$ sufficiently large.

Let's start with the simplest case first. For all $p, q, k$  with  $d-q\leq k\leq p$ and any $i\in J^{p,q}_{\textup{prim},\pm}$, consider the subspace  $$V^{p,q}_{k.i}:=\bigoplus_{0\leq l\leq p-k} N^l\left(\C u^{p,q}_i\right)$$ of $F^k$. Write $z=a+\sqrt{-1}t$. We have $$S(e^{zN}N^r u_i^{p,q},e^{\overline{z}N}N^s \overline{ u_{i}^{p,q}})=(-1)^{r} S(u_i^{p,q},e^{(\overline{z}-z)N}N^{r+s}\overline{u_{i}^{p,q}})$$$$=\left\{\begin{array}{cc}
     (-1)^r\frac{(\overline{z}-z)^{p+q-d-r-s}}{(p+q-d-r-s)!}S(u_i^{p,q},N^{p+q-d}\overline{u_{i}^{p,q}})+O(|t|^{p+q-d-r-s-1}), & r+s\leq p+q-d  \\
    0 ,&  r+s>p+q-d
 \end{array}\right..$$

Then the following lemma holds:

\begin{lemma}\label{com}
For $d-q\leq k\leq p$, let $B^{p,q}_k(z,\overline{z})$ be the $(p-k+1)\times (p-k+1)$ Hermitian matrix $(b^{p,q}_{k,r,s})_{0\leq r,s\leq p-k}$ with $$b^{p,q}_{k,r,s}=\left\{\begin{array}{cc}
(\sqrt{-1})^{d-p+q}(-1)^r\frac{(\overline{z}-z)^{p+q-d-r-s}}{(p+q-d-r-s)!}, & r+s\leq p+q-d \\
0, & r+s>p+q-d
\end{array}\right. .$$ Then \begin{enumerate}
    \item We have $$\textup{det}\left(B_k^{p,q}(z,\overline{z})\right)
=\left((-1)^{\sum_{j=k}^p(d-j)}\right)\frac{C_{q-d+k,p-k+1}(2\textup{Im}(z))^{(q-d+k)(p-k+1)}}{((q-d+k)(p-k+1))!},$$ where $C_{q-d+k,p-k+1}$ is the number of standard Young tableaux of the shape $(q-d+k) \times (p-k+1)$ squares.
\item For $\textup{Im}(z)\neq 0$, if we perform the block diagonalization to $B^{p,q}_k(z,\overline{z})$ as in Lemma \ref{linear}, the diagonal matrix $D^{p,q}_k(z,\overline{z})$ becomes  $$\textup{diag}\left(...,(-1)^{d-l}\tilde{C}_l^{p,q}(\textup{Im}(z))^{q-d-p+2l},...\right)_{l=p,p-1,p-2,...,k},$$ where $$\tilde{C}_l^{p,q}=2^{q-d-p+2l}\left(\frac{C_{q-d+l,p-l+1}}{C_{q-d+l+1,p-l}}\right)\left(\frac{((q-d+l+1)(p-l))!}{((q-d+l)(p-l+1))!}\right).$$ In particular, the signature of $B^{p,q}_k(z,\overline{z})$ is $\sum_{j=k}^p(-1)^{d-j}$.
\end{enumerate}
\end{lemma}

Assuming Lemma \ref{com}, then the Hermitian matrix $$\left((\sqrt{-1})^dS(e^{zN}N^r u_i^{p,q},e^{\overline{z}N}N^s \overline{ u_{i}^{p,q}})\right)_{0\leq r,s\leq p-k}$$$$= \left((\sqrt{-1})^{p-q}S(u_i^{p,q},N^{p+q-d}\overline{u_{i}^{p,q}})\right)B_{k}^{p,q}(z,\overline{z})+R_{k,i}^{p,q}(z,\overline{z}),$$ where each entry of $R_{k,i}^{p,q}(z,\overline{z})$ is either $0$ or has degree (in $z,\overline{z}$) strictly less than the corresponding entry of $B_k^{p,q}(z,\overline{z})$. As a result, we see that \begin{enumerate}
    \item The restriction of the filtration $\hat{F}^{\bullet}$ on $V^{p,q}_{k,i}$ satisfies the desired condition in Theorem \ref{thmcks} (1) for $\textup{Im}(z)$ sufficiently large.
    \item If we perform the block diagonalization to the Hermitian matrix $$\left((\sqrt{-1})^dS(e^{zN}N^r u_i^{p,q},e^{\overline{z}N}N^s \overline{ u_{i}^{p,q}})\right)_{0\leq r,s\leq p-k}$$ as in Lemma \ref{linear}, the diagonal matrix becomes $$\textup{diag}\left(...,(-1)^{d-p+r}\tilde{C}_{p-r}^{p,q}\lambda_i^{p,q}(\textup{Im}(z))^{p+q-d-2r}+O(|t|^{p+q-d-2r-1}),...\right)_{0\leq r\leq p-k},$$ where $$\lambda_i^{p,q}=(\sqrt{-1})^{p-q}S(u_i^{p,q},N^{p+q-d}\overline{u_{i}^{p,q}}).$$  In particular, the signature of the Hermitian matrix is $\pm \left(\sum_{j=k}^p(-1)^{d-j}\right)$ where $\pm$ depends on whether $i\in J^{p,q}_{\textup{prim},+}$ or $i\in J^{p,q}_{\textup{prim},-}$. This shows that the formula (\ref{eq2}) holds in this simplest case.
\end{enumerate}

\begin{proof}[Proof of Lemma \ref{com}]
 Note that under suitable symmetry (the reflection about the middle column), the matrix $B_k^{p,q}(z,\overline{z})$ is related to the matrix $A_{q-d+k,p-k+1}(\overline{z}-z)$ that we define in Corollary \ref{cor1}, and then we have \begin{equation*}
\begin{split}
    \textup{det}\left(B_k^{p,q}(z,\overline{z})\right)&=(\sqrt{-1})^{(d-p+q)(p-k+1)}(-1)^{\frac{(p-k)(p-k+1)}{2}} \\
    &   \ \ \ \ \ \ \ \ \ \ \ \ \ \ (-1)^{\frac{(p-k)(p-k+1)}{2}}\textup{det}\left(A_{q-d+k,p-k+1}(\overline{z}-z)\right).
\end{split}
\end{equation*} The first $(-1)^{\frac{(p-k)(p-k+1)}{2}}$ corresponds to the $(-1)^r$ factor in each entry and the second $(-1)^{\frac{(p-k)(p-k+1)}{2}}$ corresponds the permutation of columns of $B_k^{p,q}(z,\overline{z})$ to match those of $A_{q-d+k,p-k+1}(\overline{z}-z)$. Thus, by  Corollary \ref{cor1} we have $$\textup{det}\left(B_k^{p,q}(z,\overline{z})\right)=(\sqrt{-1})^{(d-p+q)(p-k+1)}\frac{C_{q-d+k,p-k+1}}{((q-d+k)(p-k+1))!}(\overline{z}-z)^{(q-d+k)(p-k+1)}$$
$$=\left((-1)^{\sum_{j=k}^p(d-j)}\right)\frac{C_{q-d+k,p-k+1}(2\textup{Im}(z))^{(q-d+k)(p-k+1)}}{((q-d+k)(p-k+1))!}.$$

On the other hand, since for $k\leq l\leq p$, $B_l^{p,q}(z,\overline{z})$ is the leading principal $(p-l+1)\times (p-l+1)$ submatrix of $B_k^{p,q}(z,\overline{z})$, one sees in particular that the leading principal $(p-l+1)\times (p-l+1)$ submatrices of $B_k^{p,q}(z,\overline{z})$ are non-degenerate for $\textup{Im}(z)\neq 0$. Thus, by Lemma \ref{linear}, the second assertion follows. 
\end{proof}

 For all $p,q,r,p',q',r,s$, we define the number $$n_{p,q,p',q',r,s}=\textup{min}\{p+q-d-r-s,p'+q'-d-r-s\}.$$ Then for all $(p,q,r,i)$ and $(p',q',s,i')$, we have $$(\sqrt{-1})^dS(e^{zN}N^r u_i^{p,q},e^{\overline{z}N}N^s \overline{ u_{i'}^{p',q'}})$$
$$=(\sqrt{-1})^d(-1)^rS(u_i^{p,q},e^{(\overline{z}-z)N}N^{r+s} \overline{ u_{i'}^{p',q'}})$$ 
$$=\left\{\begin{aligned}[c]
   & \begin{multlined}
     (\sqrt{-1})^d (-1)^r\frac{(\overline{z}-z)^{n_{p,q,p',q',r,s}}}{n_{p,q,p',q',r,s}!}S(u_i^{p,q},N^{n_{p,q,p',q',r,s}+r+s}\overline{u_{i'}^{p',q'}})+O(|t|^{n_{p,q,p',q',r,s}-1}), \\\textup{ when } n_{p,q,p',q',r,s} \geq 0 \end{multlined} \\
 &0, \ \ \textup{ when } n_{p,q,p',q',r,s} <0.
 \end{aligned}\right.$$ We can also notice that $S(u_i^{p,q},N^{n_{p,q,p',q',r,s}+r+s}\overline{u_{i'}^{p',q'}})=0$ unless $(p,q,i)=(p',q',i')$.

If the mixed Hodge structure is split over $\R$, then all of the terms $O(|t|^{n_{p,q,p',q',r,s}-1})$ and $R^{p,q}_{k,i}$ above are zero. Moreover, for $(p,q,i)\neq (p',q',i')$, $V^{p,q}_{k,i}$ and $V^{p',q'}_{k,i'}$ are mutually orthogonal with respect to the Hermitian form $(\sqrt{-1})^d S(e^{zN}\bullet,\overline{e^{zN}\bullet})$. Hence, according to our computation for the restriction of the Hermitian form $(\sqrt{-1})^{d}S(e^{zN}\bullet,\overline{e^{zN}\bullet})$ on $V^{p,q}_{k,i}$, we can conclude that for all $\textup{Im}(z)>0$, the filtration $\hat{F}^{\bullet}$ satisfies the desired condition in Theorem \ref{thmcks} (1) and the signature of the Hermitian form  $(\sqrt{-1})^d S(\bullet,\overline{\bullet})$ on $\textup{exp}(zN)F^{k}$ is as in (\ref{eq2}), and the theorem is proved.

In general, consider the Hermitian matrix $$M=\left(M_{ (p,q,i,r), (p',q',i',s)}\right)_{(p,q,i,r), (p',q',i',s)}$$$$:=\left((\sqrt{-1})^d S(e^{zN}N^ru^{p,q}_i,\overline{e^{zN}N^su^{p',q'}_{i'}})\right)_{(p,q,i,r), (p',q',i',s)}$$ indexed by the quadruple $(p,q,i,r)$, where $p\geq k, \ i\in J_{\textup{prim}}^{p,q},  \ 0\leq r\leq \textup{max}\{p-k,p+q-d\}$ and $p'\geq k, \  i'\in J_{\text{prim}}^{p',q'}, \  0\leq s\leq \textup{max}\{p'-k,p'+q'-d\}.$ We will perform the block diagonalization to $M$ as in Lemma \ref{linear} with respect to the well-ordering on $(p,q,i,r)$ (from large to small). We want to show that the terms $O(|t|^{n_{p,q,p',q',r,s}-1})$ above don't contribute to the leading terms in the block diagonalization process. Theorem \ref{thmcks} (1) is proved if we show that the successive diagonal entries are nonzero. The signature formula (\ref{eq2}) follows if we can show that after the block diagonalization, the diagonal entry $M_{(p,q,i,r),(p,q,i,r)}$ of $M$ becomes $$(-1)^{d-p+r}\tilde{C}_{p-r}^{p,q}\lambda_i^{p,q}(\textup{Im}(z))^{p+q-d-2r}+O(|t|^{p+q-d-2r-1})$$ for all $(p,q,i,r)$.

We recall that the block diagonalization of a Hermitian matrix  $\left(\begin{array}{cc}
   A  & B \\
    B^* & C
\end{array}\right)$ with $A$ invertible is given by $$\left(\begin{array}{cc}
   A  & 0 \\
    0 & C-B^*A^{-1}B
\end{array}\right).$$ Let's call the terms coming from $-B^*A^{-1}B$ the supplementary terms to the corresponding entries of $C$. We will prove that the following situation holds by induction on $(p,q,i,r)$ with respect to the well-ordering from large to small:
 
 Before we perform the $(p,q,i,r)-$diagonalization step, we have \begin{enumerate}
    \item For $(p',q',i',r')$ and $(p'',q'',i'',r'')$ such that at least one of them $> (p,q,i,r)$ and such that  $(p',q',i',r') \neq (p'',q'',i'',r'')$, $$M_{(p',q',i',r'),(p'',q'',i'',r'')}=0.$$
    \item For $(p',q',i',r')\geq (p,q,i,r )$, $$M_{(p',q',i,,r'),(p',q',i',r')}=(-1)^{d-p'+r'}\tilde{C}_{p'-r'}^{p',q'}\lambda_i^{p',q'}(\textup{Im}(z))^{p'+q'-d-2r'}+O(|t|^{p'+q'-d-2r'-1}).$$
    \item For $(p',q',i',r')$ and $(p'',q'',i'',r'')$, both $\leq (p,q,i,r)$ and such that  $(p',q',i')\neq (p'',q'',i'')$, $$M_{(p',q',i',r'),(p'',q'',i'',r'')}=O(|t|^{n_{p',q',p'',q'',r',r''}-1}).$$ 
    \item For $(p',q',i',r')$ and $(p',q',i',r'')$, both $\leq (p,q,i,r)$, $$M_{(p',q',i',r'),(p',q',i',r'')}=O(|t|^{n_{p',q',p',q',r',r''}}).$$
    \end{enumerate}
    After we perform the $(p,q,i,r)-$diagonalization step, we have 
   \begin{enumerate}
 \setcounter{enumi}{4} 
    \item For $(p',q',i',r')$ and $(p'',q'',i'',r'')$, both $< (p,q,i,r)$, and such that either $(p',q',i')\neq (p,q,i)$ or $(p'',q'',i'')\neq (p,q,i)$, the supplementary term to the entry $M_{(p',q',i',r'),(p'',q'',i'',r'')}$ has order $$O(|t|^{n_{p',q',p'',q'',r',r''}-1}).$$
    \item For $(p,q,i,r')$ and $(p,q,i,r'')$, both $< (p,q,i,r)$, the supplementary term to $M_{(p,q,i,r'),(p,q,i,r'')}$ has order $$O(|t|^{n_{p,q,p,q,r',r''}}).$$
    
\end{enumerate}

Now, we assume that $(1), (2), (3), (4)$ hold at $(p,q,i,r)$-diagonalization step and $(5),(6)$ hold for all $(p',q',i',r')$-diagonalization step with $(p',q',i',r')>(p,q,i,r)$. Then for $(p',q',i',r')$ and $(p'',q'',i'',r'')$, both $< (p,q,i,r)$, and such that either $(p',q',i')\neq (p,q,i)$ or $(p'',q'',i'')\neq (p,q,i)$, the supplementary term to the entry $M_{(p',q',i',r'),(p'',q'',i'',r'')}$ is $$-M_{(p',q',i'r'),(p,q,i,r)}M_{(p,q,i,r),(p,q,i,r)}^{-1}M_{(p,q,i,r),(p'',q'',i'',r'')},$$ which has order less than or equal to \begin{equation*}
    \begin{split}
       &\ n_{p',q',p,q,r',r}-n_{p,q,p,q,r,r}+n_{p,q,p'',q'',r,r''}-1 \\
      \leq & \ n_{p',q',p'',q'',r',r''}-1. 
    \end{split}
\end{equation*} The negative one arises since either $(p',q',i')\neq (p,q,i)$ or $(p'',q'',i'')\neq (p,q,i)$. Also, for $(p,q,i,r')$ and $(p,q,i,r'')$, both $< (p,q,i,r)$, the supplementary term to the entry $M_{(p,q,i,r'),(p,q,i,r'')}$ is $$-M_{(p,q,i,r'),(p,q,i,r)}M_{(p,q,,i,r),(p,q,i,r)}^{-1}M_{(p,q,i,r),(p,q,i,r'')},$$ which has order less than or equal to \begin{equation*}
    \begin{split}
       &\ n_{p,q,p,q,r',r}-n_{p,q,p,q,r,r}+n_{p,q,p,q,r,r''} \\
      = & \ n_{p,q,p,q,r',r''}. 
    \end{split}
\end{equation*} Thus if we let $(\hat{p},\hat{q},\hat{i},\hat{r})$ be the predecessor of $(p,q,i,r)$, then we have seen $(5), (6)$ hold for $(p,q,i,r)$-diagonalization step and $(1),(3),(4)$ hold for $(\hat{p},\hat{q},\hat{i},\hat{r})$-diagonalization step. Finally, to see $(2)$ for $(\hat{p},\hat{q},\hat{i},\hat{r})$, we notice that by $(5)$ and $(6)$ for all the steps before $(\hat{p},\hat{q},\hat{i},\hat{r})$, the only contribution of order $n_{\hat{p},\hat{p},\hat{q},\hat{q},\hat{r},\hat{r}}$ to $M_{(\hat{p},\hat{q},\hat{i},\hat{r}),(\hat{p},\hat{q},\hat{i},\hat{r})}$ comes from those entries $M_{(\hat{p},\hat{q},\hat{i},r'),(\hat{p},\hat{q},\hat{i},r'')}$ with $r'$, $r''\leq\hat{r}$. As a consequence, by Lemma \ref{com}, we see that $$M_{(\hat{p},\hat{q},\hat{i},\hat{r}),(\hat{p},\hat{q},\hat{i},\hat{r})}=(-1)^{d-\hat{p}+\hat{r}}\tilde{C}_{\hat{p}-\hat{r}}^{\hat{p},\hat{q}}\lambda_{\hat{i}}^{\hat{p},\hat{q}}(\textup{Im}(z))^{\hat{p}+\hat{q}-d-2\hat{r}}+O(|t|^{\hat{p}+\hat{q}-d-2\hat{r}-1}).$$ This proves the theorem.

\end{proof}

\begin{remark} \label{rmk'}
In the proof, we also see that there are real analytic orthogonal basis $v_i(z)$ of $\textup{exp}(zN)F^k\cap \overline{\textup{exp}(zN)F^{d-k}}$ with respect to the Hermitian form $(\sqrt{-1})^dS(\bullet,\overline{\bullet})$ such that the growth of order of $(\sqrt{-1})^dS(v_i(z),\overline{v_i(z)})$ are $\textup{Im}(z)^{l_i}$ with $k-d\leq l_i\leq k$. The asymptotic order $l_i$ is the smallest integer among those integers $l$ such that $$\lim_{\text{Im}(z)\rightarrow\infty }e^{-zN}v_i(z)\in W_{l+d}.$$ This is compatible with Deligne-Schmid's theorem \cite[Theorem 6.6$'$]{schmid} in the case of the degeneration of polarized Hodge structures, which states that a (multi-valued) flat section $v(q)$ belongs to $W_l$ if and only if, along a radical ray,  $$S(Cv(q),\overline{v(q)})=O((-\textup{log}\ |q|)^{l-d})$$ as $|q|$ tends to zero.

\end{remark}

\section{Appendix: Computations of the Limiting Mixed Hodge Structures of Sano's Non-K\"ahler Calabi--Yau $m$-folds}

In this section, we carry out the detailed computations of the limiting mixed Hodge structures associated with Sano's non-K\"ahler Calabi--Yau $m$-folds. The main tools for these computations are Propositions \ref{tool1}, \ref{tool2}, and \ref{tool3}. More precisely, we prove Propositions \ref{tedious}, \ref{tedious1}, and \ref{tedious2}, which describe the relevant restriction and Gysin maps, the graded pieces of the Steenbrink weight filtration, and the signatures of the induced Hermitian pairings.

\subsection{Proof of Proposition \ref{tedious}}

\begin{proof}[Proof of Proposition \ref{tedious}]

The computation is standard, using Propositions \ref{tool1}, \ref{tool2}, and \ref{tool3}. 

\begin{enumerate}[label=(\alph*)]
    \item \textbf{Proposition \ref{tedious}(1) and (2).}
    
    Let
    $
    \iota:S\times_{\mb{P}^1}T\hookrightarrow \mb{P}^2\times T
    $
    denote the natural inclusion. By Proposition \ref{tool1} and Theorem \ref{fib}, the Gysin morphism
    $$
    H^{k-1}(S\times_{\mb{P}^1}T)
    \xrightarrow{\iota_!}
    H^{k+1}(\mb{P}^2\times T)
    $$
    agrees, on the part coming from
    $
    R=H^\bullet(S)\otimes_{H^\bullet(\mb{P}^1)}H^\bullet(T),
    $
    with the map from degree $k-1$ to degree $k+1$ induced by the Gysin morphism associated with the inclusion
    $
    S\hookrightarrow \mb{P}^2\times\mb{P}^1;
    $
    namely, the map
    $$
    H^\bullet(S)\otimes_{H^\bullet(\mb{P}^1)}H^\bullet(T)
    \rightarrow
    H^\bullet(\mb{P}^2\times\mb{P}^1)
    \otimes_{H^\bullet(\mb{P}^1)}
    H^\bullet(T)
    =
    H^\bullet(\mb{P}^2\times T).
    $$
    Since
    $$
    \operatorname{ker}\left(
    H^2(S)
    \rightarrow
    H^4(\mb{P}^2\times\mb{P}^1)
    \right)
    =
    \langle e_j-e_i\mid i\neq j\rangle
    =
    \langle h,f'\rangle^\perp,
    $$
    it follows that
    $$
    \begin{aligned}
    &\operatorname{ker}\left(
    H^{k-1}(S\times_{\mb{P}^1}T)
    \xrightarrow{(\iota_!,-\iota_!\circ(\phi_a^{-1})^*)}
    H^{k+1}(\mb{P}^2\times T)
    \oplus
    H^{k+1}(\mb{P}^2\times T)
    \right)\\
    &\qquad=
    \left\{
    \begin{array}{ll}
    \langle h,\phi_a^*h,f\rangle^\perp
    \cup\eta^{\frac{k-3}{2}},
    & \textup{ if } k \textup{ is odd, }\\[2mm]
    0,
    & \textup{ if } k \textup{ is even. }
    \end{array}
    \right.
    \end{aligned}
    $$

    Moreover, for $i=1,\ldots,a$, the cohomology class of the divisor $F_i$ is $f$, while the cohomology class of $\Gamma_a$ is
    $$
    [\Gamma_a]
    =
    3h+3\phi_a^*h-(a-2)f.
    $$
    Therefore, when $k$ is odd,
    $$
    \langle h,\phi_a^*h,f\rangle^\perp
    \cup\eta^{\frac{k-3}{2}}
    \subseteq
    \operatorname{ker}\left(
    H^{k-1}(S\times_{\mb{P}^1}T)
    \longrightarrow
    H^{k-1}(\Gamma_a)
    \oplus
    \bigoplus_{i=1}^aH^{k-1}(F_i)
    \right),
    $$
    where the map is the direct sum of the restriction maps to $\Gamma_a$ and the $F_i$.

    Applying Proposition \ref{tool2}(2) successively to the blowups along $F_1,\ldots,F_a$ and $\Gamma_a$, we conclude that
    $$
    \operatorname{ker}\left(
    H^{k-1}(S\times_{\mb{P}^1}T)
    \xrightarrow{-\gamma}
    H^{k+1}(X_1)\oplus H^{k+1}(X_2)
    \right)
    $$$$=
    \left\{
    \begin{array}{ll}
    \langle h,\phi_a^*h,f\rangle^\perp
    \cup\eta^{\frac{k-3}{2}},
    & \textup{ if } k \textup{ is odd, }\\[2mm]
    0,
    & \textup{ if } k \textup{ is even. }
    \end{array}
    \right.
    $$
    This proves Proposition \ref{tedious}(1), and the same argument applied to the image of
    $$
    R^{m-1}\longrightarrow H^{m-1}(S\times_{\mb{P}^1}T)
    $$
    proves Proposition \ref{tedious}(2).

\item \textbf{Proposition \ref{tedious}(3): restriction maps from $\mb{P}^2\times T$.}

Similarly, the restriction morphism
$$
H^{k-1}(\mb{P}^2\times T)
\xrightarrow{\iota^*}
H^{k-1}(S\times_{\mb{P}^1}T)
$$
agrees, on the part described by Theorem \ref{fib}, with the map from degree $k-1$ to degree $k-1$ induced by the restriction morphism associated with the inclusion
$$
S\hookrightarrow \mb{P}^2\times\mb{P}^1;
$$
namely,
$$
H^\bullet(\mb{P}^2\times T)
=
H^\bullet(\mb{P}^2\times\mb{P}^1)
\otimes_{H^\bullet(\mb{P}^1)}
H^\bullet(T)
\longrightarrow
H^\bullet(S)
\otimes_{H^\bullet(\mb{P}^1)}
H^\bullet(T).
$$
Therefore, a direct computation gives that
$$
\operatorname{im}\left(
H^{k-1}(\mb{P}^2\times T)
\oplus
H^{k-1}(\mb{P}^2\times T)
\xrightarrow{(-\iota^*,\,\phi_a^*\circ\iota^*)}
H^{k-1}(S\times_{\mb{P}^1}T)
\right)
$$
is generated by
$$
\left\{
\begin{array}{ll}
\eta^{\frac{k-1}{2}},\
\langle h,\phi_a^*h,f\rangle\cup\eta^{\frac{k-3}{2}},\
h^2\cup\eta^{\frac{k-5}{2}},
&
\textup{if } k \textup{ is odd and } k\neq m-1,
\\[2mm]
H^{m-2}_{\operatorname{prim}}(T),\
\eta^{\frac{k-1}{2}},\
\langle h,\phi_a^*h,f\rangle\cup\eta^{\frac{k-3}{2}},\
h^2\cup\eta^{\frac{k-5}{2}},
&
\textup{if } k \textup{ is odd and } k=m-1,
\\[2mm]
0,
&
\textup{if } k \textup{ is even and } k\neq m-1,
\\[2mm]
H^{m-2}(T),
&
\textup{if } k \textup{ is even and } k=m-1.
\end{array}
\right.
$$
\item \textbf{Proposition \ref{tedious}(3): Gysin map from $\Gamma_a$.}

Since
$
\Gamma_a=C_a\times_{\mb{P}^1}T
$
is a smooth divisor of $C_a\times\mb{P}^{m-2}$ associated with the ample line bundle
$$
\left(\pi_{C_a}^*(g')^*\mathcal{O}_{\mb{P}^1}(1)\right)
\otimes
\left(\pi_{\mb{P}^{m-2}}^*\mathcal{O}_{\mb{P}^{m-2}}(m-1)\right),
$$
the Lefschetz hyperplane theorem implies that the following commutative diagram
\begin{center}
\begin{tikzcd}
H^\bullet(C_a)\otimes_{H^\bullet(\mb{P}^1)}
H^\bullet(\mb{P}^1\times\mb{P}^{m-2})
\arrow[d]
\arrow[r, "\cong"]
&
H^\bullet(C_a\times\mb{P}^{m-2})
\arrow[d]
\\
H^\bullet(C_a)\otimes_{H^\bullet(\mb{P}^1)}
H^\bullet(T)
\arrow[r]
&
H^\bullet(C_a\times_{\mb{P}^1}T)
\end{tikzcd}
\end{center}
induces isomorphisms in degrees at most $m-3$, and the two vertical maps are injective in degree $m-2$.

In particular, for $k\leq m$, the Gysin morphism
$$
H^{k-3}(C_a\times_{\mb{P}^1}T)
\longrightarrow
H^{k-1}(S\times_{\mb{P}^1}T)
$$
agrees, on the part described above, with the map from degree $k-3$ to degree $k-1$ induced by the Gysin morphism associated with the inclusion
$
C_a\hookrightarrow S;
$
namely,
$$
H^\bullet(C_a)\otimes_{H^\bullet(\mb{P}^1)}H^\bullet(T)
\rightarrow
H^\bullet(S)\otimes_{H^\bullet(\mb{P}^1)}H^\bullet(T).
$$
Since
$$
[C_a]
=
3h+3\phi_a^*h-(a-2)f,
$$
a direct computation gives
$$
\operatorname{im}\left(
H^{k-3}(C_a\times_{\mb{P}^1}T)
\longrightarrow
H^{k-1}(S\times_{\mb{P}^1}T)
\right)
$$$$=
\left\{
\begin{array}{ll}
\operatorname{Span}\left\{
\left(3h+3\phi_a^*h-(a-2)f\right)
\cup\eta^{\frac{k-3}{2}},
\
h^2\cup\eta^{\frac{k-5}{2}}
\right\},
&
\textup{if } k \textup{ is odd},
\\[2mm]
0,
&
\textup{if } k \textup{ is even}.
\end{array}
\right.
$$
In particular, this image is contained in the image described in part \textup{(b)}.
 
\item \textbf{Proposition \ref{tedious}(3): Gysin map from $F_i$.}

Write
$
F_i=S_t\times T_t
$
for some general $t\in\mb{P}^1$, and let
$$
\iota_t:S_t\times T_t\hookrightarrow S\times_{\mb{P}^1}T
$$
be the natural inclusion, where $S_t$ and $T_t$ are the fibers of the Lefschetz fibrations
$
S\xrightarrow{g_1}\mb{P}^1$ and
$T\xrightarrow{g_2}\mb{P}^1
$
over $t$, respectively. By the theorem of the fixed part, cf. \cite[Theorem 4.23]{PS}, for every $l\geq 0$, the image of the restriction map
$$
H^l(S\times_{\mb{P}^1}T)
\xrightarrow{\iota_t^*}
H^l(S_t\times T_t)
$$
coincides with the subspace of cohomology classes invariant under the monodromy action.

Since the sets of critical values of the two Lefschetz fibrations $g_1$ and $g_2$ are disjoint, by the generality of $S$ and $T$, the monodromy action of
$
S\times_{\mb{P}^1}T\xrightarrow{g}\mb{P}^1
$
on
$$
H^l(S_t\times T_t)
=
\bigoplus_{i+j=l}H^i(S_t)\otimes H^j(T_t)
$$
is induced by the product of the monodromy actions of $g_1$ and $g_2$. Hence,
$$
\operatorname{im}\left(
H^l(S\times_{\mb{P}^1}T)
\xrightarrow{\iota_t^*}
H^l(S_t\times T_t)
\right)$$$$
=
\bigoplus_{i+j=l}
\operatorname{im}\left(
H^i(S)\rightarrow H^i(S_t)
\right)
\otimes
\operatorname{im}\left(
H^j(T)\rightarrow H^j(T_t)
\right)$$$$=
\bigoplus_{i+j=l}
H^i_{\operatorname{fixed}}(S_t)
\otimes
H^j_{\operatorname{fixed}}(T_t),$$
where $H^i_{\operatorname{fixed}}(S_t)$ and $H^j_{\operatorname{fixed}}(T_t)$ denote the fixed parts of the cohomology of the Lefschetz fibrations $g_1$ and $g_2$, respectively; see \cite[(C-11) and Theorem C.23]{PS}. The dimension of this vector space is
$$
\left\{
\begin{array}{ll}
0, & \textup{if } l \textup{ is odd},\\[1mm]
1, & \textup{if } l=0 \textup{ or } l=2m-4,\\[1mm]
2, & \textup{if } l \textup{ is even and } 0<l<2m-4.
\end{array}
\right.
$$

By Poincar\'e duality, the image of the Gysin morphism
$$
H^l(S_t\times T_t)
\xrightarrow{(\iota_t)_!}
H^{l+2}(S\times_{\mb{P}^1}T)
$$
has the same dimension as the image of the restriction morphism $$H^{2m-4-l}(S\times_{\mb{P}^1}T)\xrightarrow{(\iota_t)^*}H^{2m-4-l}(S_t\times T_t).$$ Since
$$
\operatorname{im}\left(
H^l(S\times_{\mb{P}^1}T)
\xrightarrow{\iota_t^*}
H^l(S_t\times T_t)
\xrightarrow{(\iota_t)_!}
H^{l+2}(S\times_{\mb{P}^1}T)
\right)
$$
is contained in
$$
\operatorname{im}\left(
H^l(S_t\times T_t)
\xrightarrow{(\iota_t)_!}
H^{l+2}(S\times_{\mb{P}^1}T)
\right)
$$
and has the same dimension, the two images coincide. Moreover, by Proposition \ref{tool1}, the composition
$$
H^l(S\times_{\mb{P}^1}T)
\xrightarrow{\iota_t^*}
H^l(S_t\times T_t)
\xrightarrow{(\iota_t)_!}
H^{l+2}(S\times_{\mb{P}^1}T)
$$
is cup product with the class $[F_i]=f$. Therefore,
$$
\begin{aligned}
\operatorname{im}\left(
H^l(S_t\times T_t)
\xrightarrow{(\iota_t)_!}
H^{l+2}(S\times_{\mb{P}^1}T)
\right)
&=
\operatorname{im}\left(
H^l(S\times_{\mb{P}^1}T)
\xrightarrow{\cup f}
H^{l+2}(S\times_{\mb{P}^1}T)
\right).
\end{aligned}
$$

It follows that
$$
\operatorname{im}\left(
H^{k-3}(F_i)
\rightarrow
H^{k-1}(S\times_{\mb{P}^1}T)
\right)
$$$$=
\left\{
\begin{array}{ll}
\operatorname{Span}\left\{
f\cup\eta^{\frac{k-3}{2}},
\,
h^2\cup\eta^{\frac{k-5}{2}}
\right\},
& \textup{if } k \textup{ is odd},\\[2mm]
0,
& \textup{if } k \textup{ is even}.
\end{array}
\right.
$$
In particular, this image is contained in the image described in part \textup{(b)}.
\item \textbf{Proposition \ref{tedious}(3): conclusion.}

Combining the computations from (b), (c) and (d), we conclude that
$$
\operatorname{im}\left(
H^{k-1}(X_1)\oplus H^{k-1}(X_2)
\xrightarrow{\theta}
H^{k-1}(S\times_{\mb{P}^1}T)
\right)
$$
coincides with
$$
\operatorname{im}\left(
H^{k-1}(\mb{P}^2\times T)
\oplus
H^{k-1}(\mb{P}^2\times T)
\xrightarrow{(-\iota^*,\,\phi_a^*\circ\iota^*)}
H^{k-1}(S\times_{\mb{P}^1}T)
\right).
$$ This finishes the proof.

\end{enumerate}
\end{proof}

\subsection{Proof of Proposition \ref{tedious1}}

\begin{proof}[Proof of Proposition \ref{tedious1}]
Set
$
D:=S\times_{\mb{P}^1}T
$
and
$$
W:=
\operatorname{im}\left(
R^{m-1}\rightarrow H^{m-1}(D)
\right).
$$
\begin{enumerate}[label=(\alph*)]
\item \textbf{Proof of (1).}

 Since $X_0=X_1\cup_DX_2$ has only two irreducible components, the weight spectral sequence gives
$$
\operatorname{Gr}^{W^{St}}_{m+1}H^m(X_\infty)
=
\operatorname{ker}\left(
H^{m-1}(D)
\xrightarrow{-\gamma}
H^{m+1}(X_1)\oplus H^{m+1}(X_2)
\right).
$$
Let
$$
G_2:=
\operatorname{im}\left(
H^{m-1}(X_1)\oplus H^{m-1}(X_2)
\xrightarrow{\theta}
H^{m-1}(D)
\right).
$$
By Proposition \ref{tedious}(2) and (3), together with the computation of the cup-product pairing on $W$, we have an orthogonal decomposition
$
W=G_1\oplus G_2,
$
where
$$
G_1=
\langle h,\phi_a^*h,f\rangle^\perp
\cup\eta^{\frac{m-3}{2}}.
$$
Moreover, both $G_1$ and $G_2$ are non-degenerate. Hence
$$
H^{m-1}(D)=W\oplus W^\perp
$$
orthogonally.

By Proposition \ref{tedious}(2),
$$
\operatorname{ker}\left(
W\xrightarrow{-\gamma}
H^{m+1}(X_1)\oplus H^{m+1}(X_2)
\right)
=
G_1.
$$
On the other hand, let $\alpha\in W^\perp$. For every
$
\beta\in H^{m-1}(X_1)\oplus H^{m-1}(X_2),
$
the adjunction between the restriction and Gysin morphisms gives
$$
Q_{X_1\amalg X_2}(\gamma\alpha,\beta)
=
Q_D(\alpha,\theta\beta).
$$
Since
$
\operatorname{im}(\theta)=G_2\subset W
$
and $\alpha\perp W$, the right-hand side vanishes. By Poincar\'e duality,
$
\gamma\alpha=0.
$
Thus
$
W^\perp\subseteq\operatorname{ker}(\gamma).
$
Conversely, if
$$
\alpha=\alpha_W+\alpha_\perp\in W\oplus W^\perp
$$
and $\gamma\alpha=0$, then $\gamma\alpha_\perp=0$, and hence
$
\gamma\alpha_W=0.
$
Therefore $\alpha_W\in G_1$ by Proposition \ref{tedious}(2). We conclude that
$$
\operatorname{Gr}^{W^{St}}_{m+1}H^m(X_\infty)
=
G_1\oplus W^\perp.
$$
The decomposition is orthogonal by construction.

\item \textbf{Proof of (2).}

We next prove (2). Regard $D$ as the smooth hypersurface of
$
S\times\mb{P}^{m-2}
$
whose cohomology class is
$
[D]=f+(m-1)\eta.
$
The non-primitive part of $H^{m-1}(D)$ is contained in $W$. Therefore
$$
W^\perp\subset H^{m-1}_{\operatorname{prim}}(D).
$$
It follows from the Hodge--Riemann bilinear relations and Proposition \ref{Fu2}, that
$
S(C\bullet,\overline{(-N)\bullet})
$
is positive definite on $W^\perp$.
We now consider $G_1$. Let
$
L=h+\eta
$
be the Lefschetz operator induced by the product polarization on
$S\times\mb{P}^{m-2}$. Since
$$
\langle h,\phi_a^*h,f\rangle^\perp
\subseteq h^\perp
=
H^2_{\operatorname{prim}}(S),
$$
for every
$x\in\langle h,\phi_a^*h,f\rangle^\perp$
we have
$
L^{\frac{m-3}{2}}x
=
x\cup\eta^{\frac{m-3}{2}}.
$
Hence
$$
G_1
=
L^{\frac{m-3}{2}}
\left(
\langle h,\phi_a^*h,f\rangle^\perp
\right).
$$
By the Hodge--Riemann bilinear relations, the Hermitian pairing
$$
S(C\bullet,\overline{(-N)\bullet})
$$ is positive definite on $G_1$ if
$
m\equiv3\pmod4,
$
and negative definite on $G_1$ if
$
m\equiv1\pmod4.
$

\item \textbf{Proof of (3): Computation of Kernel.}

Set
$
D:=S\times_{\mb{P}^1}T.
$
The weight spectral sequence gives
$$
\operatorname{Gr}^{W^{St}}_mH^m(X_\infty)
=
\frac{
\operatorname{ker}\left(
H^m(X_1)\oplus H^m(X_2)
\xrightarrow{\theta}
H^m(D)
\right)}
{
\operatorname{im}\left(
H^{m-2}(D)
\xrightarrow{-\gamma}
H^m(X_1)\oplus H^m(X_2)
\right)
}.
$$
By Proposition \ref{tool2},
$$
H^m(X_1)
\cong
H^m(\mb{P}^2\times T)
\oplus H^{m-2}(\Gamma_a)
\oplus
\bigoplus_{i=1}^aH^{m-2}(F_i),
$$
while
$$
H^m(X_2)=H^m(\mb{P}^2\times T).
$$
Since $m$ is odd,
$$
H^m(\mb{P}^2\times T)
=
H^2(\mb{P}^2)\otimes H^{m-2}(T).
$$

We first consider the blowup summand $H^{m-2}(\Gamma_a)$. Let
$
j:\Gamma_a\hookrightarrow D
$
be the inclusion. By Proposition \ref{tool2}(2), its contribution to the restriction map is the Gysin morphism
$$
j_!:H^{m-2}(\Gamma_a)\longrightarrow H^m(D).
$$
By adjunction,
$
\operatorname{ker}(j_!)
=
\operatorname{im}(j^*)^\perp.
$
Theorem \ref{fib} gives
$$
H^{m-2}(D)\cong H^{m-2}(T),
$$
and
$$
\operatorname{im}(j^*)
=
H^0(C_a)\otimes H^{m-2}(T).
$$
Using the orthogonal Lefschetz decomposition
$$
H^{m-2}(\Gamma_a)
=
H^{m-2}_{\operatorname{prim}}(\Gamma_a)
\oplus
L^{\frac{m-3}{2}}H^1(\Gamma_a)
$$
and the definition of $U$, we therefore obtain
$$
\operatorname{ker}(j_!)
=
U\oplus
L^{\frac{m-3}{2}}H^1(\Gamma_a).
$$

Similarly, for
$
j_i:F_i\hookrightarrow D,
$
the restriction of the blowup summand $H^{m-2}(F_i)$ is $(j_i)_!$. Since $m-2$ is odd, the fixed part of $H^{m-2}(F_i)$ vanishes, and hence
$$
j_i^*:H^{m-2}(D)\longrightarrow H^{m-2}(F_i)
$$
is zero. By adjunction, $(j_i)_!=0$. Thus the whole space
$$
H^{m-2}(F_i)
=
H^{m-2}_{\operatorname{prim}}(F_i)
\oplus
L^{\frac{m-3}{2}}H^1(F_i)
$$
contributes to the kernel.

It remains to consider the part involving
$$
H^m(\mb{P}^2\times T)
\oplus
\left(H^0(C_a)\otimes H^{m-2}(T)\right)
\oplus
H^m(X_2).
$$
For $\alpha\in H^{m-2}(T)$, we have
$$
[\Gamma_a]
=
3h+3\phi_a^*h-(a-2)f.
$$
Moreover,
$$
\phi_a^*h\cup\alpha=h\cup\alpha,
\qquad
f\cup\alpha=0
$$
in $H^m(D)$. Hence
$$
j_!\alpha
=
[\Gamma_a]\cup\alpha
=
6h\cup\alpha.
$$
By Proposition \ref{tool2}(2),
$$
(-6h\cup\alpha,\alpha)
$$
therefore lies in the kernel of the restriction map. This gives
$$
U'
=
\left\{
(-6h\cup\alpha,\alpha)
\ \middle|\
\alpha\in H^{m-2}(T)
\right\}.
$$ This proves that $U_1,...,U_5$ are contained in $\operatorname{ker}(\theta)$

\item \textbf{Proof of (3): Linear Independence with Image.}

 We also need to verify that the above classes remain linearly independent after passing to the quotient by $\operatorname{im}(\gamma)$. By Proposition \ref{tool3}, 
$$
\operatorname{im}(\gamma)\perp\operatorname{ker}(\theta).
$$
And since the cup-product pairing $S(\bullet,\bullet)$ is non-degenerate on $$
\operatorname{Gr}^{W^{St}}_mH^m(X_\infty)
=
\frac{\operatorname{ker}(\theta)}{\operatorname{im}(\gamma)},
$$, $\operatorname{im}(\gamma)$ is the null subspace of $\operatorname{ker}(\theta)$. Therefore, it suffices to prove that the restriction of the cup-product pairing to
$$
U'\oplus U
\oplus L^{\frac{m-3}{2}}H^1(\Gamma_a)
\oplus\bigoplus_{i=1}^aH^{m-2}_{\operatorname{prim}}(F_i)
\oplus\bigoplus_{i=1}^aL^{\frac{m-3}{2}}H^1(F_i)
$$
is non-degenerate: this will be verified in (f) when we compute the signature of the cup product on these subspaces.

\item \textbf{Proof of (3): Dimension Counting}

It remains only to check that these classes exhaust
$\operatorname{Gr}^{W^{St}}_mH^m(X_\infty)$. Set
$$
t:=\dim H^{m-2}(T),\qquad
b_\Gamma:=\dim H^{m-2}(\Gamma_a),\qquad
b_F:=\dim H^{m-2}(F_i).
$$
By Proposition \ref{tedious}(1), with $k=m-1$,
$$
H^{m-2}(D)
\xrightarrow{-\gamma}
H^m(X_1)\oplus H^m(X_2)
$$
is injective. Thus, by Proposition \ref{tool3},
$
\operatorname{rank}(\gamma)
=
\operatorname{rank}(\theta)
=
t,
$
and hence
$$
\dim\operatorname{Gr}^{W^{St}}_mH^m(X_\infty)
=
b_\Gamma+a b_F.
$$
On the other hand,
$$
\dim U'=t,
$$
$$
\dim U+
\dim L^{\frac{m-3}{2}}H^1(\Gamma_a)
=
b_\Gamma-t,
$$
and
$$
\dim H^{m-2}_{\operatorname{prim}}(F_i)
+
\dim L^{\frac{m-3}{2}}H^1(F_i)
=
b_F.
$$
Therefore the displayed summands have total dimension
$$
t+(b_\Gamma-t)+a b_F
=
b_\Gamma+a b_F,
$$
so there are no additional classes.

Finally, the decomposition is orthogonal by Proposition \ref{tool2}(3), the orthogonal Lefschetz decompositions on $\Gamma_a$ and the $F_i$, and the definition of $U$. Thus
$$
\operatorname{Gr}^{W^{St}}_mH^m(X_\infty)
=
U'\oplus U
\oplus L^{\frac{m-3}{2}}H^1(\Gamma_a)
\oplus\bigoplus_{i=1}^aH^{m-2}_{\operatorname{prim}}(F_i)
\oplus\bigoplus_{i=1}^aL^{\frac{m-3}{2}}H^1(F_i)
$$
orthogonally.

\item \textbf{Proof of (4).}

The spaces $U$ and $H^{m-2}_{\operatorname{prim}}(F_i)$ are contained in primitive middle cohomology. Hence the Hodge--Riemann bilinear relations imply that
$
S(C\bullet,\overline{\bullet})
$
is positive definite on
$$
U\oplus
\bigoplus_{i=1}^aH^{m-2}_{\operatorname{prim}}(F_i).
$$

For
$$
u_\alpha=(-6h\cup\alpha,\alpha)\in U',
$$
Proposition \ref{tool2}(3) and the fact that
$$
\deg(\Gamma_a\rightarrow T)=18
$$
give
$$
\begin{aligned}
S(Cu_\alpha,\overline{u_\alpha})
&=
36\epsilon(m)Q_T(C\alpha,\overline{\alpha})
+
18\epsilon(m-2)Q_T(C\alpha,\overline{\alpha})\\
&=
-18\epsilon(m-2)Q_T(C\alpha,\overline{\alpha}).
\end{aligned}
$$
Since $m-2$ is odd, $H^{m-2}(T)$ is primitive, and hence
$$
\epsilon(m-2)Q_T(C\alpha,\overline{\alpha})>0
$$
for $\alpha\neq0$. Thus the pairing is negative definite on $U'$.

Finally, put
$
r=\frac{m-3}{2}.
$
For $Z=\Gamma_a$ or $F_i$, the Lefschetz decomposition and the Hodge--Riemann bilinear relations show that the Hermitian form on
$
L^rH^1(Z)
$
has sign $(-1)^r$. Consequently,
$
S(C\bullet,\overline{\bullet})
$
is positive definite on
$$
L^{\frac{m-3}{2}}H^1(\Gamma_a)
\oplus
\bigoplus_{i=1}^aL^{\frac{m-3}{2}}H^1(F_i)
$$
if
$
m\equiv3\pmod4,
$
and negative definite if
$
m\equiv1\pmod4.
$
This proves (4).
\end{enumerate}

\end{proof}

\subsection{Proof of Proposition \ref{tedious2}}

\begin{proof}[Proof of Proposition \ref{tedious2}]
Set
$
D:=S\times_{\mb{P}^1}T
$
and put
$
n:=m-2,\  r:=\frac{m-2}{2}.
$
\begin{enumerate}[label=(\alph*)]
    \item \textbf{Proof of (1).}

 Since $X_0=X_1\cup_DX_2$ has only two irreducible components,
$$
\operatorname{Gr}^{W^{St}}_{m+1}H^m(X_\infty)
=
\operatorname{ker}\left(
H^{m-1}(D)
\xrightarrow{-\gamma}
H^{m+1}(X_1)\oplus H^{m+1}(X_2)
\right).
$$
By Proposition \ref{tedious}(4), since $m$ is even,
$$
\operatorname{im}\left(
H^{m-1}(X_1)\oplus H^{m-1}(X_2)
\xrightarrow{\theta}
H^{m-1}(D)
\right)=0.
$$
Since $\theta$ and $\gamma$ are adjoint with respect to the Poincar\'e pairings, it follows that
$
\gamma=0.
$
Hence
$$
\operatorname{Gr}^{W^{St}}_{m+1}H^m(X_\infty)
=
H^{m-1}(D).
$$
Moreover, by Theorem \ref{fib},
$
H^{m-3}(D)=0,
$
because $m-3$ is odd. Therefore the Lefschetz decomposition gives
$$
H^{m-1}(D)=H^{m-1}_{\operatorname{prim}}(D).
$$
The Hodge--Riemann bilinear relations and Proposition \ref{Fu2} then imply that
$
S(C\bullet,\overline{(-N)\bullet})
$
is positive definite on
$\operatorname{Gr}^{W^{St}}_{m+1}H^m(X_\infty)$.

\item \textbf{Proof of (2): Computation of Kernel.}

We have
$$
\operatorname{Gr}^{W^{St}}_mH^m(X_\infty)
=
\frac{
\operatorname{ker}\left(
H^m(X_1)\oplus H^m(X_2)
\xrightarrow{\theta}
H^m(D)
\right)}
{
\operatorname{im}\left(
H^{m-2}(D)
\xrightarrow{-\gamma}
H^m(X_1)\oplus H^m(X_2)
\right)
}.
$$
By Proposition \ref{tool2},
$$
H^m(X_1)
\cong
H^m(\mb{P}^2\times T)
\oplus H^n(\Gamma_a)
\oplus
\bigoplus_{i=1}^aH^n(F_i),
$$
and
$$
H^m(X_2)=H^m(\mb{P}^2\times T).
$$

We first consider the $\Gamma_a$-summand. Let
$
j:\Gamma_a\hookrightarrow D.
$
Proposition \ref{tool2}(2) identifies the restriction of
$H^n(\Gamma_a)$ with
$$
j_!:H^n(\Gamma_a)\longrightarrow H^{n+2}(D).
$$
By adjunction,
$
\operatorname{ker}(j_!)
=
\operatorname{im}(j^*)^\perp.
$
The image of
$$
j^*:H^n(D)\longrightarrow H^n(\Gamma_a)
$$
is
$
p_T^*H^n(T),
$
which we identify with
$
H^0(C_a)\otimes H^n(T).
$
Thus
$$
\operatorname{ker}(j_!)=U_1.
$$

Notice also that the non-primitive part of $H^n(\Gamma_a)$ is contained in
$p_T^*H^n(T)$. Indeed, by the Lefschetz hyperplane theorem the non-primitive middle cohomology of $\Gamma_a$ comes from
$C_a\times\mb{P}^{m-2}$ and is generated by
$$
\eta^r
\qquad\textup{and}\qquad
(g')^*[\operatorname{pt}]\cup\eta^{r-1}.
$$
Both classes belong to $p_T^*H^n(T)$. Consequently,
$$
U_1\subset H^n_{\operatorname{prim}}(\Gamma_a).
$$

We next consider the $F_i$-summands. Let
$
j_i:F_i\hookrightarrow D.
$
By the theorem of the fixed part, cf. \cite[Theorem 4.23]{PS}, the image of
$$
j_i^*:H^n(D)\rightarrow H^n(F_i)
$$
is the two-dimensional fixed subspace
$
V_i
=
\operatorname{Span}_{\C}
\left\{
A_i,B_i
\right\},
$
where
$$
A_i:=\phi_i^*\eta^r,
\qquad
B_i:=\phi_i^*\left(h\cup\eta^{r-1}\right).
$$
This is exactly the non-primitive Lefschetz part of $H^n(F_i)$. Hence
$$
V_i^\perp=H^n_{\operatorname{prim}}(F_i).
$$
By adjunction,
$
\operatorname{ker}(j_i)_!
=
H^n_{\operatorname{prim}}(F_i).
$
This gives
$$
U_2=
\bigoplus_{i=1}^aH^n_{\operatorname{prim}}(F_i).
$$

The fixed parts of the different $F_i$ give an additional kernel. Since the $F_i$ are fibers of the same fibration,
$
(j_i)_!A_i
$
is independent of $i$, and the same is true for
$
(j_i)_!B_i.
$
Moreover, these two Gysin images are linearly independent. Hence the kernel of
$$
\bigoplus_{i=1}^aV_i
\rightarrow
H^{n+2}(D)
$$
is precisely
$$
U_4
=
\operatorname{Span}_{\C}
\left\{
A_j-A_k,\,
B_j-B_k
\ \middle|\
j\neq k
\right\}.
$$
In particular,
$
\dim U_4=2(a-1).
$

We now consider the vanishing middle cohomology of $T$. Let
$
\alpha\in H^n_{\operatorname{van}}(T).
$
A direct computation in
$$
R=
H^\bullet(S)\otimes_{H^\bullet(\mb{P}^1)}H^\bullet(T)
$$
gives
$$
f\cup\alpha=0
\qquad\textup{and}\qquad
\phi_a^*h\cup\alpha=h\cup\alpha.
$$
Since
$$
[\Gamma_a]
=
3h+3\phi_a^*h-(a-2)f,
$$
we obtain
$$
j_!\alpha=6h\cup\alpha.
$$
It follows from Proposition \ref{tool2}(2) that
$$
(-6h\cup\alpha,\alpha)
$$
lies in the kernel of the restriction map. Therefore
$$
U_3
=
\left\{
(-6h\cup\alpha,\alpha)
\ \middle|\
\alpha\in H^n_{\operatorname{van}}(T)
\right\}.
$$

It remains to consider the fixed cohomology of $T$ together with the ambient cohomology of the two components. The fixed part of $H^n(T)$ is
$$
H^n_{\operatorname{fixed}}(T)
=
\operatorname{Span}_{\C}
\left\{
\eta^r,\,
f\cup\eta^{r-1}
\right\}.
$$
A direct computation using Proposition \ref{tool2}(2) shows that the following classes belong to the kernel of $\theta$.

If $m\geq6$, set
$$
\begin{aligned}
v_1&=
\left(\eta^{\frac m2},0,\eta^{\frac m2}\right),\\
v_2&=
\left(f\cup\eta^{\frac{m-2}{2}},0,
f\cup\eta^{\frac{m-2}{2}}\right),\\
v_3&=
\left((3h-f)\cup h\cup\eta^{\frac{m-4}{2}},0,0\right),\\
v_4&=
\left(0,0,(3h-f)\cup h\cup\eta^{\frac{m-4}{2}}\right),\\
v_5&=
\left(h^2\cup f\cup\eta^{\frac{m-6}{2}},0,0\right),\\
v_6&=
\left(
(3h-(a-2)f)\cup\eta^{\frac{m-2}{2}},
-\eta^{\frac{m-2}{2}},
-3h\cup\eta^{\frac{m-2}{2}}
\right).
\end{aligned}
$$
The first two classes lie in the kernel because their restrictions from the two components agree. The classes $v_3$, $v_4$, and $v_5$ restrict to zero, using
$
h\cup(3h-f)=0
$
on $S$. Finally,
$$
\begin{aligned}
&-\left(3h-(a-2)f\right)\cup\eta^r
-j_!\left(-\eta^r\right)
+\phi_a^*\left(-3h\cup\eta^r\right)\\
&=
-\left(3h-(a-2)f\right)\cup\eta^r
+\left(3h+3\phi_a^*h-(a-2)f\right)\cup\eta^r
-3\phi_a^*h\cup\eta^r\\
&=0,
\end{aligned}
$$
so $v_6$ also lies in the kernel. We define
$$
U_5=\operatorname{Span}_{\C}\{v_1,\ldots,v_6\}.
$$
When $m=4$, $v_2$ is proportional to $v_1$ and $v_5$ is absent. Hence
$$
U_5=\operatorname{Span}_{\C}\{v_1,v_3,v_4,v_6\}.
$$

This proves that $U_1,...,U_5$ are contained in $\operatorname{ker}(\theta)$.

\item \textbf{Proof of (2): Linear Independence with Image.}
The same reason as in (d) in the proof of Proposition \ref{tedious1} applies to show that $U_1\oplus...\oplus U_5$ is linearly independent with $\operatorname{im}(\gamma).$ 

\item \textbf{Proof of (2): Dimension Counting.}

It remains to verify that the above subspaces exhaust
$\operatorname{Gr}^{W^{St}}_mH^m(X_\infty)$. Set
$$
n:=m-2,\qquad
t:=\dim H^n(T),\qquad
b_\Gamma:=\dim H^n(\Gamma_a),\qquad
b_F:=\dim H^n(F_i).
$$

We first compute the dimension of
$\operatorname{Gr}^{W^{St}}_mH^m(X_\infty)$. By Theorem \ref{fib},
$
H^n(D)\cong R^n.
$
If $m\geq6$, then $R^n$ is generated by
$$
H^n_{\operatorname{prim}}(T),\qquad
\eta^{\frac n2},\qquad
H^2(S)\cup\eta^{\frac{n-2}{2}},\qquad
h^2\cup\eta^{\frac{n-4}{2}}.
$$
Since
$$
\dim H^n_{\operatorname{prim}}(T)=t-2
$$
and $\dim H^2(S)=10$, we obtain
$
\dim H^n(D)=t+10.
$
When $m=4$, the last term is absent, and hence
$
\dim H^2(D)=t+9.
$
Therefore,
$$
\dim H^{m-2}(D)
=
\begin{cases}
t+10, & m\geq6,\\
t+9, & m=4.
\end{cases}
$$

On the other hand, the K\"unneth formula and the Lefschetz hyperplane theorem give
$$
\dim H^m(\mb{P}^2\times T)
=
\begin{cases}
t+4, & m\geq6,\\
t+2, & m=4.
\end{cases}
$$
Indeed, if $m\geq6$, then
$$
H^m(\mb{P}^2\times T)
\cong
H^{m}(T)
\oplus
\left(h\cup H^{m-2}(T)\right)
\oplus
\left(h^2\cup H^{m-4}(T)\right),
$$
and
$$
\dim H^m(T)=\dim H^{m-4}(T)=2.
$$
For $m=4$, both outer summands are one-dimensional.

By the blowup formula,
$$
\dim H^m(X_1)+\dim H^m(X_2)
=
2\dim H^m(\mb{P}^2\times T)
+b_\Gamma+a b_F.
$$
Moreover, Proposition \ref{tedious}(1), applied with $k=m-1$, gives
$$
\operatorname{ker}\left(
H^{m-2}(D)
\xrightarrow{-\gamma}
H^m(X_1)\oplus H^m(X_2)
\right)
=
\langle h,\phi_a^*h,f\rangle^\perp
\cup\eta^{\frac{m-4}{2}}.
$$
Since
$
\dim\langle h,\phi_a^*h,f\rangle^\perp=7,
$
we have
$$
\operatorname{rank}(\gamma)
=
\dim H^{m-2}(D)-7.
$$
By the adjunction between the Gysin and restriction morphisms,
$$
\operatorname{rank}(\theta)=\operatorname{rank}(\gamma).
$$
Consequently,
$$
\begin{aligned}
\dim\operatorname{Gr}^{W^{St}}_mH^m(X_\infty)
&=
\dim H^m(X_1)+\dim H^m(X_2)
-2\operatorname{rank}(\gamma)\\
&=
\begin{cases}
b_\Gamma+a b_F+2, & m\geq6,\\
b_\Gamma+a b_F, & m=4.
\end{cases}
\end{aligned}
$$

We now compare this with the dimensions of the subspaces constructed above. Since
$$
\operatorname{im}\left(
H^{m-2}(D)\rightarrow H^{m-2}(\Gamma_a)
\right)
\cong H^{m-2}(T),
$$
we have
$$
\dim U_1=b_\Gamma-t.
$$
The fixed part of $H^{m-2}(F_i)$ is two-dimensional, so
$$
\dim U_2=a(b_F-2).
$$
Furthermore,
$$
\dim U_3
=
\dim H^{m-2}_{\operatorname{van}}(T)
=
t-2,
$$
and
$$
\dim U_4=2(a-1).
$$
Finally,
$$
\dim U_5=
\begin{cases}
6, & m\geq6,\\
4, & m=4.
\end{cases}
$$
Thus, if $m\geq6$,
$$
\begin{aligned}
\sum_{i=1}^5\dim U_i
&=
(b_\Gamma-t)
+a(b_F-2)
+(t-2)
+2(a-1)
+6\\
&=
b_\Gamma+a b_F+2,
\end{aligned}
$$
while, if $m=4$,
$$
\begin{aligned}
\sum_{i=1}^5\dim U_i
&=
(b_\Gamma-t)
+a(b_F-2)
+(t-2)
+2(a-1)
+4\\
&=
b_\Gamma+a b_F.
\end{aligned}
$$
These agree exactly with
$\dim\operatorname{Gr}^{W^{St}}_mH^m(X_\infty)$.

The above dimension computation shows that 
$$
\operatorname{Gr}^{W^{St}}_mH^m(X_\infty)
=
U_1\oplus U_2\oplus U_3\oplus U_4\oplus U_5.
$$

\item \textbf{Proof of (3).}

We now compute the pairings on the above subspaces. First, $U_1$ and $U_2$ are contained in primitive middle cohomology, so the Hodge--Riemann bilinear relations imply that
$
S(C\bullet,\overline{\bullet})
$
is positive definite on
$
U_1\oplus U_2.
$

For
$$
u_\alpha=(-6h\cup\alpha,\alpha)\in U_3,
$$
Proposition \ref{tool2}(3) gives
$$
\begin{aligned}
S(Cu_\alpha,\overline{u_\alpha})
&=
36\epsilon(m)Q_T(C\alpha,\overline{\alpha})
+
\epsilon(n)Q_{\Gamma_a}(C\alpha,\overline{\alpha}).
\end{aligned}
$$
Since
$
\deg(\Gamma_a\rightarrow T)=18,
$
we have
$$
Q_{\Gamma_a}(C\alpha,\overline{\alpha})
=
18Q_T(C\alpha,\overline{\alpha}).
$$
As
$
\epsilon(m)=-\epsilon(n),
$
it follows that
$$
S(Cu_\alpha,\overline{u_\alpha})
=
-18\epsilon(n)Q_T(C\alpha,\overline{\alpha}).
$$
The vanishing cohomology $H^n_{\operatorname{van}}(T)$ is primitive. Hence
$$
\epsilon(n)Q_T(C\alpha,\overline{\alpha})>0
$$
for $\alpha\neq0$, and therefore the pairing is negative definite on $U_3$.

We next compute the pairing on $U_4$. On $V_i$ we have
$$
Q_{F_i}(A_i,A_i)=Q_{F_i}(B_i,B_i)=0
$$
and
$$
Q_{F_i}(A_i,B_i)=3(m-1).
$$
Thus the pairing on $V_i$ has signature $(1,1)$. Since
$$
U_4
\cong
\left\{
(c_1,\ldots,c_a)\in\C^a
\ \middle|\
\sum_{i=1}^ac_i=0
\right\}
\otimes V_1,
$$
the restriction of the Hermitian pairing to $U_4$ has signature
$$
(a-1,a-1).
$$

Finally, we compute the pairing on $U_5$. Suppose first that $m\geq6$. With respect to the ordered basis
$
\{v_1,v_2,v_3,v_4,v_5,v_6\},
$
the Gram matrix, up to multiplication by the common factor $\epsilon(m)$, is
$$
\left(
\begin{array}{cccccc}
0&0&3&3&m-1&0\\
0&0&3(m-1)&3(m-1)&0&0\\
3&3(m-1)&0&0&0&-3(a-1)(m-1)\\
3&3(m-1)&0&0&0&3(m-1)\\
m-1&0&0&0&0&0\\
0&0&-3(a-1)(m-1)&3(m-1)&0&0
\end{array}
\right).
$$
Its determinant is
$$
-81a^2(m-1)^6\neq0.
$$
Moreover,
$
\operatorname{Span}_{\C}\{v_2,v_5,v_6\}
$
is a three-dimensional totally isotropic subspace. Hence the signature on $U_5$ is
$
(3,3).
$

When $m=4$, with respect to
$
\{v_1,v_3,v_4,v_6\},
$
the Gram matrix, again up to the common factor $\epsilon(m)$, is
$$
\left(
\begin{array}{cccc}
0&3&3&0\\
3&0&0&-9(a-1)\\
3&0&0&9\\
0&-9(a-1)&9&0
\end{array}
\right).
$$
Its determinant is
$$
729a^2\neq0.
$$
Since
$
\operatorname{Span}_{\C}\{v_1,v_6\}
$
is a two-dimensional totally isotropic subspace, the signature is
$
(2,2).
$
This completes the proof.

\end{enumerate}

\end{proof}

\clearpage
\markboth{References}{}

\bibliographystyle{amsalpha}
\bibliography{refs}

\end{document}